\documentclass[11pt,leqno]{amsart}
\usepackage{epsfig}
\usepackage{amssymb}
\usepackage{amscd}
\usepackage{ifthen}
\usepackage{enumitem}
\setenumerate[1]{label=(\alph*)}
\setenumerate[2]{label=(\roman*)}

\usepackage{amsthm}
\usepackage{newlfont}
\usepackage{amsmath}
\usepackage{amsfonts}
\usepackage[mathscr]{euscript}
\usepackage[ngerman,french,english]{babel}
\usepackage{mathrsfs}

\usepackage{mathtools}

\usepackage{fancybox}
\usepackage{nicefrac}
\usepackage{pifont}
\usepackage[all]{xy}
\usepackage{stmaryrd}
\usepackage{srcltx}
\usepackage{ifthen}
\usepackage{ulem}
\usepackage{accents}
\usepackage{leftidx}

\usepackage{xr-hyper}

\usepackage{hyperref}
\hypersetup{colorlinks=true, anchorcolor=red, citecolor= green,
filecolor=magenta, linkcolor=red, menucolor=red, urlcolor=cyan}

\newtheorem{theorem}{Theorem}[section]
\newtheorem{lemma}[theorem]{Lemma}
\newtheorem{definition}[theorem]{Definition}
\newtheorem{proposition}[theorem]{Proposition}

\newtheorem{corollary}[theorem]{Corollary}
\newtheorem{remark}[theorem]{Remark}

\newtheorem{example}[theorem]{Example}

\numberwithin{equation}{section}

\newcommand{\bR}{\mathbf{R}}
\newcommand{\bL}{\mathbf{L}}

\def\pre-tr{\operatorname{pre-tr}}

\newcommand{\supp}{\operatorname{Supp}}

\newcommand{\Ext}{\operatorname{Ext}}

\newcommand{\res}{\operatorname{res}}

\newcommand{\Spec}{\operatorname{Spec}}

\newcommand{\id}{\operatorname{id}}

\newcommand{\Supp}{\operatorname{Supp}}

\DeclareMathAlphabet{\mathpzc}{OT1}{pzc}{m}{it}

\newcommand{\matfak}[4]{{\xymatrix@C3ex{{{#1}} \ar@<0.3ex>[r]^-{{#2}} & {{#3}} \ar@<0.3ex>[l]^-{{#4}}}}}

\newcommand{\op}{\operatorname}

\newcommand{\ra}{\rightarrow}

\newcommand{\la}{\leftarrow}

\newcommand{\hra}{\hookrightarrow}

\newcommand{\xra}[1]{\xrightarrow{#1}}
\newcommand{\xla}[1]{\xleftarrow{#1}}

\newcommand{\sira}{\xra{\sim}}
\newcommand{\xsira}[1]{\xrightarrow[\sim]{#1}}

\newcommand{\sila}{\overset{\sim}{\leftarrow}}
\newcommand{\xsila}[1]{\xleftarrow[\sim]{#1}}

\newcommand{\mar}{\ar@{|->}}
\newcommand{\sar}{\ar@{->>}}
\newcommand{\iar}{\ar@{^{(}->}}
\newcommand{\gar}{\ar@{=}}
\newcommand{\gleichar}{\ar@{}|{=}}
\newcommand{\congar}{\ar@{}|{\cong}}

\newcommand{\verteq}{\rotatebox{90}{=}}

\newcommand{\Bl}[1]{{\mathbb{#1}}}
\newcommand{\DZ}{\Bl{Z}}
\newcommand{\DN}{\Bl{N}}

\newcommand{\DF}{\Bl{F}}
\newcommand{\DA}{{\Bl{A}}}
\newcommand{\DP}{{\Bl{P}}}

\newcommand{\mf}[1]{{\mathfrak{#1}}}

\newcommand{\mfm}{{\mf{m}}}

\newcommand{\mfp}{{\mf{p}}}
\newcommand{\mfq}{{\mf{q}}}
\newcommand{\mfr}{{\mf{r}}}

\newcommand{\Hom}{\op{Hom}}

\newcommand{\reduced}{{\op{red}}}

\newcommand{\Sh}{{\op{Sh}}}

\newcommand{\Qcoh}{{\op{Qcoh}}}
\newcommand{\Coh}{{\op{Coh}}}
\newcommand{\mfPerf}{{{\mathfrak P}{\mathfrak e}{\mathfrak r}{\mathfrak f}}}
\newcommand{\InjSh}{{\op{InjSh}}}
\newcommand{\InjQcoh}{{\op{InjQcoh}}}

\newcommand{\inj}{{\op{inj}}}

\newcommand{\Cech}{{\mathrm{\check{C}ech}}}

\newcommand{\Cechobjstar}{{\mathrm{\check{C}ob}*}}

\newcommand{\Cechobjstarbox}{{\mathrm{\check{C}ob}*\boxtimes}}

\newcommand{\sheafHom}{{\mkern3mu\mathcal{H}{om}\mkern3mu}}

\newcommand{\End}{\op{End}}

\newcommand{\pr}{\op{pr}}

\newcommand{\gldim}{\op{gldim}}

\newcommand{\projdim}{\op{projdim}}

\newcommand{\ol}[1]{{\overline{#1}}}
\newcommand{\ul}[1]{{\underline{#1}}}

\newcommand{\cek}{\vee}

\newcommand{\inv}{^{-1}}
\newcommand{\can}{\op{can}}

\newcommand{\fib}{{\operatorname{fib}}}

\newcommand{\per}{\op{per}}

\newcommand{\cpt}{{\op{cpt}}}

\newcommand{\tildew}[1]{\widetilde{#1}}

\newcommand{\comp}{\circ}

\newcommand{\opp}{{\op{op}}}

\newcommand{\define}[1]{{\textbf{#1}}}

\newboolean{showkeepdontpublishbool}
\newcommand{\keepdontpublish}[1]
{\ifthenelse{\boolean{showkeepdontpublishbool}}
  {
    \rule{0cm}{0.05cm} \\
    \rule{15cm}{0.05cm} \\
    {#1} \rule{0cm}{0.05cm} \\ 
    \rule{15cm}{0.05cm} \\
    \rule{1mm}{0mm}
  }
} 

\setboolean{showkeepdontpublishbool}{true}          
\setboolean{showkeepdontpublishbool}{false}         

\newboolean{comment}
\newcommand{\comline}[1]{\ifthenelse{\boolean{comment}}{{\bf
      \noindent\shortstack[r]{\rule{12cm}{0.02cm}\\#1\\\rule{12cm}{0.02cm}
        \pagestyle{myheadings}\markboth{Links}{Rechts} 
        }}\\}}       

\newcommand{\comtxt}[1]{\ifthenelse{\boolean{comment}}{{\bf #1}}} 
\newcommand{\comtxtn}[1]{\ifthenelse{\boolean{comment}}{{\bf #1} \\ }} 
\newcommand{\ncomtxt}[1]{\ifthenelse{\boolean{comment}}{\\ {\bf #1}}} 
\newcommand{\ncomtxtn}[1]{\ifthenelse{\boolean{comment}}{\\ {\bf #1} \\ }} 

\newcommand{\commar}[1]{\ifthenelse{\boolean{comment}}{\marginpar{{#1}}}{}}

\setboolean{comment}{true}          

\newcommand{\kommentar}[1]{}

\numberwithin{equation}{section}

\newcommand{\proj}{\op{proj}}

\newcommand{\hinj}{{\op{h-inj}}}

\newcommand{\Vb}{{\op{Vb}}}

\newcommand{\Cone}{\op{Cone}}

\renewcommand{\epsilon}{{\varepsilon}}

\newcommand{\oktaeder}[6]
{
  \xymatrix@dr{
    & {[1]{#1}} \ar[r] &  {[1]{#2}} \ar[r] & {[1]{#3}}\\
    {#3} \ar@(ur,ul)[ru] \ar[r] &
    {#5}
    \ar[r] \ar[u] & 
    {#6} 
    \ar@(dr,dl)[ru] \ar[u]\\ 
    {#2} \ar[u] \ar[r] & {#4} \ar@(dr,dl)[ru] \ar[u]\\
    {#1} \ar[u] \ar@(dr,dl)[ru]
  }
}

\newcommand{\oktaedergepunktetdown}[6]
{
  \xymatrix@dr{
    & {[1]{#1}} \ar[r] &  {[1]{#2}} \ar[r] & {[1]{#3}}\\
    {#3} \ar@(ur,ul)[ru] \ar@{..>}[r] &
    {#5}
    \ar@{..>}[r] \ar[u] & 
    {#6} 
    \ar@(dr,dl)[ru] \ar[u]\\ 
    {#2} \ar[u] \ar[r] & {#4} \ar@(dr,dl)[ru] \ar[u]\\
    {#1} \ar[u] \ar@(dr,dl)[ru]
  }
}

\newcommand{\oktaedergepunktetdownmitopt}[7]
{
  \xymatrix@dr#7{
    & {[1]{#1}} \ar[r] &  {[1]{#2}} \ar[r] & {[1]{#3}}\\
    {#3} \ar@(ur,ul)[ru] \ar@{..>}[r] &
    {#5}
    \ar@{..>}[r] \ar[u] & 
    {#6} 
    \ar@(dr,dl)[ru] \ar[u]\\ 
    {#2} \ar[u] \ar[r] & {#4} \ar@(dr,dl)[ru] \ar[u]\\
    {#1} \ar[u] \ar@(dr,dl)[ru]
  }
}

\newcommand{\oktaedergepunktetup}[6]
{
  \xymatrix@dr{
    & {[1]{#1}} \ar[r] &  {[1]{#2}} \ar[r] & {[1]{#3}}\\
    {#3} \ar@(ur,ul)[ru] \ar[r] &
    {#5}
    \ar[r] \ar@{..>}[u] & 
    {#6} 
    \ar@(dr,dl)[ru] \ar[u]\\ 
    {#2} \ar[u] \ar[r] & {#4} \ar@(dr,dl)[ru] \ar@{..>}[u]\\
    {#1} \ar[u] \ar@(dr,dl)[ru]
  }
}

\newcommand{\oktaederalles}[9]
{
  \xymatrix@dr{
    & {#7} \ar[r] &  {#8} \ar[r] & {#9}\\
    {#3} \ar@(ur,ul)[ru] \ar[r] &
    {#5}
    \ar[r] \ar[u] & 
    {#6} 
    \ar@(dr,dl)[ru] \ar[u]\\ 
    {#2} \ar[u] \ar[r] & {#4} \ar@(dr,dl)[ru] \ar[u]\\
    {#1} \ar[u] \ar@(dr,dl)[ru]
  }
}

\newcommand{\oktaedertemplatedotted}
{
  \xymatrix@dr{
    && {G} \ar[r]^-{g} 
    & {H} \ar[r]^-{h} 
    & {I} 
    \\ 
    {T4:} 
    & {C} \ar@(ur,ul)[ru]^-{c_1}
    \ar@{..>}[r]^-{c_2} 
    & {E} \ar@{..>}[r]^-{e_1} \ar[u]^-{e_2} 
    & {F}  \ar@(dr,dl)[ru]^-{f_1} \ar[u]_-{f_2}
    \\ 
    {T3:} 
    & {B} \ar[u]^-{b_1} \ar[r]^-{b_2} 
    & {D} \ar@(dr,dl)[ru]^-{d_1} \ar[u]_-{d_2} 
    \\  
    {T2:} 
    & {A} \ar[u]^-{a_1} \ar@(dr,dl)[ru]^-{a_2}
    \ar@{}[ru]|{\triangle}
    \\
    & {T1:}
  }
}

\newcommand{\oktaedertemplate}
{
  \xymatrix@dr{
    && {G} \ar[r]^-{g} 
    & {H} \ar[r]^-{h} 
    & {I} 
    \\ 
    {T4:} 
    & {C} \ar@(ur,ul)[ru]^-{c_1}
    \ar[r]^-{c_2} 
    & {E} \ar[r]^-{e_1} \ar[u]^-{e_2} 
    & {F}  \ar@(dr,dl)[ru]^-{f_1} \ar[u]_-{f_2}
    \\ 
    {T3:} 
    & {B} \ar[u]^-{b_1} \ar[r]^-{b_2} 
    & {D} \ar@(dr,dl)[ru]^-{d_1} \ar[u]_-{d_2} 
    \\  
    {T2:} &{A} \ar[u]^-{a_1} \ar@(dr,dl)[ru]^-{a_2}
    \ar@{}[ru]|{\triangle}
    \\
    & {T1:}
  }
}

\newcommand{\symeighttransone}[9]
{
  { } \ar@{}[r]^{#9} & { }\\
  {#1} \ar@(u,d)[ruu] & {#2} \ar@(u,d)[luu] & {#3} \ar[uu] & {#4} \ar[uu] & 
  {#5} \ar[uu] & {#6} \ar[uu] & {#7} \ar[uu] & {#8} \ar[uu] \\ 
}

\newcommand{\symeighttranstwo}[9]
{
  & { } \ar@{}[r]^{#9} & { }\\
  {#1} \ar[uu] & {#2} \ar@(u,d)[ruu] & {#3} \ar@(u,d)[luu] & {#4} \ar[uu] & 
  {#5} \ar[uu] & {#6} \ar[uu] & {#7} \ar[uu] & {#8} \ar[uu] \\ 
}

\newcommand{\symeighttransthree}[9]
{
  && { } \ar@{}[r]^{#9} & { }\\
  {#1} \ar[uu] &  {#2} \ar[uu] & {#3} \ar@(u,d)[ruu] & {#4} \ar@(u,d)[luu] &
  {#5} \ar[uu] & {#6} \ar[uu] & {#7} \ar[uu] & {#8} \ar[uu] \\ 
}

\newcommand{\symeighttransfour}[9]
{
  &&& { } \ar@{}[r]^{#9} & { }\\
  {#1} \ar[uu] &  {#2} \ar[uu] &  {#3} \ar[uu] & {#4} \ar@(u,d)[ruu] & 
  {#5} \ar@(u,d)[luu] & {#6} \ar[uu] & {#7} \ar[uu] & {#8} \ar[uu] \\
}

\newcommand{\symeighttransfive}[9]
{
  &&&& { } \ar@{}[r]^{#9} & { }\\
  {#1} \ar[uu] &  {#2} \ar[uu] &  {#3} \ar[uu] & {#4} \ar[uu] & 
  {#5} \ar@(u,d)[ruu] & {#6} \ar@(u,d)[luu] & {#7} \ar[uu] & {#8}
  \ar[uu] \\
}

\newcommand{\symeighttranssix}[9]
{
  &&&&& { } \ar@{}[r]^{#9} & { }\\
  {#1} \ar[uu] & {#2} \ar[uu] &  {#3} \ar[uu] & {#4} \ar[uu] & 
  {#5} \ar[uu] & 
  {#6} \ar@(u,d)[ruu] & {#7} \ar@(u,d)[luu] & {#8} \ar[uu] \\
}

\newcommand{\symeighttransseven}[9]
{
  &&&&&& { } \ar@{}[r]^{#9} & { }\\
  {#1} \ar[uu] & {#2} \ar[uu] &  {#3} \ar[uu] & {#4} \ar[uu] & 
  {#5} \ar[uu] & 
  {#6} \ar[uu] & 
  {#7} \ar@(u,d)[ruu] & {#8} \ar@(u,d)[luu] \\
}

\newcommand{\symdowneighttransone}[9]
{
  {#1} \ar@(d,u)[rdd] & {#2} \ar@(d,u)[ldd] & {#3} \ar[dd] & {#4} \ar[dd] & 
  {#5} \ar[dd] & {#6} \ar[dd] & {#7} \ar[dd] & {#8} \ar[dd] \\ 
  { } \ar@{}[r]^{#9} & { }\\
}

\newcommand{\symdowneighttranstwo}[9]
{
  {#1} \ar[dd] & {#2} \ar@(d,u)[rdd] & {#3} \ar@(d,u)[ldd] & {#4} \ar[dd] & 
  {#5} \ar[dd] & {#6} \ar[dd] & {#7} \ar[dd] & {#8} \ar[dd] \\ 
  & { } \ar@{}[r]^{#9} & { }\\
}

\newcommand{\symdowneighttransthree}[9]
{
  {#1} \ar[dd] &  {#2} \ar[dd] & {#3} \ar@(d,u)[rdd] & {#4}
  \ar@(d,u)[ldd] &
  {#5} \ar[dd] & {#6} \ar[dd] & {#7} \ar[dd] & {#8} \ar[dd] \\ 
  && { } \ar@{}[r]^{#9} & { }\\
}

\newcommand{\symdowneighttransfour}[9]
{
  {#1} \ar[dd] &  {#2} \ar[dd] &  {#3} \ar[dd] & {#4} \ar@(d,u)[rdd] & 
  {#5} \ar@(d,u)[ldd] & {#6} \ar[dd] & {#7} \ar[dd] & {#8} \ar[dd] \\
  &&& { } \ar@{}[r]^{#9} & { }\\
}

\newcommand{\symdowneighttransfive}[9]
{
  {#1} \ar[dd] &  {#2} \ar[dd] &  {#3} \ar[dd] & {#4} \ar[dd] & 
  {#5} \ar@(d,u)[rdd] & {#6} \ar@(d,u)[ldd] & {#7} \ar[dd] & {#8}
  \ar[dd] \\
  &&&& { } \ar@{}[r]^{#9} & { }\\
}

\newcommand{\symdowneighttranssix}[9]
{
  {#1} \ar[dd] & {#2} \ar[dd] &  {#3} \ar[dd] & {#4} \ar[dd] & 
  {#5} \ar[dd] & 
  {#6} \ar@(d,u)[rdd] & {#7} \ar@(d,u)[ldd] & {#8} \ar[dd] \\
  &&&&& { } \ar@{}[r]^{#9} & { }\\
}

\newcommand{\symdowneighttransseven}[9]
{
  {#1} \ar[dd] & {#2} \ar[dd] &  {#3} \ar[dd] & {#4} \ar[dd] & 
  {#5} \ar[dd] & 
  {#6} \ar[dd] & 
  {#7} \ar@(d,u)[rdd] & {#8} \ar@(d,u)[ldd] \\
  &&&&&& { } \ar@{}[r]^{#9} & { }\\
}

\newcommand{\quiverAzwei}[7]
{
  \xymatrix@R0pc@C2pc{
    {{#2}} \ar[r]^{#4} \ar@{}[r]^>(0){#1}_>(0){#3} &
    {{#6}}  \ar@{}[l]_>(0){#5}^>(0){#7}
  }
}

\newcommand{\quiverAeins}[3]
{
  \xymatrix@R0pc@C0pt{
    {{#2}} \ar@{}[r]^>(0){#1}_>(0){#3} & {}
  }
}

\newcommand{\bruhatAzwei}[9]
{
  \xymatrix@C17.3pt@R10pt{
    & 
    {{#1}} \ar[rd]^-{{#9}} \ar[ld]_-{{#7}} \ar[dddd]|(0.25){{#8}} \\
    {{#2}} \ar[dd]_-{{#8}} \ar[rrdd]|(0.25){{{#9}}} &&
    {{#3}} \ar[dd]^-{{#8}} \ar[lldd]|(0.25){{#7}} \\ 
    \\
    {{#4}} \ar[rd]_-{{#9}} &&
    {{#5}} \ar[ld]^-{{#7}} \\
    &
    {{#6}}
  }
}

\newcommand{\weylAzwei}[6]
{
  \xymatrix@C17.3pt@R10pt{
    & 
    {{#1}} \ar[rd] \ar[ld] \ar[dddd] \\
    {{#2}} \ar[dd] \ar[rrdd] &&
    {{#3}} \ar[dd] \ar[lldd] \\ 
    \\
    {{#4}} \ar[rd] &&
    {{#5}} \ar[ld] \\
    &
    {{#6}}
  }
}

\newcommand{\pretr}{{\op{pre-tr}}}

\newcommand{\thick}{\op{thick}}

\author{Valery A.~Lunts \and Olaf M.~Schn{\"u}rer}

\address{
  Department of Mathematics\\
  Indiana University\\
  Rawles Hall\\
  831 East 3rd Street\\
  Bloomington, IN 47405\\
  USA
}

\email{vlunts@indiana.edu} 

\address{
  Mathematisches Institut\\ 
  Universit{\"a}t Bonn\\
  Endenicher Allee 60\\
  53115 Bonn\\
  Germany
}

\email{olaf.schnuerer@math.uni-bonn.de}

\title[Enhancements of derived categories of 
coherent sheaves and applications]{New enhancements of derived
  categories of coherent sheaves and applications}

\thanks{}

\begin{document}

\begin{abstract}
  We introduce new
  enhancements for the bounded derived category $D^b(\Coh(X))$ of
  coherent sheaves on a suitable scheme $X$ and for its
  subcategory $\mfPerf(X)$ of perfect complexes.  They are used
  for translating Fourier-Mukai functors to functors between
  derived categories of dg algebras, for relating homological
  smoothness of $\mfPerf(X)$ to geometric smoothness of $X,$ and
  for proving homological smoothness of $D^b(\Coh(X)).$ Moreover,
  we characterize properness of $\mfPerf(X)$ and
  $D^b(\Coh(X))$ geometrically.
\end{abstract}

\maketitle
\setcounter{tocdepth}{1}

\tableofcontents

\section{Introduction}
\label{sec:introduction}

Given quasi-compact separated schemes $X$ and $Y$ over a field
$k,$ 
any $K \in D(\Qcoh(X \times Y))$ 
gives rise to the Fourier-Mukai functor
\begin{equation*}
  \Phi_K= \bR q_*(p^*(-) \otimes^{\bL} K) \colon D(\Qcoh(X)) \ra
  D(\Qcoh(Y)) 
\end{equation*}
where $X \xla{p} X \times Y \xra{q} Y$ are the projections.
It is well-known that any choice of compact generators 
provides dg algebras $A$ and $B$ such that 
$D(\Qcoh(X))$ and $D(\Qcoh(Y))$ are equivalent to the derived
categories $D(A)$ and $D(B)$ of dg modules, respectively.
It is therefore natural to expect that there is a dg $A^\opp
\otimes B$-module $M$ corresponding to $K$ such that the
diagram
\begin{equation*}
  \xymatrix{
    {D(\Qcoh(X))} 
    \ar[d]^-{\sim}
    \ar[r]^-{\Phi_K} &
    {D(\Qcoh(Y))} \ar[d]^-{\sim}\\
    {D(A)} 
    \ar[r]^-{- \otimes^{\bL}_A M}
    &
    {D(B)}
  }
\end{equation*}
commutes.

\begin{theorem}[{Fourier-Mukai kernels and dg bimodules, see Theorem~\ref{t:translate-fm-to-dg}}]
  \label{t:FM-intro}
  Let $X$ and $Y$ be Noetherian separated schemes over a
  field $k$ such that $X \times Y$
  is Noetherian 
  and the following condition holds for both $X$ and $Y$: 
  any perfect complex is isomorphic to a strictly perfect complex
  (i.\,e.\ a bounded  
  complex of vector bundles).
  Then there is an equivalence
  $\theta \colon D(\Qcoh(X \times Y)) \sira D(A^\opp \otimes B)$
  such that for any $K \in D(\Qcoh(X \times Y))$ with
  corresponding $M=\theta(K) \in D(A^\opp \otimes B)$ 
  the diagram above commutes.
\end{theorem}


We refer the reader to Theorem~\ref{t:translate-fm-to-dg}
for a more precise formulation of this theorem.

We believe that this theorem is an important step in the writing
of the dictionary between derived categories of schemes and those
of dg algebras. 
The commutativity of the diagram is claimed under more general
assumptions in
\cite[after Cor.~8.12]{toen-homotopy-of-dg-cats-morita} without
proof. The main difficulty in the proof of
Theorem~\ref{t:FM-intro} 
arises from the fact 
that the different functors involved (inverse image, tensor
product, direct image, $\bR\!\Hom$) are usually computed via
different types of replacements (h-flat, h-injective) and it is
hard to treat these functors compatibly.

Our main tools to overcome these difficulties are 
new enhancements of the categories
$\mfPerf(X),$ $D^b(\Coh(X)),$ 
and $D^-(\Coh(X))$
that we introduce in this article
(see Propositions~\ref{p:abstract-cech-!-object-enhancement},
\ref{p:abstract-cech-*-object-enhancement},
\ref{p:shriek-cech-enhancement-D-minus-and-b-Coh}). 
These enhancements are modeled
on left and right ($!$ and $*$) \v{C}ech resolutions and are
certain
non-full subcategories of 
the dg category
of complexes of sheaves of $\mathcal{O}_X$-modules.

They also enable us to prove the following two
theorems. 
We call the category $\mfPerf(X)$ (resp.\ $D^b(\Coh(X))$)
smooth over $k$ if its h-injective enhancement is smooth over $k$
as a dg $k$-category (see Definition~\ref{d:Perf-DbCoh-smooth}).

\begin{theorem} 
  [{Homological versus geometric smoothness, see Theorem~\ref{t:mfPerf-abstract-Cechobj-smooth-vs-diagonal-sheaf-perfect}}]
  \label{t:perf-smooth-intro}
  Let $X$ be a Noetherian separated scheme over a field $k$ such that
  $X \times X$ is Noetherian and any perfect complex on $X$ is
  isomorphic to a strictly perfect complex.
  Let $\Delta \colon  X \ra X \times X$ be the diagonal
  (closed) immersion.
  Then the following two conditions are equivalent:
  \begin{enumerate}
  \item $\mfPerf(X)$ is smooth over $k;$
  \item $\Delta_*(\mathcal{O}_X) \in
  \mfPerf(X \times X).$
  \end{enumerate}
  If $X$ is in addition of finite type
  over $k,$  
  they are also equivalent to:
  \begin{enumerate}[resume]
  \item $X$ is smooth over $k.$
  \end{enumerate}
  In particular, if 
  $X$ is a separated scheme of finite type over $k$ 
  having
  the resolution
  property, i.\,e.\ any coherent sheaf is a quotient of a vector
  bundle, for example if $X$ is quasi-projective over $k$,
  then the above three conditions are equivalent.
\end{theorem}

\begin{theorem} 
  [{see Theorem~\ref{t:D-b-Coh-smooth}}]
  \label{t:DbCoh-smooth-intro}
  Let $X$ be a separated scheme of finite type over a perfect
  field $k$ that has the resolution property.
  Then $D^b(\Coh(X))$ is smooth over $k.$
\end{theorem}

To our knowledge
Theorem~\ref{t:perf-smooth-intro}
is ``well-known'' folklore. However, it seems there is no proof
available in the literature, cf.\ the footnote in the introduction of
\cite{shklyarov-serre-duality-cpt-smooth-arXiv}.
Versions of  
Theorems~\ref{t:perf-smooth-intro}
and \ref{t:DbCoh-smooth-intro}
are claimed in
\cite[Prop~3.13, Thm.~6.3]{lunts-categorical-resolution}. However
the proof of the key Proposition~6.17 there is incomplete.

Note the following consequence of
Theorem~\ref{t:perf-smooth-intro} (see Corollary~\ref{c:smooth-quasi-projective}):
if $X$ is a smooth quasi-compact separated scheme 
over a field $k,$ then $\mfPerf(X) =D^b(\Coh(X))$ is
smooth over $k.$ 

We also characterize properness of the
categories $\mfPerf(X)$ and $D^b(\Coh(X))$ geometrically.  We
call a triangulated $k$-linear category $\mathcal{T}$ proper over
$k$ if it has a classical generator and $\dim_k (\bigoplus_{n \in
  \DZ} \Hom_\mathcal{T}(E,[n]F)) < \infty$ for all objects $E,$
$F \in \mathcal{T}$ (see Definition~\ref{d:tricat-proper}).

\begin{theorem} 
  [{Homological versus geometric properness, see Theorem~\ref{t:scheme-proper-iff-Perf-proper}}]
  \label{t:intro-scheme-proper-iff-Perf-proper}
  %
  Let $X$ be a separated scheme of finite type over a field
  $k.$
  If $X$ is proper over $k$, then $\mfPerf(X)$ is proper over
  $k.$ 
  If $X$ has the resolution property, the converse is also true.
\end{theorem}

\begin{theorem} 
  [{see Theorem~\ref{t:scheme-proper+reg-iff-DbCoh-proper}}]
  \label{t:intro-scheme-proper+reg-iff-DbCoh-proper}
  Let $X$ be a separated scheme $X$ of finite type over a field
  $k.$
  Then $D^b(\Coh(X))$ is proper over $k$ if and only if $X$ is
  proper over $k$ and regular.
\end{theorem}

The proofs of the two
Theorems~\ref{t:intro-scheme-proper-iff-Perf-proper} 
and \ref{t:intro-scheme-proper+reg-iff-DbCoh-proper}
are short
and independent of the other results of this article.
A statement similar to Theorem~\ref{t:intro-scheme-proper-iff-Perf-proper}
appeared in the recent preprint
\cite{orlov-smooth-proper-glueing-arxiv}.



Let us finally mention that
we define and study \v{C}ech enhancements
for locally integral schemes in
appendix~\ref{sec:vcech-enhanc-loc-integral}; this appendix is
included because its results are 
used and referred to in \cite{valery-olaf-matrix-factorizations-and-motivic-measures}.

\subsection*{Acknowledgements}
\label{sec:acknowledgements}

We thank Ragnar-Olaf Buchweitz, Henning Krause, Alexander
Kuz\-net\-sov, Daniel Pomerleano, Anatoly Preygel, Paolo
Stellari, and Greg Stevenson for helpful discussions. The
results of this article were reported on at a workshop in
Oberwolfach in May 2014, see \cite{olaf-OWR-enhancements}. We
thank the participants for their interest. We also thank the
referee for detailed comments.

The first author was supported by NSA grant H98230-14-1-0110.
The second author was supported by postdoctoral
fellowships of the DAAD and the DFG,
and by the SPP 1388 and the SFB/TR 45 of the DFG.


\subsection*{Conventions}
\label{sec:conventions}

When we take products of schemes (resp.\ tensor products of
algebras 
or dg (= differential $\DZ$-graded) algebras
or modules over algebras) and work over a field $k$ we
write $\times$ 
(resp.\ $\otimes$) instead of $\times_k=\times_{\Spec k}$ (resp.\
$\otimes_k$). If $\mathcal{F}$ and $\mathcal{G}$ are sheaves of
$\mathcal{O}_X$-modules on a
ringed space $(X, \mathcal{O}_X),$ we usually abbreviate
$\mathcal{F} \otimes \mathcal{G} =
\mathcal{F} \otimes_{\mathcal{O}_X} \mathcal{G},$
$\sheafHom(\mathcal{F}, \mathcal{G})=
\sheafHom_{\mathcal{O}_X}(\mathcal{F}, \mathcal{G})$ and
$\mathcal{F}^\cek=\sheafHom(\mathcal{F}, \mathcal{O}_X).$
If $\mathcal{A}$ is a dg category, 
$D(\mathcal{A})$ denotes the derived category of dg
$\mathcal{A}$-modules and $\per(\mathcal{A})$ its subcategory
of compact (or perfect) objects.

\section{Derived categories of sheaves and subcategories}
\label{sec:deriv-categ-sheav-subcats}

Let $X$ be a scheme. By a sheaf we mean a sheaf of
$\mathcal{O}_X$-modules.
We denote by $\Sh(X)$ (resp.\ $\Qcoh(X)$) the category of sheaves
(resp.\ quasi-coherent sheaves) on $X.$


\subsection{Derived categories of sheaves}
\label{sec:deriv-categ-sheav}

Let $D(\Sh(X))$ (resp.\ $D(\Qcoh(X))$) denote the
(unbounded) 
derived category of sheaves (resp.\ quasi-coherent sheaves) on
$X.$ Let $D_{\Qcoh}(\Sh(X))$ be the full subcategory of $D(\Sh(X))$
consisting of objects with quasi-coherent cohomologies.
By
$\mfPerf'(X)$ we denote the full 
subcategory of $D(\Sh(X))$) whose objects are perfect, i.\,e.\
locally isomorphic to a bounded complex of vector
bundles;
it is a thick subcategory (see
\cite[Prop.~2.2.13]{thomason-trobaugh-higher-K-theory}).
If $X$ is quasi-compact and quasi-separated, 
then 
$\mfPerf'(X)$ consists precisely of the compact objects of
$D_\Qcoh(\Sh(X)),$ and $D_\Qcoh(\Sh(X))$ is generated by a single
perfect object (see \cite[Thm.~3.1.1]{bondal-vdbergh-generators}).

Assume that our scheme $X$ is quasi-compact and separated.
Then the obvious
functor $D(\Qcoh(X)) \ra D(\Sh(X))$ defines an equivalence
\begin{equation*}
  D(\Qcoh(X)) \sira D_\Qcoh(\Sh(X)) \subset D(\Sh(X))
\end{equation*}
(see \cite[Cor.~5.5]{neeman-homotopy-limits}).
By
$\mfPerf(X)$
we denote the full 
subcategory of $D(\Qcoh(X))$ corresponding to $\mfPerf'(X)$ under
this equivalence; it consists precisely of those objects that are locally
isomorphic to a bounded complex of vector bundles.

Assume in addition that $X$ is Noetherian. Let
$\Coh(X)$ be the category of coherent sheaves on $X$ and 
$D^-(\Coh(X))$ 
its bounded above derived
category. The obvious functor $D^-(\Coh(X)) \ra
D(\Qcoh(X))$ then defines an equivalence
\begin{equation}
  \label{eq:D-minus-coh-noetherian}
  D^-(\Coh(X)) \sira D^-_\Coh(\Qcoh(X)) \subset D(\Qcoh(X))
\end{equation}
where $D^-_\Coh(\Qcoh(X)) \subset D(\Qcoh(X))$ is the full
subcategory of complexes whose cohomologies are bounded above and
coherent 
(see \cite[Exp.~II, Prop.~2.2.2,
p.~167]{berthelot-grothendieck-illusie-SGA-6}). This of course
remains true if we replace ``$-$'' by ``$b$'' and ``bounded
above'' by ``bounded''. 
Since $X$ is quasi-compact we have $\mfPerf(X) \subset D^b_\Coh(\Qcoh(X)).$

The relations between the above categories are summarized by the
following diagram
where the upper index ``$\cpt$'' stands for ``compact objects''.

\begin{equation*}
  \xymatrix@C1.5em{
    {D^-(\Coh(X))} \ar[r]^-\sim & 
    {D^-_{\Coh}(\Qcoh(X))} \ar@{}[r]|-{\subset} &
    {D(\Qcoh(X))} \ar[r]^-\sim &
    {D_{\Qcoh}(\Sh(X))} \ar@{}[r]|-{\subset} &
    {D(\Sh(X))} \\
    {D^b(\Coh(X))} \ar[r]^-\sim \ar@{}[u]|-{\cup} & 
    {D^b_{\Coh}(\Qcoh(X))} \ar@{}[r]|-{\supset} \ar@{}[u]|-{\cup} &
    {D(\Qcoh(X))^\cpt} \ar[r]^-\sim \ar@{}[u]|-{\cup} &
    {D_{\Qcoh}(\Sh(X))^\cpt} \ar@{}[u]|-{\cup} \\
    & 
    &
    {\mfPerf(X)} \ar[r]^-\sim
    \ar@{}[u]|-{\verteq} 
    &
    {\mfPerf'(X).} \ar@{}[u]|-{\verteq} 
  }
\end{equation*}

\begin{proposition}
  \label{p:regular-vs-singularity-cat}
  Let $X$ be a Noetherian separated scheme.
  Then $\mfPerf(X)=D^b(\Coh(X))$ implies that $X$ is
  regular.
  If $X$ is of finite dimension, the converse is also true. 
\end{proposition}

\begin{proof}
  We always have $\mfPerf(X) \subset D^b(\Coh(X)).$
  Recall the Auslander-Buchsbaum-Serre theorem 
  (\cite[Thm.~2.2.7]{bruns-herzog-cm},
  \cite[IV.D]{serre-local-algebra}) which says that a Noetherian
  local ring $(A, \mfm)$ is regular if and only if
  $\gldim A < \infty$ if and only if
  $A/\mfm$ has finite projective dimension as an $A$-module;
  moreover, if
  $A$ is regular, then $\dim A= \gldim A.$


  Assume that $\mfPerf(X)=D^b(\Coh(X)).$
  Since any point of $X$ contains a closed point in its closure
  and the localization of a regular local ring is regular it is
  enough to show that the local ring of each closed point is
  regular.

  Let $x \in X$ be a closed point. Equip $\{x\}=\ol{\{x\}}$ with
  the induced 
  reduced scheme structure and let $i\colon \{x\} \ra X$ be the
  closed embedding. View $\kappa(x)=\mathcal{O}_{X,x}/\mfm_x$
  as a coherent sheaf on $\{x\}.$ Then $i_*(\kappa(x)) \in
  \Coh(X) \subset D^b(\Coh(X)) =\mfPerf(X).$ This implies that
  the restriction of $i_*(\kappa(x))$ to an affine open
  neighborhood $U$ 
  of $x$ has a finite resolution by
  finitely generated projective
  $\mathcal{O}_X(U)$-modules. Taking the stalk at 
  $x$ shows that the $\mathcal{O}_{X,x}$-module
  $(i_*(\kappa(x)))_x=\kappa(x)$ has finite projective dimension.
  Hence $\mathcal{O}_{X,x}$ is regular.

  Assume that $X$ is regular and of finite dimension. By
  intelligent truncation it is sufficient to show that any $F \in
  \Coh(X)$ is in $\mfPerf(X).$ Let $U=\Spec R \subset X$ be an
  affine open subset and $d=\dim R <\infty.$ 
  Choose an exact sequence $0 \ra K \ra P^{-d} \ra \dots \ra P^0 \ra
  F|_U \ra 0$ where all $P^i$ are
  finitely generated projective $R$-modules.
  Localizing at an arbitrary $\mfp \in \Spec R$ shows that
  $K_\mfp$ is a projective $R_\mfp$-module
  (here we use that $\gldim R_\mfp=\dim R_\mfp \leq d$).
  Hence $K$ is a
  projective $R$-module. This shows that $F|_U \in \mfPerf(U)$
  and hence $F \in \mfPerf(X).$
\end{proof}

\subsection{Resolution property}
\label{sec:resolution-property}

A Noetherian scheme is said to have the resolution property if
any coherent sheaf is a quotient of a vector bundle.
For example, 
any Noetherian separated scheme that is integral and locally
factorial (for example regular)
has the resolution property, by a theorem of Kleiman
\cite[Ex.~III.6.8]{Hart};
any Noetherian scheme with an ample family of line bundles
has the resolution property, by
\cite[Lemma.~2.1.3(b)]{thomason-trobaugh-higher-K-theory}.

We say that a scheme $X$ satisfies condition~\ref{enum:RES}
or that $X$ is a \ref{enum:RES}-scheme (for ``resolution'')
if
\begin{enumerate}[label=(RES)]
\item
  \label{enum:RES}
  $X$ is a Noetherian separated scheme of finite 
  dimension
  that has the resolution property.
\end{enumerate}

\begin{proposition}
  \label{p:D-minus-Coh(Sh)-strict-coherent}
  If $X$ satisfies condition~\ref{enum:RES} then 
  any object of $D^-_\Coh(\Sh(X))$ 
  is isomorphic in $D(\Sh(X))$ to a bounded above complex of
  vector bundles.
\end{proposition}

If on a Noetherian scheme $X$ any coherent sheaf is isomorphic in
$D(\Sh(X))$ to a bounded above complex of vector bundles, then it
is easy to see (using intelligent truncation) that $X$ has the
resolution property.  

\begin{proof}
  The objects of $D^-_\Coh(\Sh(X))$ are precisely the 
  pseudo-coherent complexes, by
  \cite[Example~2.2.8]{thomason-trobaugh-higher-K-theory}.

  Now observe that the proof of 
  \cite[Prop.~2.3.1.(e)]{thomason-trobaugh-higher-K-theory}
  works without the assumption that $X$ has an ample family.
  Its important ingredient
  \cite[Lemma~2.1.3(c)]{thomason-trobaugh-higher-K-theory}
  is true in our setting. Namely, if
  $\mathcal{G} \ra \mathcal{F}$ is an epimorphism
  of quasi-coherent sheaves with $\mathcal{F}$ coherent,
  then 
  there is a vector bundle $\mathcal{E}$ and a morphism
  $\mathcal{E} \ra \mathcal{G}$ such that the composition
  $\mathcal{E} \ra \mathcal{G} \ra \mathcal{F}$ is an
  epimorphism onto $\mathcal{F}.$ This follows from
  \cite[Exercise II.5.15]{Hart} and the resolution property.
\end{proof}

\subsection{Strictly perfect complexes}
\label{sec:strictly-perf-compl}

We say that a scheme $X$ satisfies condition~\ref{enum:GSP}
or that $X$ is a \ref{enum:GSP}-scheme (for
``globally strictly perfect'') if
\begin{enumerate}[label=(GSP)]
\item
  \label{enum:GSP}
  $X$ is a quasi-compact separated
  scheme such that any perfect
  complex on $X$ is isomorphic in $D(\Sh(X))$ to a bounded
  complex of vector bundles (= a strictly perfect
  complex). 
\end{enumerate}

\begin{example}
  \label{exam:perfect-on-affine}
  Any affine scheme $U$ satisfies condition~\ref{enum:GSP} by
  \cite[Prop.~2.3.1(d)]{thomason-trobaugh-higher-K-theory}.
  
  Another way to see this is as follows.
  Let $R=\Gamma(U, \mathcal{O}_U).$
  Since $U$ is quasi-compact and separated, we have
  $D(R)=D(\Qcoh(U)) \sira D_\Qcoh(\Sh(U))$ where $D(R)$ is the derived
  category of $R$-modules.
  It is well-known that $D(R)^\cpt=\per(R)$ where $\per(R)$
  consists of those complexes that are isomorphic to a bounded
  complex of projective $R$-modules.
  Hence $\per(R)=\mfPerf(U)\sira
  \mfPerf'(U).$ 

  In particular, the restriction of a perfect complex on an
  arbitrary scheme to any affine open subscheme is 
  isomorphic to a bounded complex of vector bundles.
\end{example}

We give some criteria for a scheme to satisfy
condition~\ref{enum:GSP} in 
the following Remark~\ref{rem:onGSP}. These criteria show: Any
scheme $X$ which is 
quasi-projective over an 
affine scheme satisfies condition~\ref{enum:GSP} (since it has an
ample family of line bundles).
Any regular (or, more generally, integral locally factorial)
Noetherian separated scheme 
satisfies condition~\ref{enum:GSP}
because it has the resolution property.

\begin{remark}
  \label{rem:onGSP}
  \rule{1mm}{0mm}
  \begin{enumerate}
  \item
    \label{enum:ample-family}
    A separated scheme $X$ that 
    has an ample family of line bundles (and hence is
    quasi-compact) satisfies
    condition~\ref{enum:GSP}, by
    \cite[Prop.~2.3.1(d)]{thomason-trobaugh-higher-K-theory}.
  \item 
    \label{enum:tt-lemma213c}
    Let $X$ be a quasi-compact separated scheme.
    Assume that for any epimorphism $\mathcal{G} \ra \mathcal{F}$
    of 
    quasi-coherent sheaves with $\mathcal{F}$ of finite type,
    there is a vector bundle $\mathcal{E}$ and a morphism
    $\mathcal{E} \ra \mathcal{G}$ such that the composition
    $\mathcal{E} \ra \mathcal{G} \ra \mathcal{F}$ is an
    epimorphism onto $\mathcal{F}.$
    Then $X$ satisfies 
    condition~\ref{enum:GSP}. This follows by inspection of the
    proof of 
    \cite[Prop.~2.3.1]{thomason-trobaugh-higher-K-theory},
    our condition being its important ingredient
    \cite[Lemma~2.1.3(c)]{thomason-trobaugh-higher-K-theory}.
  \item 
    \label{enum:resolution-property}
    Let $X$ be a Noetherian separated scheme that has the
    resolution property.
    (This is a little bit weaker than condition~\ref{enum:RES}.)
    Then $X$ satisfies condition~\ref{enum:GSP}. 
    This follows easily from \ref{enum:tt-lemma213c} using
    \cite[Exercise II.5.15]{Hart}. 
  \end{enumerate}
  We refer the reader to 
  \cite[2.1.2]{thomason-trobaugh-higher-K-theory}
  and the discussion in \cite[section 2]{totaro-resolution} for
  examples and more information. 
\end{remark}


\subsection{Some useful facts}
\label{sec:some-useful-facts}

The category of injective sheaves (resp.\ injective
quasi-coherent sheaves) on a scheme $X$ is denoted $\InjSh(X)$
(resp.\ $\InjQcoh(X)$).

\begin{theorem}
  \label{t:injective-in-qcoh-vs-all-OX}
  Let $X$ be a locally Noetherian scheme.
  \begin{enumerate}
  \item
    \label{enum:qcoh-embeds-in-inj-OX-that-qcoh}
    Every object of $\Qcoh(X)$ can be embedded in an object of
    $\InjSh(X) \cap \Qcoh(X).$ 
  \item
    \label{enum:inj-qcoh-equal-inj-OX-that-qcoh}
    The injective objects in $\Qcoh(X)$ are precisely the
    injective objects of $\Sh(X)$ that are
    quasi-coherent, $\InjQcoh(X) = \InjSh(X) \cap
    \Qcoh(X).$
  \item 
    \label{enum:inj-qcoh-restrict-to-inj-qcoh}
    If $I \in \Qcoh(X)$ is an injective object and $U \subset
    X$ is open, then $I|_U \in \Qcoh(U)$ is again injective.
  \item 
    \label{enum:inj-and-inj-qcoh-arbitrary-sums}
    Any direct sum of objects of $\InjSh(X)$ (resp.\
    $\InjQcoh(X)$) is in $\InjSh(X)$ (resp.\ $\InjQcoh(X)$).
  \end{enumerate}
\end{theorem}

\begin{proof}
  This follows from \cite[II.\S 7]{hartshorne-residues-duality}
  as explained in
  \cite[Thm.~2.1]{valery-olaf-matfak-semi-orth-decomp}
  (cf.\ \cite[Lemma\ 2.1.3]{conrad-grothendieck-duality-bc}).
\end{proof}

\begin{lemma}
  \label{l:qcqs-preserves-qcoh}
  If $f:X \ra Y$ is a quasi-compact quasi-separated morphism of
  schemes (for
  example an affine morphism or a morphism with Noetherian
  source or a morphism between quasi-compact quasi-separated
  schemes), then $f_*:\Sh(X) \ra \Sh(Y)$ preserves
  coproducts (even filtered colimits),
  maps $\Qcoh(X)$ to 
  $\Qcoh(Y),$ and the
  induced functor 
  $f_*:\Qcoh(X) \ra \Qcoh(Y)$ preserves coproducts.
\end{lemma}

\begin{proof}
  See \cite[Lemmata~B.6,
  B.12]{thomason-trobaugh-higher-K-theory}. The statement that
  $f_*$ maps $\Qcoh(X)$ to $\Qcoh(Y)$ is also shown in
  \cite[Cor.~10.27]{goertz-wedhorn-AGI}.
\end{proof}

\begin{lemma}
  \label{l:u'-ls-qls-acyclic}
  Let $X \xra{f} Y \xra{g} Z$
  be morphisms. Assume that 
  $g$ is a quasi-compact separated morphism and that 
  $g \comp f$ is an affine morphism. 
  %
  Then for any $M \in \Qcoh(X),$ the object $f_*(M) \in \Qcoh(Y)$
  is acyclic with respect to the functor 
  $g_*:\Qcoh(Y) \ra
  \Qcoh(Z).$
\end{lemma}

\begin{proof}
  Note that $f$ is then affine
  and hence $f_*(M) \in \Qcoh(Y)$ by
  Lemma~\ref{l:qcqs-preserves-qcoh}. 
  From \cite[Cor.~1.3.2 and Cor.~1.3.4]{EGAIII-i} we see that
  $(\bR^p_\Sh g_*)(f_*(M))$ vanishes for all $p>0$ where
  $\bR^p_\Sh g_*$ 
  is the $p$-th right derived functor of $g: \Sh(Y) \ra \Sh(Z).$
  This means that $f_*(M)$ is acyclic with respect to 
  $g_*:\Sh(Y) \ra \Sh(Z).$  
  From \cite[Cor.~B.9]{thomason-trobaugh-higher-K-theory}
  we obtain 
  $(\bR^p_\Qcoh g_*)(f_*(M))=0$ for all $p>0$ where $\bR^p_\Qcoh g_*$
  is the $p$-th right derived functor of $g: \Qcoh(Y) \ra
  \Qcoh(Z).$
  This means that $f_*(M)$ is acyclic with respect to 
  $g_*:\Qcoh(Y) \ra \Qcoh(Z).$  
\end{proof}

\begin{lemma}
  \label{l:fin-coho-dim}
  If $f \colon X \ra Y$ is a morphism between quasi-compact
  separated schemes over a field $k$
  then 
  the functor $f_*: \Qcoh(X) \ra \Qcoh(Y)$ has finite
  cohomological dimension. 
\end{lemma}

\begin{proof}
  The functor $f_*: \Qcoh(X) \ra \Qcoh(Y)$ is well-defined by
  Lemma~\ref{l:qcqs-preserves-qcoh}.
  If $U$ is an affine open subset of $X$ note that the
  composition $U \hra X \xra{f} Y$ is affine because its
  composition with the separated morphism $Y \ra \Spec k$ is
  affine. 
  Fixing some ordered finite
  affine open covering of $X$ provides 
  for any 
  $A \in 
  \Qcoh(X)$
  the finite $*$-\v{C}ech-resolution $A \ra
  \mathcal{C}_*(A)$, cf. \eqref{eq:F-*-Cech-resolution}.
  Lemma~\ref{l:u'-ls-qls-acyclic} then implies that 
  $\bR f_*(A)\cong f_*(\mathcal{C}_*(A))$ in $D(\Qcoh(Y)).$
\end{proof}



\section{Enhancements}
\label{sec:enhancements}

We denote the dg category of
complexes in a (pre)additive category $\mathcal{A}$ by
$C(\mathcal{A}).$ If $\mathcal{D}$ is any dg category, we denote
by $Z^0(\mathcal{D})$ the category with the same objects but
closed degree zero morphisms, and by $[\mathcal{D}]$ the
homotopy category of $\mathcal{D}.$

For example, $[C(\Qcoh(X))]$ denotes the category of complexes of
quasi-coherent sheaves on $X$ whose morphisms are given by
homotopy classes of degree
zero maps that commute with the respective differentials.

\subsection{Injective enhancements}
\label{sec:inject-enhanc}

Let $X$ be a scheme. Then $\Sh(X)$ is a Grothendieck category
(\cite[18.1.6.(v)]{KS-cat-sh}), the full dg
subcategory $C^\hinj(\Sh(X))$ of $C(\Sh(X))$ consisting of
h-injective objects is pretriangulated and the canonical functor
$[C^\hinj(\Sh(X))] \ra D(\Sh(X))$ is an equivalence of
triangulated categories
(\cite[Thm.~14.3.1.(iii)]{KS-cat-sh}), so $C^\hinj(\Sh(X))$
is naturally an 
enhancement of $D(\Sh(X)).$ 
Similarly, if $X$ is a quasi-compact and quasi-separated scheme,
then $\Qcoh(X)$ is a Grothendieck category
(\cite[B.3]{thomason-trobaugh-higher-K-theory}), so that the full
dg subcategory $C^\hinj(\Qcoh(X))$ of $C(\Qcoh(X))$ consisting of
h-injective objects is an enhancement of $D(\Qcoh(X)).$

Assume that $X$ is quasi-compact and separated.
By restricting to suitable subcategories we obtain enhancements
of all the full triangulated subcategories of $D(\Qcoh(X))$ and
$D(\Sh(X))$ mentioned so far. We give some examples.

The full dg subcategory 
$C^\hinj_\Qcoh(\Sh(X))$
of $C^\hinj(\Sh(X))$ consisting of complexes
with quasi-coherent cohomologies is 
an enhancement of $D(\Qcoh(X)) \sira D_\Qcoh(\Sh(X)).$

The full dg subcategory
$\mfPerf^\hinj(X)$
of $C^\hinj(\Qcoh(X))$ 
of objects belonging to $\mfPerf(X)$
and also its
full dg subcategory
$\mfPerf^\inj(X)$ of
bounded below complexes of injective quasi-coherent
sheaves
are enhancements of $\mfPerf(X) \sira \mfPerf'(X).$
Similarly, by considering arbitrary sheaves instead of
quasi-coherent ones, we obtain enhancements $\mfPerf'^\hinj(X)$ and
$\mfPerf'^\inj(X)$ of $\mfPerf'(X).$

If $X$ is a Noetherian separated scheme,
$D^b(\Coh(X))
\sira D^b_\Coh(\Qcoh(X))
\sira D^b_\Coh(\Sh(X))$ has the following enhancements:
the full dg
subcategories
$C_\Coh^{\hinj,b}(\Sh(X))$ 
of $C^\hinj(\Sh(X))$
and 
$C_\Coh^{\hinj,b}(\Qcoh(X))$ 
of $C^\hinj(\Qcoh(X))$
consisting of objects with
bounded coherent
cohomology; the full dg
subcategory
$C_\Coh^{+,b}(\InjSh(X))$ 
of 
$C_\Coh^{\hinj,b}(\Sh(X))$ 
consisting of bounded below complexes of
injective sheaves with bounded coherent
cohomology; the full dg
subcategory
$C_\Coh^{+,b}(\InjQcoh(X))$ 
of 
$C_\Coh^{\hinj,b}(\Qcoh(X))$ 
consisting of bounded below complexes of
injective quasi-coherent sheaves with bounded coherent
cohomology.

\begin{remark}
  \label{rem:fibrant-explained}
  Instead of h-injective complexes we could also use 
  fibrant complexes in order to define enhancements; any
  complex admits a monomorphic quasi-isomorphism to a fibrant
  complex; a complex is fibrant
  if and only if it is an h-injective complex of injective
  objects; in particular, any bounded below complex of
  injective objects is fibrant;
  cf.~\cite[12]{wolfgang-olaf-locallyproper}.
\end{remark}

\subsection{\v{C}ech enhancements}
\label{sec:vcech-enhanc}

\subsubsection{\v{C}ech resolutions }
\label{sec:vcech-resolutions-}

Let $X$ be a quasi-compact separated scheme. 
If $j \colon V \hra X$ is the
inclusion of an open subscheme and $F$ is a sheaf on $X,$ we define
$\leftidx{_V}{F}{}:=j_*j^*(F)=j_*j^!(F)$ and
$\leftidx{^V}{F}{}:=j_!j^!(F)=j_!j^*(F).$
Let $\mathcal{U}=(U_s)_{s \in S}$ be an ordered finite
open covering
of $X$;
here
``ordered finite'' means that $S$ is a totally ordered finite
set. 
As usual, we abbreviate $U_I:=
\bigcap_{i \in I} U_i$ for a subset $I \subset S,$ and write
$U_{s_0s_1 \dots s_n}$ instead of $U_{\{s_0,s_1, \dots, s_n\}}.$
For any sheaf $F$ on $X$ we can consider its 
bounded $*$-\v{C}ech resolution
\begin{equation}
  \label{eq:F-*-Cech-resolution}
  F \ra 
  \mathcal{C}_*(F):=
  \Big(\prod_{s_0 \in S} \leftidx{_{U_{s_0}}}{F}{} \ra
  \prod_{s_0, s_1 \in S,\; s_0 < s_1} \leftidx{_{U_{s_0s_1}}}{F}{} \ra
  \dots
  \Big)
\end{equation}
and its 
bounded $!$-\v{C}ech resolution
\begin{equation}
  \label{eq:F-!-Cech-resolution}
  \Big(
  \mathcal{C}_!(F):=
  \Big(
  \ldots 
  \ra
  \prod_{s_0, s_1 \in S,\; s_0 < s_1} \leftidx{^{U_{s_0s_1}}}{F}{}
  \ra
  \prod_{s_0 \in S} \leftidx{^{U_{s_0}}}{F}{} 
  \Big)
  \Big)\ra F
\end{equation}
with the usual differentials. 
Both $\mathcal{C}_*(F)$ and $\mathcal{C}_!(F)$ depend  on
$\mathcal{U}$ but we do not emphasize this in the notation.
(Choosing another total ordering on $S$ gives rise to isomorphic
resolutions.)

More generally, if $F$ is a complex of sheaves on $X,$ 
then $\mathcal{C}_*(F)$ and
$\mathcal{C}_!(F)$ are defined as the
totalizations of the obvious double complexes
and we have
natural quasi-isomorphisms $F \ra \mathcal{C}_*(F)$ and
$\mathcal{C}_!(F) \ra F$ (use \cite[Lemma~III.4.2]{Hart} for the
$*$-\v{C}ech resolution, look at the stalks for the $!$-\v{C}ech
resolution, and use \cite[Thm.~1.9.3]{KS}). 
This means that $F \mapsto \mathcal{C}_*(F)$ and $F \mapsto
\mathcal{C}_!(F)$ define two dg functors
\begin{equation*}
  \mathcal{C}_*, \mathcal{C}_!\colon C(\Sh(X)) \ra C(\Sh(X)).
\end{equation*}
They come with morphisms $\id \ra
\mathcal{C}_*$ and $\mathcal{C}_! \ra \id$ of dg functors that
are given at each object by a quasi-isomorphism.

The complexes $\mathcal{C}_*(F)$ (resp.\ $\mathcal{C}_!(F)$) are
complexes of sheaves that are
finite direct sums of objects 
$\leftidx{_V}{(F^i)}{}$ (resp.\ $\leftidx{^V}{(F^i)}{}$) where
$V$ is a finite
intersection of elements of $\mathcal{U}$ and $i \in \DZ.$
 Note that $\mathcal{C}_*(F)$ and $\mathcal{C}_!(F)$ are bounded
(resp.\ bounded above, bounded below) if $F$ is  bounded
(resp.\ bounded above, bounded below). 
If $F$ is a complex of vector bundles
(= locally free sheaves of finite type),
$\mathcal{C}_!(F)$ is a complex of flat sheaves.
If all $U_s$ are quasi-compact and
$F \in C(\Qcoh(X))$ we have $\mathcal{C}_*(F) \in
C(\Qcoh(X))$ by 
Lemma~\ref{l:qcqs-preserves-qcoh}
because 
all
inclusions $U_I \ra X$ are quasi-compact 
and separated. 

\subsubsection{Some auxiliary constructions}
\label{sec:some-auxil-constr}

Let $X$ be a scheme. 
We define an additive category 
$\Vb^\subset(X)$
as follows. Its objects are finite formal direct sums of pairs
$(U,P)$ where $U \subset X$
is an open subset and
$P$ is a vector bundle on $U.$
Morphism spaces are defined by
\begin{equation*}
  \Hom_{\Vb^\subset(X)}((U,P), (V,Q)) :=
  \begin{cases}
    \Hom_{\mathcal{O}_V}(j_!P, Q), & \text{if $U \subset V$,
      where $j\colon U \subset V$;}\\
    0, & \text{otherwise.}
  \end{cases}
\end{equation*}
Identities are obvious, and composition is defined by
\begin{align*}
  \Hom_{\Vb^\subset(X)}((V,Q), (W,R))
  \times
  \Hom_{\Vb^\subset(X)}((U,P), (V,Q))
  & \ra
  \Hom_{\Vb^\subset(X)}((U,P), (W,R)),\\
  (g,f) & \mapsto
  \begin{cases}
    g \comp j'_!f, & \text{if $U \subset V \subset W$,}\\
    & \text{where $j'\colon V \subset W$;}\\
    0, & \text{otherwise.}
  \end{cases}
\end{align*}

Similarly, we define an additive category $\Vb_\supset(X).$ It has
the same objects as $\Vb^\subset(X),$ morphism spaces
\begin{equation*}
  \Hom_{\Vb_\supset(X)}((V,Q), (U,P)) :=
  \begin{cases}
    \Hom_{\mathcal{O}_V}(Q, j_*P), & 
    \text{if $V \supset U$,
      where $j\colon U \subset V$;}\\
    0, & \text{otherwise,}
  \end{cases}
\end{equation*}
and obvious identities and composition.

\begin{lemma}
  \label{l:Hom-u!Pv!Q-v*Qu*P}
  Let $U \subset V \subset X$ be open subschemes of a scheme $X,$ 
  with inclusion morphisms  
  $u \colon U \subset X$, $j\colon U \subset V$ and $v \colon V
  \subset X$. 
  Let $P \in \Sh(U)$ and $Q \in \Sh(V)$.
  Then restriction to $V$ and $U$ yields 
  natural isomorphisms
  \begin{equation*}
    \Hom_{\mathcal{O}_X}(u_!P, v_!Q)
    =
    \Hom_{\mathcal{O}_V}(j_!P,Q)    
    =
    \Hom_{\mathcal{O}_U}(P,j^*Q)
  \end{equation*}
  and
  \begin{equation*}
    \Hom_{\mathcal{O}_X}(v_*Q, u_*P)
    =
    \Hom_{\mathcal{O}_V}(Q,j_*P)
    =
    \Hom_{\mathcal{O}_U}(j^*Q,P).
  \end{equation*}
\end{lemma}

\begin{proof}
  All four isomorphisms are obtained from adjunction isomorphisms
  using
  $u_!P= v_!j_!P$, $Q \sira v^*v_!Q$,
  and $u_*P=v_*j_*P$, $v^*v_*Q \sira Q$.
  %
  %
\end{proof}

\begin{remark}
  \label{rem:alternative-def-Vb}
  Lemma~\ref{l:Hom-u!Pv!Q-v*Qu*P} provides alternative
  equivalent definitions of the categories
  $\Vb^\subset(X)$ and $\Vb_\supset(X)$. A priori it just gives
  equivalent descriptions of the morphism spaces, but when working
  with them, composition is defined in the obvious
  way. 
\end{remark}

\paragraph{Realization functors and duality}
\label{sec:real-funct-dual}

Lemma~\ref{l:Hom-u!Pv!Q-v*Qu*P} shows that we obtain faithful
additive realization functors 
\begin{align}
  \label{eq:real-subset}
  \Vb^\subset(X) \ra \Sh(X), & \quad
  (U,P) \mapsto u_!P, \text{ and}\\
  \label{eq:real-supset}
  \Vb_\supset(X) \ra \Sh(X), & \quad
  (U,P) \mapsto u_*P,
\end{align}
where $u\colon U \subset X$ is the inclusion.


\begin{remark}
  \label{rem:not-full}
  In general, these realization functors are not full,
  see 
  Remarks~\ref{rem:Hom-UPVQ-weaker-assumptions}
  and
  \ref{rem:shriek-extension-counterexample-to-lemma}.\ref{enum:shriek-extension}.
  Lemmata~\ref{l:Hom-UPVQ-star} and
  \ref{l:Hom-UPVQ-shriek} provide classes of objects where these
  functors are full.
\end{remark}

We define an additive functor
\begin{align*}
  (-)^\cek: \Vb^\subset(X)^\opp & \ra \Vb_\supset(X),\\
  (U,P) & \mapsto (U,P^\cek),
\end{align*}
which is given on morphism spaces by
\begin{align}
  \label{eq:transpose-morphisms}
  \Hom_{\Vb^\subset(X)}((U,P), (V,Q)) & \ra
  \Hom_{\Vb_\supset(X)}((V,Q^\cek), (U,P^\cek)),\\
  \notag
  f & \mapsto f^\cek
\end{align}
where we use $(j_!P)^\cek=\sheafHom(j_!P,
\mathcal{O}_V) \cong j_*\sheafHom(P,
j^*\mathcal{O}_V)=j_*(P^\cek)$ if 
$j\colon U \subset V$
(cf.\ \cite[Exp.~I, Cor.~1.5]{SGA-2-new} for the isomorphism).
Using Lemma~\ref{l:Hom-u!Pv!Q-v*Qu*P}, the map
\eqref{eq:transpose-morphisms} for $U \subset V$ 
is given by 
\begin{equation*}
  (-)^\cek \colon \Hom_{\mathcal{O}_U}(P, Q|_U) \ra
  \Hom_{\mathcal{O}_U}(Q^\cek|_U, P^\cek)
\end{equation*}
where we use the obvious identification $(Q|_U)^\cek=Q^\cek|_U$.
This description clearly shows that $(-)^\cek$ is a duality with
inverse $(-)^\cek$ defined 
similarly, $\id \sira ((-)^\cek)^\cek.$ 

If $u\colon U \ra X$ is an open immersion and $P$ a vector bundle
on $U$ we have
%
just observed that
$\sheafHom(u_!P, \mathcal{O}_X)
\cong u_*(P^\cek).$
This implies that the diagram
\begin{equation}
  \label{eq:Vb-sup-sub-duality}
  \xymatrix{
    {\Vb^\subset(X)^\opp} \ar[d]_-{(-)^\cek}^-{\sim} \ar[rr]^-{(U,P)
      \mapsto u_!P}_-{\eqref{eq:real-subset}}  
    && {\Sh(X)^\opp} \ar[d]^-{(-)^\cek=\sheafHom(-,\mathcal{O}_X)}\\
    {\Vb_\supset(X)} \ar[rr]^-{(U,P) \mapsto
      u_*P}_-{\eqref{eq:real-supset}}  
    && {\Sh(X)}
  }
\end{equation}
commutes up to a natural isomorphism.

\subsubsection{Construction of \v{C}ech enhancements}
\label{sec:constr-vcech-enhanc}

Assume that $X$ is quasi-compact separated.
Let $\mathcal{U}=(U_s)_{s \in S}$ 
be an ordered
finite 
open covering of $X.$
We fix this covering for the rest of this section; mostly we
will even assume that all $U_s$ are affine. 
Although many constructions will depend on
this fixed covering we usually do not indicate this in our
notation.
If $P$ is a vector bundle on $X$ we may
consider the complex
\begin{equation*}
  P_\supset:=\Big(
  \prod_{s_0 \in S} (U_{s_0}, P|_{U_{s_0}}) \ra
  \prod_{s_0, s_1 \in S,\; s_0 < s_1} (U_{s_0s_1}, P|_{U_{s_0s_1}}) 
  \ra
  \dots
  \Big)
\end{equation*}
(whose first term is in degree zero) in $\Vb_\supset(X)$ whose
differentials are defined in the obvious way such that the image
of this complex under the functor \eqref{eq:real-supset} is the
complex $\mathcal{C}_*(P)$ (see \eqref{eq:F-*-Cech-resolution}).
Here we use that the additive category $\Vb_\supset(X)$ has
finite products. Similarly, if $P$ is a complex of vector bundles
on $X,$ we define $P_\supset \in C(\Vb_\supset(X))$ as the
totalization of the obvious double complex. This defines a dg
functor $(-)_\supset$ from the dg category of complexes of vector
bundles to $C(\Vb_\supset(X)).$

We define the dg category $\Cech_*(X)$ 
as follows. Its objects are bounded complexes $P$ of vector
bundles on $X,$ and morphism spaces are defined by 
\begin{equation*}
  \Hom_{\Cech_*(X)}(P,Q) := \Hom_{C(\Vb_\supset(X))}(P_\supset, Q_\supset)
\end{equation*}
with obvious identities and composition. 
There is an obvious full and faithful dg functor $\Cech_*(X) \ra
C(\Vb_\supset(X)),$ $P \mapsto P_\supset.$ If we compose it with the dg functor induced
by \eqref{eq:real-supset} we obtain a faithful dg functor
\begin{equation}
  \label{eq:*-cech-realization}
  \mathcal{C}_*: \Cech_*(X) \ra 
  C(\Sh(X))
\end{equation}
which we call $\mathcal{C}_*$ since it maps
an object $P$ to $\mathcal{C}_*(P).$ In general, this functor is
not full, see Remark~\ref{rem:realization-full}.


Similarly, for each vector bundle $P$ we consider the complex
\begin{equation*}
  P^\subset:=
  \Big(
  \ldots 
  \ra
  \prod_{s_0, s_1 \in S,\; s_0 < s_1} (U_{s_0s_1},P|_{U_{s_0s_1}})
  \ra
  \prod_{s_0 \in S} (U_{s_0},P|_{U_{s_0}})
  \Big)
\end{equation*}
(whose last term is in degree zero) in $\Vb^\subset(X)$ whose
differentials are defined in the obvious way such that the image of
this complex under the functor \eqref{eq:real-subset} is
$\mathcal{C}_!(P)$ (see \eqref{eq:F-!-Cech-resolution}). By
totalization we define $P^\subset$ for complexes $P$ of vector
bundles. Let $\Cech_!(X)$ (resp.\ $\Cech_!^-(X)$) be the dg
category with objects bounded (resp.\ bounded above)
complexes of
vector bundles and morphism spaces
\begin{equation*}
  \Hom_{\Cech_!^{\natural}(X)}(P,Q) :=
  \Hom_{C(\Vb^\subset(X))}(P^\subset, Q^\subset) 
\end{equation*}
(where 
$\Cech_!^{\natural}(X)$ is
$\Cech_!(X)$ (resp.\ $\Cech_!^-(X)$))
with obvious identities
and composition.  There is an obvious full and faithful dg
functor $\Cech_!^\natural(X) \ra C(\Vb^\subset(X)),$ $P \mapsto
P^\subset.$ If we
compose it with the dg functor induced by \eqref{eq:real-subset}
we obtain a faithful dg functor
\begin{equation}
  \label{eq:!-cech-realization}
  \mathcal{C}_!: \Cech_!^\natural(X) \ra C(\Sh(X)) 
\end{equation}
which we call $\mathcal{C}_!$ since it maps
an object $P$ to $\mathcal{C}_!(P).$ In general, this functor is
not full, see Remark~\ref{rem:realization-full}.

Obviously, the dg categories 
$\Cech_*(X),$ 
$\Cech_!(X),$ 
$\Cech_!^-(X)$
contain all shifts $[m]P$ of their objects $P,$ and the
functors 
\eqref{eq:*-cech-realization} and
\eqref{eq:!-cech-realization}
are compatible with
shifts.  

\begin{remark}
  \label{rem:realization-full}
  If $X$ is a Noetherian separated integral scheme
  and all $U_s$ are affine
  then the
  realization 
  functor 
  \eqref{eq:!-cech-realization}
  is full (and faithful)
  as follows from 
  Lemma~\ref{l:Hom-UPVQ-shriek}.
  If $X$ is in addition Nagata then 
  \eqref{eq:*-cech-realization} is full (and faithful) by
  Lemma~\ref{l:Hom-UPVQ-star}. 

  These statements are not longer true in general if $X$ is Nagata
  quasi-compact separated 
  and locally integral: let $X$ be the disjoint union
  of two non-empty affine Nagata integral schemes $U$ and
  $V$ and 
  consider the open covering $\mathcal{U}=\{X,U,V\}.$
\end{remark}



Diagram~\eqref{eq:Vb-sup-sub-duality} 
induces a similar diagram on the level of complexes
whose left vertical arrow
$(-)^\cek \colon
C(\Vb^\subset(X))^\opp \sira
C(\Vb_\supset(X))$
maps 
$P^\subset$ to $(P^\cek)_\supset$. From this diagram 
we obtain the diagram
\begin{equation}
  \label{eq:duality-cech-enhancements}
  \xymatrix{
    {\Cech_!(X)^\opp} \ar[d]_-{(-)^\cek}^-{\sim}
    \ar[r]^-{\mathcal{C}_!} 
    & {C(\Sh(X))^\opp} \ar[d]^-{\sheafHom(-,\mathcal{O}_X)}\\
    {\Cech_*(X)} \ar[r]^-{\mathcal{C}_*}
    & {C(\Sh(X))}
  }
\end{equation}
of dg categories which is commutative up to a natural
isomorphism.
Its left vertical arrow is an
isomorphism of dg 
categories (and given on objects by $P \mapsto P^\cek$).

\begin{remark}
  \label{rem:warning-Vb-sub-sub-needs-dual}
  Let $P$ and $Q$ be bounded vector bundles on $X.$ 
  If $U$ and $V$ are open subsets of $X$ then
  \begin{equation*}
    \Hom_{\Vb^\subset(X)}((U,P|_U), (V,Q|_V))
    =
    \Hom_{\Vb_\supset(X)}((V,P|_V), (U,Q|_U))
  \end{equation*}
  since both sides are equal to $\Hom_{\mathcal{O}_U}(P|_U,
  Q|_U)$ if $U \subset V,$ 
  by Lemma~\ref{l:Hom-u!Pv!Q-v*Qu*P},
  and zero otherwise.
  These equalities combine to an isomorphism of dg modules
  \begin{equation*}
    \Hom_{\Cech_!(X)}(P,Q) \sira
    \Hom_{\Cech_*(X)}(P,Q).
  \end{equation*}
  It is however not true that these isomorphisms (together with
  the identity map on objects) define an isomorphism of dg
  categories 
  $\Cech_!(X) \ra \Cech_*(X)$: compatibility with composition is
  violated (except for very trivial situations).
\end{remark}

The following proposition is the main ingredient for showing that
our construction provides enhancements (Propositions~\ref{p:abstract-cech-!-object-enhancement},
\ref{p:abstract-cech-*-object-enhancement},
\ref{p:shriek-cech-enhancement-D-minus-and-b-Coh}).

\begin{proposition}
  \label{p:abstract-cech-!-bounded-above-full-faithful}
  Let $X$ be a quasi-compact separated scheme with
  an 
  ordered finite affine open covering 
  $\mathcal{U}=(U_s)_{s \in S}.$
  Let $P$ and $Q$ be complexes of vector bundles on
  $X,$ with $P$
  bounded above. Then the morphism
  \begin{equation}
    \label{shriek-enhancement-iso}
    \Hom_{[C(\Vb^\subset(X))]}(P^\subset, Q^\subset)
    \ra
    \Hom_{D(\Sh(X))}(\mathcal{C}_!(P), \mathcal{C}_!(Q))
  \end{equation}
  induced by \eqref{eq:real-subset}
  is an isomorphism.
\end{proposition}

\begin{remark}
  The assumption that $P$ is bounded above is necessary: 
  Let $X=\Spec A$ where $A=k[\epsilon]/(\epsilon^2)$ with $k$ a
  field, and consider the trivial covering $\mathcal{U}=\{X\}$ of $X.$ 
  Let $P$ be the complex $\ldots \ra A \xra{\epsilon} A
  \xra{\epsilon} A \ra \dots.$ Then $P=\mathcal{C}_!(P)$ is zero
  in $D(\Sh(X))$ but $0\not= \id
  \in \End_{[C(\Vb^\subset(X))]}(P^\subset) = \End_{[C(\Qcoh(X))]}(P).$
\end{remark}

\begin{proof}
  We first prove this under the additional assumption that $P$ is
  bounded.
  We can even assume that $P$ is a vector bundle sitting in
  a single degree: brutal truncation provides a sequence of
  closed degree zero morphisms $\sigma^{\geq p}(P) \ra P \ra
  \sigma^{<p}(P) \ra 
  [1]\sigma^{\geq p}(P)$ in the dg category of complexes of
  vector bundles, for any $p \in \DZ$, and we can apply the
  functors 
  $(-)^\subset$ and $\mathcal{C}_!$ to this sequence.

  Let $q \colon \mathcal{C}_!(Q) \ra Q$ be the $!$-\v{C}ech
  resolution and 
  $q'\colon Q^\subset \ra (X,Q)$ 
  the
  obvious morphism 
  in $Z^0(C(\Vb^\subset(X)))$ 
  whose image under 
  the functor induced by \eqref{eq:real-subset}
  is $q.$ Consider the commutative
  diagram
  \begin{equation}
    \label{eq:diagram-for-shriek-enhancement-iso}
    \xymatrix{
      {\Hom_{[C(\Vb^\subset(X))]}(P^\subset, Q^\subset)}
      \ar[r]
      \ar[d]_-{q'_*}
      &
      {\Hom_{D(\Sh(X))}(\mathcal{C}_!(P), \mathcal{C}_!(Q))}
      \ar[dd]_-{q_*}^-{\sim} \\
      {\Hom_{[C(\Vb^\subset(X))]}(P^\subset, (X,Q))}
      \ar[d]^-{\sim}
      \\
      {\Hom_{[C(\Sh(X))]}(\mathcal{C}_!(P),Q)}
      \ar[r]^-{\can} & 
      {\Hom_{D(\Sh(X))}(\mathcal{C}_!(P), Q)}
    }
  \end{equation}
  whose right vertical arrow $q_*$ and lower left vertical
  arrow are obviously isomorphisms.
  Since $\mathcal{C}_!(P)$ is a
  bounded 
  complex with components finite products of objects
  $\leftidx{_{U_J}}{P^p}{}$, where $\emptyset \not= J \subset S$ 
  and $p \in \DZ$, 
  Lemma~\ref{l:hom-u!-qcoh} below implies that
  $\can$ is an
  isomorphism as well.
  %

  We need to show that the upper left vertical map $q'_*$ in 
  \eqref{eq:diagram-for-shriek-enhancement-iso}
  is an isomorphism. Equivalently, we show that
  \begin{equation*}
    q'_* \colon  
    \Hom_{C(\Vb^\subset(X))}(P^\subset, Q^\subset)
    \ra
    \Hom_{C(\Vb^\subset(X))}(P^\subset, (X,Q))
  \end{equation*}
  is a quasi-isomorphism for any vector bundle $P$ considered as
  a complex concentrated
  in degree zero.
  Applying brutal truncation to $P^\subset$ (and passing to
  direct summands and shifting) we see that
  it is 
  enough to show that for an arbitrary fixed $\emptyset \not= I
  \subset S$ the morphism
  \begin{equation}
    \label{eq:basic-part}
    \Hom_{C(\Vb^\subset(X))}((V,P|_V), Q^\subset)
    \ra
    \Hom_{C(\Vb^\subset(X))}((V,P|_V), (X,Q))
  \end{equation}
  is a quasi-isomorphism where $V:=U_I.$ 
  This morphism is the morphism associated to a morphism of
  double complexes whose $b$-th row (for $b \in \DZ$) is
  \begin{equation}
    \label{eq:basic-part-row}
    \Hom_{C(\Vb^\subset(X))}((V,P|_V), (Q^b)^\subset)
    \ra
    \Hom_{C(\Vb^\subset(X))}((V,P|_V), (X,Q^b)).
  \end{equation}
  Since the rows of both double complexes involved are uniformly
  bounded,
  the morphism
  \eqref{eq:basic-part}
  is a quasi-isomorphism as soon as we have shown that
  \eqref{eq:basic-part-row} is a quasi-isomorphism for any
  $b \in \DZ$, see \cite[Thm.~1.9.3]{KS}.

  Hence it is enough to show that
  \eqref{eq:basic-part} is a quasi-isomorphism if $Q$ is a vector
  bundle sitting in degree zero.
  The degree zero component
  of the right hand side is 
  \begin{equation*}
    H:=\Hom_{\mathcal{O}_V}(P|_V, Q|_V),
  \end{equation*}
  by Lemma~\ref{l:Hom-u!Pv!Q-v*Qu*P},
  and all other components vanish.
  The graded components of the left-hand side are direct sums
  of objects 
  \begin{equation*}
    \Hom_{\Vb^\subset(X)}((V,P|_V), (U_J,Q|_{U_J})) =
    \begin{cases}
      H=\Hom_{\mathcal{O}_V}(P|_V, Q|_V), & \text{if $V \subset U_J$;}\\
      0, & \text{otherwise.}
    \end{cases}
  \end{equation*}
  for non-empty $J \subset S.$
  By assumption, $M(V):=\{s \in S \mid V \subset U_s\}$ is
  non-empty, and we have $V \subset U_J$ if and only if $J
  \subset M(V).$ 
  Hence the 
  left-hand side of
  \eqref{eq:basic-part} is the chain complex
  \begin{equation*}
    \ldots \ra 
    \prod_{s_0, s_1 \in M(V), \; s_0 < s_1} H \ra
    \prod_{s_0 \in M(V)} H \ra 0 \ra 
    \dots
  \end{equation*}
  of a (non-empty) simplex with coefficients in $H.$ The map
  \eqref{eq:basic-part} is the augmentation map to $H$ which is
  a homotopy equivalence and in particular a quasi-isomorphism.
  This proves that \eqref{shriek-enhancement-iso} is an
  isomorphism if $P$ is bounded.

  Now let $P$ be a bounded above complex of vector bundles.
  For $a \in \DZ$ denote by $P^{\geq a}$ the brutal truncation of
  $P$ which is zero in all degrees $<a$ and coincides with $P$ in
  all other degrees, and let $P^{\geq
    -{\infty}}:=P.$ 
  Then $P$ 
  is the filtered colimit of the filtered 
  diagram
  $0 \hra \dots \hra
  P^{\geq a+1} \hra P^{\geq a} \hra \dots$
  in the category $Z^0(C(\Qcoh(X))).$
  
  Let $\kappa \colon  \mathcal{C}_!(Q) \ra \mathcal{K}$ be a
  quasi-isomorphism 
  with $\mathcal{K}$ an h-injective complex of sheaves.
  Consider for any $a \in \DZ \cup \{-\infty\}$ the commutative
  diagram
  \begin{equation}
    \label{eq:shriek-enhancement-unbounded}
    \xymatrix{
      {\Hom_{[C(\Vb^\subset(X))]}((P^{\geq a})^\subset, Q^\subset)}
      \ar[rd]^-{\can \comp \rho}
      \ar[d]_-{\rho}
      \\
      {\Hom_{[C(\Sh(X))]}(\mathcal{C}_!(P^{\geq
          a}),\mathcal{C}_!(Q))}
      \ar[d]_-{\kappa_*}
      \ar[r]^-{\can}
      & 
      {\Hom_{D(\Sh(X))}(\mathcal{C}_!(P^{\geq a}), \mathcal{C}_!(Q))}
      \ar[d]_-{\kappa_*}^-{\sim} 
      \\
      {\Hom_{[C(\Sh(X))]}(\mathcal{C}_!(P^{\geq a}),\mathcal{K})}
      \ar[r]^-{\can}_-{\sim} &
      {\Hom_{D(\Sh(X))}(\mathcal{C}_!(P^{\geq a}), \mathcal{K})}
    }
  \end{equation}
  whose right vertical arrow $\kappa_*$ and lower horizontal arrow
  $\can$ obviously are isomorphisms; hence the diagonal arrow
  $\can \comp \rho$ is an 
  isomorphism if and only if the vertical composition $\kappa_* \comp
  \rho$ is an 
  isomorphism. We 
  already know that $\can \comp \rho$ is an isomorphism for all $a
  \in 
  \DZ,$ and this is also true if we replace $(P^{\geq a})^{\subset}$ and
  $\mathcal{C}_!(P^{\geq a})$ by their shifts $[m](P^{\geq a})^{\subset}$ and
  $[m]\mathcal{C}_!(P^{\geq a}).$ Hence we know that
  \begin{equation}
    \label{eq:iota-a}
    \kappa_* \comp \rho \colon 
    \Hom_{C(\Vb^\subset(X))}((P^{\geq a})^\subset, Q^\subset)
    \ra
    {\Hom_{C(\Sh(X))}(\mathcal{C}_!(P^{\geq a}),\mathcal{K})}
  \end{equation}
  is a quasi-isomorphism for all $a \in \DZ,$ and we need to
  prove this for $a = -\infty.$
  Both sides of \eqref{eq:iota-a}, for $a
  \in \DZ,$ 
  form inverse systems of dg modules with surjective transitions
  maps. Hence the inverse limit of the quasi-isomorphisms 
  \eqref{eq:iota-a}, for $a \in \DZ,$ is again a
  quasi-isomorphism,
  by Corollary~\ref{c:qiso-surj-inv-system},
  and 
  this inverse
  limit is canonically isomorphic to the morphism 
  \eqref{eq:iota-a} for $a=-\infty.$
  Hence $\can \comp \rho$ in
  \eqref{eq:shriek-enhancement-unbounded} 
  is an isomorphism for $a=-\infty.$ 
  This proves the proposition. 
\end{proof}

\begin{lemma}
  \label{l:hom-u!-qcoh}
  Let $X$ be a scheme and $u\colon U \subset X$ the inclusion of
  an affine open subscheme. Let $P$ be a bounded above complex of
  vector bundles on $U$ and $G \in C(\Qcoh(X))$. Then the
  canonical map is an isomorphism
  \begin{equation*}
    \Hom_{[C(\Sh(X))]}(u_!P, G)
    \sira
    \Hom_{D(\Sh(X))}(u_!P, G).
  \end{equation*}
\end{lemma}

\begin{proof}
  Our map appears as the left vertical map in the 
  commutative diagram
  \begin{equation*}
    \xymatrix{
      {\Hom_{[C(\Sh(X))]}(u_!P, G)} \ar[d] \ar[r]^-{\sim} 
      & 
      {\Hom_{[C(\Sh(U))]}(P, u^*G)} \ar[d] \ar@{}[r]|-{=}
      &
      {\Hom_{[C(\Qcoh(U))]}(P, u^*G)} \ar[d]^-{\sim}
      \\
      {\Hom_{D(\Sh(X))}(u_!P, G)} \ar[r]^-{\sim} 
      & 
      {\Hom_{D(\Sh(U))}(P, u^*G)} 
      &
      {\Hom_{D(\Qcoh(U))}(P, u^*G)} \ar[l]_-{\sim}
    }
  \end{equation*}
  whose left horizontal arrows are isomorphisms because
  $u_!$ is exact and left adjoint to the exact functor $u^*$,
  whose lower right horizontal arrow
  (which is well defined because $u^*G \in C(\Qcoh(U))$)
  is an isomorphism because
  $D(\Qcoh(U)) \sira D_\Qcoh(\Sh(U))$ is an equivalence ($U$
  being quasi-compact and separated),
  and
  whose right vertical arrow 
  is an isomorphism because $U$ is
  affine and hence $P$ can be viewed as a bounded above complex
  of projective
  $\mathcal{O}_U(U)$-modules. 
\end{proof}

\begin{lemma}
  \label{l:surj-inv-system-acyclic}
  Let $A_0 \xla{p_1} A_1 \xla{p_2} A_2 \la \dots$ be a directed
  inverse 
  system of acyclic complexes of abelian groups with all
  transition maps $p_i$ surjective. Then its inverse limit
  $\varprojlim A_i$ is also acyclic.
\end{lemma}

\begin{proof}
  We split each complex $A_i$ into short exact sequences
  $0 \ra K_i^j \ra A_i^j \ra K_i^{j+1} \ra 0$
  where $K_i^j$ is the image of $A_i^{j-1} \ra A_i^j.$
  The maps $p_{i+1}$ induce surjective maps $K_{i+1}^j \ra
  K_i^j.$ 
  For each $j \in \DZ$ we obtain a short
  exact sequence  
  $0 \ra K_i^j \ra A_i^j \ra K_i^{j+1} \ra 0$ of directed inverse
  systems of abelian groups and the direct system $(K_i^j)_{i \in
    \DN}$ is Mittag-Leffler. Hence 
  $0 \ra \varprojlim_i K_i^j \ra \varprojlim_i A_i^j \ra
  \varprojlim_i K_i^{j+1} \ra 0$ is exact.
  This implies that
  $\varprojlim A_i$ is acyclic.
\end{proof}

\begin{corollary}
  \label{c:qiso-surj-inv-system}
  Let $X=(X_i)_{i \in \DN}$ and $Y=(Y_i)_{i \in \DN}$ be directed
  inverse
  systems of complexes of abelian groups with surjective
  transition maps. Assume that $\phi=(\phi_i)_{i \in \DN}: X \ra
  Y$ is a morphism of directed inverse systems such that each
  $\phi_i\colon X_i \ra Y_i$ is a quasi-isomorphism. Then the
  induced morphism $\varprojlim \phi_i \colon \varprojlim X_i \ra
  \varprojlim Y_i$
  on the inverse limits is a quasi-isomorphism.
\end{corollary}

\begin{proof}
  Consider the short exact sequences $Y_i \ra \Cone(\phi_i) \ra
  [1]X_i$ of complexes with the obvious transition
  maps which are all surjective. By assumption each
  $\Cone(\phi_i)$ is acyclic, so  
  $\varprojlim \Cone(\phi_i) =
  \Cone(\varprojlim \phi_i)$ is acyclic, by
  Lemma~\ref{l:surj-inv-system-acyclic}. 
  Now take the long exact sequence associated to the 
  short exact sequence 
  $\varprojlim Y_i \ra \Cone(\varprojlim \phi_i) \ra
  [1]\varprojlim X_i$ of complexes.
\end{proof}

\begin{proposition}
  \label{p:abstract-cech-!-object-enhancement}
  Let $X$ be a \ref{enum:GSP}-scheme with an ordered finite
  affine open covering $\mathcal{U}=(U_s)_{s \in S}.$ Then
  $\Cech_!(X)$ is a pretriangulated dg category and the functor
  \begin{equation*}
    \mathcal{C}_! \colon [\Cech_!(X)] \ra \mfPerf'(X)
  \end{equation*}
  induced by \eqref{eq:!-cech-realization}
  is an equivalence of triangulated categories.
  Hence the dg category $\Cech_!(X)$
  is naturally an enhancement of $\mfPerf'(X).$ 
  We call it the 
  \define{$!$-\v{C}ech enhancement}.
\end{proposition}

\begin{proof}
  Certainly $\mathcal{C}_!$ is
  well-defined.
  By condition~\ref{enum:GSP}, any object of $\mfPerf'(X)$ is
  isomorphic to a bounded complex $R$ of vector bundles and hence
  to $\mathcal{C}_!(R).$ This shows that $\mathcal{C}_!$ is
  essentially surjective.
  The full and faithful dg
  functor $\Cech_!(X) \ra C(\Vb^\subset(X)),$ $P \mapsto
  P^\subset,$ and 
  Proposition~\ref{p:abstract-cech-!-bounded-above-full-faithful}
  then show that 
  $\mathcal{C}_! \colon [\Cech_!(X)] \ra \mfPerf'(X)$
  is an equivalence of categories.
  We already observed that $\Cech_!(X)$ contains all shifts of
  its objects, but 
  we need to prove that $\Cech_!(X)$ is pretriangulated.
  The functor \eqref{eq:!-cech-realization}
  extends to a dg functor from the pretriangulated envelope 
  $\Cech_!(X)^\pretr$ of
  $\Cech_!(X)$ to $C(\Sh(X))$ which obviously induces an
  equivalence 
  $\mathcal{C}_! \colon [\Cech_!(X)^\pretr] \ra \mfPerf'(X).$
  Hence $[\Cech_!(X)] \sira [\Cech_!(X)^\pretr]$ and
  $\Cech_!(X)$ is pretriangulated.
\end{proof}


\begin{proposition}
  \label{p:abstract-cech-*-object-enhancement}
  Let $X$ be a \ref{enum:GSP}-scheme with
  an 
  ordered finite affine open covering 
  $\mathcal{U}=(U_s)_{s \in S}.$
  Then the functor
  \begin{equation*}
    \mathcal{C}_* \colon [\Cech_*(X)] \ra \mfPerf(X)
  \end{equation*}
  is an equivalence of triangulated categories.
  Hence the dg category $\Cech_*(X)$
  is naturally an enhancement of $\mfPerf(X).$ 
  We call it the \define{$*$-\v{C}ech enhancement}.
\end{proposition}

\begin{proof}
  This is just a variation of the proofs of
  Propositions~\ref{p:abstract-cech-!-bounded-above-full-faithful}
  and \ref{p:abstract-cech-!-object-enhancement}, cf.\ also 
  the proof of Proposition~\ref{p:cech-*-object-enhancement}.
\end{proof}

The following Proposition~\ref{p:duality-cech-enhancements}
is not used in this article.

\begin{proposition}
  \label{p:duality-cech-enhancements}
  Let $X$ be a Noetherian separated scheme with
  an 
  ordered finite affine open covering 
  $\mathcal{U}=(U_s)_{s \in S}.$
  Then the diagram 
  \begin{equation*}
    \xymatrix{
      {[\Cech_!(X)]^\opp} \ar[d]_-{(-)^\cek}^-{\sim}
      \ar[r]^-{\mathcal{C}_!} 
      & {D(\Sh(X))^\opp} \ar[d]^-{\bR\sheafHom(-,\mathcal{O}_X)}\\
      {[\Cech_*(X)]} \ar[r]^-{\mathcal{C}_*}
      & {D(\Sh(X))}
    }
  \end{equation*}
  of categories
  induced from diagram~\eqref{eq:duality-cech-enhancements}
  commutes up to a natural
  isomorphism. 
  (It is a diagram of triangulated categories if $X$
  satisfies condition~\ref{enum:GSP}.)
\end{proposition}

\begin{proof}
  Let $u:U \ra X$ be the inclusion of an affine open subscheme
  and $P$ a vector bundle on $X.$ Then the functor
  $\sheafHom(\leftidx{^U}{P}{}, -) \cong u_*\sheafHom(u^*(P),
  u^*(-))\colon \Qcoh(X) \ra \Qcoh(X)$ is (well-defined by
  Lemma~\ref{l:qcqs-preserves-qcoh}) and exact since $u$ is
  affine.  Hence, if $\mathcal{O}_X \ra \mathcal{I}$ is a
  resolution by injective quasi-coherent sheaves, we obtain a
  quasi-isomorphism $\sheafHom(\leftidx{^U}{P}{}, \mathcal{O}_X)
  \ra \sheafHom(\leftidx{^U}{P}{}, \mathcal{I}).$

  By
  Theorem~\ref{t:injective-in-qcoh-vs-all-OX}.\ref{enum:inj-qcoh-equal-inj-OX-that-qcoh},
  $\mathcal{I}$ consists of injective sheaves and hence, being
  bounded below, is an h-injective complex of sheaves
  (Remark~\ref{rem:fibrant-explained}).  Hence
  $\sheafHom(\leftidx{^U}{P}{}, \mathcal{I}) \cong
  \bR\sheafHom(\leftidx{^U}{P}{}, \mathcal{O}_X)$ in $D(\Sh(X)).$

  These facts imply that the canonical morphism
  $\sheafHom(\mathcal{C}_!(Q), \mathcal{O}_X) \ra 
  \bR\sheafHom(\mathcal{C}_!(Q), \mathcal{O}_X)$ is an
  isomorphism for any bounded complex $Q$ of vector bundles on
  $X.$ 
  Now use the commutativity of
  diagram~\eqref{eq:duality-cech-enhancements}.
\end{proof}

Let $\Cech^{-,b}_!(X)$ 
be the full dg subcategory of $\Cech^-_!(X)$ of those objects $P$
whose
cohomology is bounded, i.\,e.\ $H^i(P) =0$ for all but finitely
many $i \in \DZ.$
 
\begin{proposition}
  \label{p:shriek-cech-enhancement-D-minus-and-b-Coh}
  Let $X$ be a \ref{enum:RES}-scheme with an ordered finite
  affine open covering $\mathcal{U}=(U_s)_{s \in S}.$ Then the dg
  categories $\Cech^-_!(X)$ and $\Cech^{-,b}_!(X)$ are
  pretriangulated and the functors
  \begin{align*}
    \mathcal{C}_! \colon [\Cech^-_!(X)] & \ra
    D^-_\Coh(\Sh(X)) \quad \text{and}\\
    \mathcal{C}_! \colon [\Cech^{-,b}_!(X)] & \ra D^b_\Coh(\Sh(X))
  \end{align*}
  are equivalences 
  of triangulated categories.
  Hence the dg category $\Cech^-_!(X)$ (resp.\ $\Cech^{-,b}_!(X)$)
  is naturally an enhancement of
  $D^-(\Coh(X)) \sira D^-_\Coh(\Qcoh(X)) \sira D^-_\Coh(\Sh(X))$
  (resp.\ of 
  $D^b(\Coh(X)) \sira D^b_\Coh(\Qcoh(X)) \sira
  D^b_\Coh(\Sh(X))$),
  cf.\ equivalence \eqref{eq:D-minus-coh-noetherian}.
  We call it the \define{$!$-\v{C}ech enhancement}. 
\end{proposition}

\begin{proof}
  If $P$ is a bounded above complex of vector bundles on $X$, all
  its cohomology sheaves $H^i(P)$ are coherent since $X$ is
  Noetherian.
  Hence both functors are well-defined. They are essentially
  surjective by 
  Proposition~\ref{p:D-minus-Coh(Sh)-strict-coherent}.
  Proposition~\ref{p:abstract-cech-!-bounded-above-full-faithful}
  then shows that both functors are equivalences of categories.
  As 
  in the proof of Proposition~\ref{p:abstract-cech-!-object-enhancement}
  one shows that the dg categories $\Cech^-_!(X)$ and
  $\Cech^{-,b}_!(X)$ are pretriangulated.
\end{proof}


\subsubsection{Pullbacks}
\label{sec:pullbacks}

We discuss some properties of the categories introduced in 
section~\ref{sec:some-auxil-constr}.
These results will be used in 
section~\ref{sec:four-mukai-funct}. 

Let $f \colon Y \ra X$ be a morphism of schemes.  
Then there are obvious functors
\begin{align*}
  f^* \colon \Vb^\subset(X) \ra \Vb^\subset(Y), & \quad
  (U,P) \mapsto (f\inv(U), f^*P), \text{ and}\\
  f^* \colon \Vb_\supset(X) \ra \Vb_\supset(Y), & \quad
  (U,P) \mapsto (f\inv(U), f^*P).
\end{align*}
Note that their definition on morphism spaces uses
both identifications from Lemma~\ref{l:base-change-ringed-spaces-open}, or Remark~\ref{rem:alternative-def-Vb}.
Both these pullback functors commute with the realizations
functors, 
i.\,e.\ we have diagrams
\begin{equation}
  \label{eq:upper*-realization}
  \xymatrix{
    {\Vb^\subset(X)} \ar[d]^-{f^*} \ar[rr]^-{(U,P)
      \mapsto u_!P}_-{\eqref{eq:real-subset}}  
    && {\Sh(X)} \ar[d]^-{f^*}\\
    {\Vb^\subset(Y)} \ar[rr]^-{(U',P') \mapsto
      u'_!P'}_-{\eqref{eq:real-subset}}  
    && {\Sh(Y),}
  }
  \quad\quad
  \xymatrix{
    {\Vb_\supset(X)} \ar[d]^-{f^*} \ar[rr]^-{(U,P)
      \mapsto u_*P}_-{\eqref{eq:real-supset}}  
    && {\Sh(X)} \ar[d]^-{f^*}\\
    {\Vb_\supset(Y)} \ar[rr]^-{(U',P') \mapsto
      u'_*P'}_-{\eqref{eq:real-supset}}  
    && {\Sh(Y)}
  }
\end{equation}
that commute up to natural isomorphisms coming from
Lemma~\ref{l:base-change-ringed-spaces-open}.
Pullback also commutes with the duality, i.\,e.\ there is a
diagram 
\begin{equation}
  \label{eq:Vb-sup-sub-duality-pull}
  \xymatrix{
    {\Vb^\subset(X)^\opp} \ar[r]^-{(-)^\cek}_-{\sim} \ar[d]^-{f^*}
    & 
    {\Vb_\supset(X)} \ar[d]^-{f^*}
    \\
    {\Vb^\subset(Y)^\opp} \ar[r]^-{(-)^\cek}_-{\sim} 
    & 
    {\Vb_\supset(Y)} 
  }
\end{equation}
which commutes up to a natural isomorphism, as follows from
Lemma~\ref{l:pullback-sheafHom}.

\begin{lemma}
  \label{l:pullback-sheafHom}
  Let $f\colon Y \ra X$ be a morphism of schemes. Given
  $\mathcal{F},$ $\mathcal{G} \in \Sh(X)$, there is a natural
  morphism 
  $f^*\sheafHom(\mathcal{F}, \mathcal{G}) \ra
  \sheafHom(f^*\mathcal{F}, f^*\mathcal{G})$
  which is an isomorphism if $\mathcal{F}$ is a vector bundle.
\end{lemma}

\begin{proof}
  Obvious.
\end{proof}

The pullback functors, the realization functors and the dualities
combine to a cube of functors with five commutative faces whose 
sixth face is given by 
\begin{equation*}
  \xymatrix{
    {\Sh(X)^\opp} 
    \ar[rrr]^-{(-)^\cek=\sheafHom(-, \mathcal{O}_X)}
    \ar[d]^-{f^*} 
    &&& 
    {\Sh(X)} \ar[d]^-{f^*}
    \\
    {\Sh(Y)^\opp}
    \ar[rrr]^-{(-)^\cek=\sheafHom(-, \mathcal{O}_Y)}
    &&& 
    {\Sh(Y)} 
  }
\end{equation*}
which comes with a morphism of functors $f^*\sheafHom(-,
\mathcal{O}_X) \ra \sheafHom(f^*(-), \mathcal{O}_Y)$ coming from
Lemma~\ref{l:pullback-sheafHom} which is an isomorphism on all
objects of the form $u_!P$ where $u\colon U \subset X$ is an
affine inclusion of an open subscheme and $P$ is a vector bundle
on $U$ (use Lemmata~\ref{l:pullback-sheafHom},
\ref{l:base-change-ringed-spaces-open}.\ref{enum:proper-base-change-open},
\ref{l:push-affine-then-pull} and the obvious adjunctions).


\subsubsection{}
\label{sec:morphism-functors}


Let $M$ be a sheaf on a scheme $X$.
Let $(U, P) \in \Vb^\subset(X)$ and let $u\colon U \subset X$ be
the inclusion. There are morphisms
\begin{equation}
  \label{eq:compatible}
  (u_*P^\cek) \otimes M 
  \ra 
  u_*(P^\cek \otimes u^*M) 
  \sira 
  u_*\sheafHom(P, u^*M)
  \sira
  \sheafHom(u_!P, M)
\end{equation}
constructed as follows: the first morphism corresponds to the
obvious morphism $u^*((u_*P^\cek) \otimes M) \sira (u^*u_*P^\cek)
\otimes u^*M \ra P^\cek \otimes u^*M$ under the adjunction
$(u^*,u_*)$ and is an isomorphism if $M$ is quasi-coherent and
$u$ is affine, by Lemma~\ref{l:push-from-open-and-tensor}; the
second morphism is the isomorphism coming from the obvious
isomorphism $P^\cek \otimes u^*M \sira \sheafHom(P, u^*M)$, and
the third morphism is the adjunction isomorphism.

\begin{lemma}
  \label{l:compatible}
  Let $M$ be a quasi-coherent sheaf on a scheme $X$.
  Then mapping an object
  $(U,P)$ of $\Vb^\subset(X)$ to the morphism
  \eqref{eq:compatible} defines a morphism
  $\tau$ from the functor
  \begin{equation*}
    \Vb^\subset(X)^\opp
    \xsira{(-)^\cek}
    \Vb_\supset(X)
    \xrightarrow[\eqref{eq:real-supset}]{(U,P) \mapsto u_*P}
    \Sh(X) 
    \xra{(-\otimes M)}
    \Sh(X)
  \end{equation*}
  to the functor
  \begin{equation*}
    \Vb^\subset(X)^\opp
    \xrightarrow[\eqref{eq:real-subset}]{(U,P) \mapsto u_!P}
    \Sh(X)^\opp 
    \xra{\sheafHom(-, M)}
    \Sh(X).
  \end{equation*}
  Moreover, if 
  $(U,P)$ is an object of $\Vb^\subset(X)$ such that the
  inclusion $U \subset X$ is affine, then
  $\tau_{(U,P)}$ is an isomorphism.
\end{lemma}

\begin{proof}
  We need to show that any morphism $\alpha \colon (U,P) \ra
  (V,Q)$ in $\Vb^\subset(X)$ gives rise to a commutative square
  in $\Sh(X)$.
  We can assume that $j\colon U \subset V$. Then our morphism is
  given by $\alpha\colon j_!P \ra Q$, and 
  $\alpha^\cek \colon Q^\cek \ra (j_!P)^\cek=\sheafHom(j_!P,
  \mathcal{O}_V)=j_*\sheafHom(P, \mathcal{O}_U)=j_*(P^\cek)$.
  Let $u \colon U \subset X$ and
  $v\colon V \subset X$.
  The morphisms 
  $\tau_{(V,Q)}$ resp.\ $\tau_{(U,P)}$ appear as the 
  upper resp.\ lower row in the 
  commutative diagram 
  \begin{equation*}
    \xymatrix{
      {(v_*Q^\cek) \otimes M} 
      \ar[r]
      \ar[d]^-{v_*(\alpha^\cek)\otimes \id_M}
      &
      {v_*(Q^\cek \otimes v^*M)}
      \ar[r]^-{\sim}
      \ar[d]^-{v_*(\alpha^\cek \otimes \id_{v^*M})}
      &
      {v_*\sheafHom(Q, v^*M)}
      \ar[r]^-{\sim}
      \ar[d]^-{v_*(\alpha^*)}
      &
      {\sheafHom(v_!Q, M)}
      \ar[d]^-{(v_!\alpha)^*}
      \\
      {(v_*((j_!P)^\cek)) \otimes M} 
      \ar[r]
      \ar[d]^{\sim}
      &
      {v_*((j_!P)^\cek \otimes v^*M)}
      \ar[r]
      \ar[d]
      &
      {v_*\sheafHom(j_!P, v^*M)}
      \ar[r]^-{\sim}
      \ar[d]^{\sim}
      &
      {\sheafHom(v_!j_!P, M)}
      \ar[d]^{\sim}
      \\
      {(u_*P^\cek) \otimes M} 
      \ar[r]
      &
      {u_*(P^\cek \otimes u^*M)}
      \ar[r]^-{\sim}
      &
      {u_*\sheafHom(P, u^*M)}
      \ar[r]^-{\sim}
      &
      {\sheafHom(u_!P, M)}
    }
  \end{equation*}
  whose non-labeled morphisms are the
  obvious ones. The outer square is the one we need.
  The last statement is clear from above.
\end{proof}

\begin{corollary}
  \label{c:compatible}
  Let $X$ and $Y$ be schemes over a field $k,$ let $p \colon Y
  \times X \ra X$ be the second projection, and let
  $M$ be a quasi-coherent sheaf on $Y \times X.$ 
  Then there is a morphism $\tau'$ from the functor
  \begin{equation*}
    \Vb^\subset(X)^\opp
    \xsira{(-)^\cek}
    \Vb_\supset(X)
    \xrightarrow[\eqref{eq:real-supset}]{(U,P) \mapsto u_*P}
    \Sh(X)
    \xra{p^*}
    \Sh(Y \times X) 
    \xra{(-\otimes M)}
    \Sh(Y \times X)
  \end{equation*}
  to the functor
  \begin{equation*}
    \Vb^\subset(X)^\opp
    \xrightarrow[\eqref{eq:real-subset}]{(U,P) \mapsto u_!P}
    \Sh(X)^\opp 
    \xra{p^*}
    \Sh(Y \times X)^\opp 
    \xra{\sheafHom(-, M)}
    \Sh(Y \times X).
  \end{equation*}
  Moreover, if 
  $(U,P)$ is an object of $\Vb^\subset(X)$ such that the
  inclusion $U \subset X$ is affine, then
  $\tau'_{(U,P)}$ is an isomorphism.
\end{corollary}

\begin{proof}
  The commutative diagrams
  \eqref{eq:upper*-realization}
  and \eqref{eq:Vb-sup-sub-duality-pull}
  show that the first functor is isomorphic to the composition
  \begin{equation*}
    \Vb^\subset(X)^\opp
    \xra{p^*}
    \Vb^\subset(Y \times X)
    \xsira{(-)^\cek}
    \Vb_\supset(Y \times X)
    \xrightarrow[\eqref{eq:real-supset}]{(U,P) \mapsto u_*P}
    \Sh(Y \times X)
    \xra{(-\otimes M)}
    \Sh(Y \times X)
  \end{equation*}
  and that the second functor is isomorphic to the composition
  \begin{equation*}
    \Vb^\subset(X)^\opp
    \xra{p^*}
    \Vb^\subset(Y \times X)^\opp 
    \xrightarrow[\eqref{eq:real-subset}]{(U,P) \mapsto u_!P}
    \Sh(Y \times X)^\opp 
    \xra{\sheafHom(-, M)}
    \Sh(Y \times X).
  \end{equation*}
  Hence we can use Lemma~\ref{l:compatible}.
\end{proof}

\subsection{All enhancements are equivalent}
\label{sec:all-enhancements-are}

\begin{remark}
  \label{rem:enhancements-equiv}
  All enhancements mentioned above are equivalent when defined,
  i.\,e.\ the corresponding dg categories are quasi-equivalent.
  In the non-obvious cases this can be proved using the method of
  \cite[Prop.~2.50]{valery-olaf-matfak-semi-orth-decomp}.  
\end{remark}

\begin{remark}
  \label{rem:uniqueness-enhancements}
  In many cases all enhancements of $\mfPerf(X)$ (resp.\
  $D^b(\Coh(X))$) are quasi-equivalent, for example for $X$ a
  quasi-projective scheme over a field $k,$ by \cite[Thm.~2.12,
  Thm.~2.13]{lunts-orlov-enhancement}.
  Paolo Stellari informed the second author that he and Alberto
  Canonaco can prove uniqueness of enhancements of $\mfPerf(X)$
  and $D^b(\Coh(X))$ for a Noetherian semi-separated scheme $X$
  having the resolution property; in the meantime, their preprint
  has appeared, see
  \cite{canonaco-stellari-uniqueness-of-dg-enhancements}.  
\end{remark}

If we work on a scheme over a
ring $R,$ all the above
constructions and results have obvious $R$-linear analogs,
e.\,g.\ all 
enhancements discussed above are then dg $R$-categories in the
obvious way.

There are other enhancements one could consider, for example
enhancements using Drinfeld dg quotient categories or ``morphism
oriented \v{C}ech enhancements'', see
\cite{valery-olaf-matfak-semi-orth-decomp}.
We do not consider
these two types of enhancements in this article because they seem
to be badly behaved
with respect to products of schemes. The \v{C}ech
enhancements from section~\ref{sec:vcech-enhanc} were found
starting from the ``object oriented $*$-\v{C}ech enhancements''
discussed in appendix~\ref{sec:vcech-enhanc-loc-integral} 
(and used in \cite{valery-olaf-matrix-factorizations-and-motivic-measures})
which
are based on
\cite[Lemma~6.7]{bondal-larsen-lunts-grothendieck-ring}.

\section{Smoothness of categories and schemes}
\label{sec:smoothn-categ-schem}

Let $k$ be a field.
Recall that a dg $k$-category $\mathcal{A}$ is smooth over $k$ if 
$\mathcal{A}$ is a perfect dg $\mathcal{A} \otimes
\mathcal{A}^\opp$-module, cf.\
\cite[Def.~3.7, Rem.~3.9]{valery-olaf-smoothness-equivariant},
and that 
smoothness is invariant under quasi-equivalences and even under
Morita equivalences (see \cite[Lemma~2.30]{valery-olaf-matrix-factorizations-and-motivic-measures}).

\begin{definition}
  \label{d:Perf-DbCoh-smooth}
  Let $X$ be a quasi-compact separated scheme over a field $k$.
  We say 
  that the triangulated 
  category $\mfPerf(X) \cong \mfPerf'(X)$ is \textbf{smooth over
    $k$} if the dg 
  $k$-category $\mfPerf^\hinj(X)$ is smooth over $k$.
  Similarly, if $X$ is a Noetherian separated scheme over a field
  $k$, we say that the 
  triangulated category $D^b(\Coh(X))$ is \textbf{smooth over
    $k$} if the dg $k$-category $C_\Coh^{\hinj,b}(\Qcoh(X))$ is
  smooth over $k$.
\end{definition}

\begin{remark}
  \label{rem:Perf-DbCoh-smooth}
  In the above definition we could have chosen any
  of the 
  equivalent enhancements from section~\ref{sec:enhancements},
  cf.\ Remark~\ref{rem:enhancements-equiv}. 
  If a uniqueness result for enhancements 
  of $\mfPerf(X)$ (resp.\ $D^b(\Coh(X))$)  
  is known
  (cf.\ Remark~\ref{rem:uniqueness-enhancements})
  one can test
  $k$-smoothness on any enhancement.
\end{remark}


\begin{theorem}  [{Homological versus geometric smoothness}]
  \label{t:mfPerf-abstract-Cechobj-smooth-vs-diagonal-sheaf-perfect}
  Let $X$ be a Noetherian \ref{enum:GSP}-scheme over a field $k$
  and assume that $X \times X$ is also Noetherian.
  Let $\Delta \colon  X \ra X \times X$ be the diagonal
  (closed) immersion.
  Then the following two conditions are equivalent:
  \begin{enumerate}
  \item
    \label{enum:mfPerf-smooth}
    $\mfPerf(X)$ is smooth over $k;$
  \item 
    \label{enum:diagonal-cpt}
    $\Delta_*(\mathcal{O}_X) \in
  \mfPerf(X \times X).$
  \end{enumerate}
  If $X$ is in addition of finite type
  over $k,$  
  they
  are also equivalent to:
  \begin{enumerate}[resume]
  \item 
    \label{enum:geometrically-smooth}
    $X$ is smooth over $k.$
  \end{enumerate}
  In particular, if 
  $X$ is a separated scheme of finite type over $k$ 
  having
  the resolution
  property (for example if $X$ is quasi-projective over $k$),
  then the above three conditions are equivalent.
\end{theorem}

\begin{corollary}
  \label{c:smooth-quasi-projective}
  Let $X$ be 
  a smooth quasi-compact separated scheme 
  over a field $k.$ Then $\mfPerf(X) =D^b(\Coh(X))$ is
  smooth over $k.$ 
\end{corollary}

\begin{proof}[{Proof of
    Corollary~\ref{c:smooth-quasi-projective}}]  
  Smoothness over $k$ and quasi-compactness imply that 
  $X$ is of finite type over $k$ and hence of finite
  dimension, so  
  Proposition~\ref{p:regular-vs-singularity-cat}
  shows that $\mfPerf(X)=D^b(\Coh(X)).$
  Any regular Noetherian separated scheme 
  has the resolution property, by a theorem of Kleiman
  \cite[Ex.~III.6.8]{Hart}, so that we can apply
  Theorem~\ref{t:mfPerf-abstract-Cechobj-smooth-vs-diagonal-sheaf-perfect}.
\end{proof}


\begin{theorem} 
  \label{t:D-b-Coh-smooth}
  Let $X$ be a separated scheme of finite type over a perfect
  field $k$ that has the resolution property.
  Then $D^b(\Coh(X))$ is smooth over
  $k.$
\end{theorem}

The rest of this section is devoted to the proof of these two
theorems.

\begin{remark}
  \label{rem:tautology-for-Spec-R}
  The first two conditions of 
  Theorem~\ref{t:mfPerf-abstract-Cechobj-smooth-vs-diagonal-sheaf-perfect}
  are tautologically equivalent for any affine scheme over a
  field. 
  Namely, assume that $U=\Spec R$ is an affine scheme over a
  field $k.$ 
  Then Example~\ref{exam:perfect-on-affine} shows that
  $\mfPerf(U)=\per(R)$ and $\mfPerf(U \times U) = \per(R \otimes
  R).$ The dg
  category $C^b(\proj(R))$ of bounded complexes of finitely
  generated projective  
  $R$-modules is an enhancement of 
  $\mfPerf(U)=\per(R),$ in fact it is the enhancement
  $\Cech_*(U)=\Cech_!(U)$ for the trivial open covering of $U.$  
  Viewing $R$ as a dg category, the obvious inclusion $R \ra
  C^b(\proj(R))$ is a Morita equivalence.
  This implies that $\mfPerf(U)$ is $k$-smooth if and only if
  $R$ is $k$-smooth
  as a dg algebra, i.\,e.\  
  $R \in
  \per(R \otimes R),$
  if and only if $\Delta_*(\mathcal{O}_U) \in
  \mfPerf(U \times U).$
\end{remark}








\subsection{K\"unneth formula and some consequences}
\label{sec:kunneth-formula-some}

We refer to Appendix~\ref{sec:extern-tens-prod} for the
definition of the bifunctor $\boxtimes.$ This bifunctor is exact
(Lemma~\ref{l:boxtimes-exact}) and computed on all complexes
of 
sheaves, in particular on objects of the derived categories, in
the naive way (Remark~\ref{rem:boxtimes-exact}).
 
\begin{proposition}
  [{K\"unneth formula, cf.\
    \cite[Thm.~14]{kempf-elementary-proofs}}] 
  \label{p:kuenneth-formula}
  Let $X$ and $Y$ be quasi-compact separated schemes over a field
  $k.$
  Let $I$ (resp.\ $J$) be a 
  complex of
  $\Gamma(X,-)$-acyclic (resp.\ $\Gamma(Y,-)$-acyclic)
  quasi-coherent sheaves on $X$ (resp.\ $Y$)
  and let $\sigma\colon I \boxtimes J \ra L$ be a
  quasi-isomorphism 
  where $L$ is a 
  complex of $\Gamma(X \times
  Y,-)$-acyclic 
  quasi-coherent sheaves (the global section functors are
  considered as functors 
  between 
  categories of quasi-coherent sheaves).
  Then the composition
  \begin{equation*}
    \Gamma(X,I) \otimes \Gamma(Y,J)
    \xra{\boxtimes}
    \Gamma(X \times Y, I \boxtimes J)
    \xra{\Gamma(\sigma)}
    \Gamma(X \times Y, L)
  \end{equation*}
  is a quasi-isomorphism of dg modules.
\end{proposition}

From the following proof of this proposition one easily deduces 
an isomorphism 
\begin{equation*}
  \bR \Gamma(X, -) \otimes \bR \Gamma(Y, ?) 
  \sira 
  \bR \Gamma(X \times Y, - \boxtimes ?)
\end{equation*}
of functors $D(\Qcoh(X)) \times D(\Qcoh(Y)) \ra
D(\Qcoh(\Spec k)).$ 

\begin{proof}
  Let $\mathcal{U}$ and $\mathcal{V}$ 
  be ordered finite affine open coverings of $X$ and $Y,$
  respectively.
  Consider the 
  $*$-\v{C}ech resolutions
  $i\colon I \ra
  \mathcal{C}_*(I)$
  and $j\colon J \ra
  \mathcal{C}_*(J).$
  By
  Lemma~\ref{l:u'-ls-qls-acyclic},
  the complex $\mathcal{C}_*(I)$ (resp.\ $\mathcal{C}_*(J)$)
  consists of 
  $\Gamma(X,-)$-acyclic (resp.\ $\Gamma(Y,-)$-acyclic)
  quasi-coherent sheaves.
  Similarly, using Lemma~\ref{l:boxtimes-restriction-extension}.\ref{enum:boxtimes-*-extension},
  $i \boxtimes j \colon I \boxtimes J \ra 
  \mathcal{C}_*(I) \boxtimes
  \mathcal{C}_*(J)$
  is
  a quasi-isomorphism to a
  componentwise 
  $\Gamma(X \times Y, -)$-acyclic complex,
  and
  there is a canonical
  isomorphism (essentially an equality)
  \begin{equation*}
    \gamma \colon  
    \Gamma(X, \mathcal{C}_*(I))
    \otimes \Gamma(Y, \mathcal{C}_*(J))
    \sira
    \Gamma(X \times Y, 
    \mathcal{C}_*(I) \boxtimes
    \mathcal{C}_*(J)).
  \end{equation*}
  Let $\lambda\colon L \ra L'$ be a quasi-isomorphism with $L'$ 
  a fibrant
  complex of 
  quasi-coherent sheaves.
  Let $\tau \colon \mathcal{C}_*(I) \boxtimes \mathcal{C}_*(J)
  \ra L'$ be a quasi-isomorphism such that $\tau \comp
  (i \boxtimes j)= \lambda \comp \sigma$ in the homotopy category
  $[C(\Qcoh(X \times Y))]$ (we could even assume that this holds
  in $Z^0(C(\Qcoh(X \times Y)))$ since $i \boxtimes j$ is a
  trivial cofibration and $L'$ is fibrant).  
  Consider the diagram
  \begin{equation*}
    \xymatrix{
      {\Gamma(X, \mathcal{C}_*(I))
        \otimes \Gamma(Y, \mathcal{C}_*(J))}
      \ar[r]^-{\sim}_-{\gamma} &
      {\Gamma(X \times Y, 
        \mathcal{C}_*(I) \boxtimes
        \mathcal{C}_*(J))} \ar[r]^-{\Gamma(\tau)}
      & 
      {\Gamma(X \times Y, L')} 
      \\
      {\Gamma(X,I) \otimes \Gamma(Y,J)}
      \ar[r]^-{\boxtimes}
      \ar[u]^-{\Gamma(i) \otimes \Gamma(j)}
      &
      {\Gamma(X \times Y, I \boxtimes J)}
      \ar[r]^-{\Gamma(\sigma)}
      \ar[u]^-{\Gamma(i \boxtimes j)} 
      &
      {\Gamma(X \times Y, L)} \ar[u]^-{\Gamma(\lambda)}
    }
  \end{equation*}
  which is commutative in the homotopy category $[C(\Qcoh(\Spec
  k))].$ 

  Since the functor $\Gamma(X,-): \Qcoh(X) \ra \Qcoh(\Spec
  k)$ is of finite cohomological dimension 
  (by Lemma~\ref{l:fin-coho-dim})
  the map $\Gamma(i)$ is a
  quasi-isomorphism
  (by
  \cite[Lemma~12.4.(b)]{wolfgang-olaf-locallyproper}). 
  Similarly,  
  $\Gamma(j)$ is a 
  quasi-isomorphism, and $\Gamma(i) \otimes \Gamma(j)$ is a
  quasi-isomorphism because $k$ is a field.
  Similarly, since both $\lambda$ and $\tau$ are
  quasi-isomorphisms between 
  complexes
  of $\Gamma(X \times Y,-)$-acyclic quasi-coherent
  sheaves
  (use 
  Remark~\ref{rem:fibrant-explained}),
 $\Gamma(\lambda)$ and
  $\Gamma(\tau)$ are  
  quasi-isomorphisms.
  This implies the proposition.
\end{proof}

\begin{proposition}
  \label{p:tensor-product-of-endos-vs-endos-of-boxproduct}
  Let $X$ and $Y$ be Noetherian separated schemes over a field
  $k$ such 
  that $X \times Y$ is also
  Noetherian.  
  Let $E \in D^-_\Coh(\Sh(X))$ and $F \in D^-_\Coh(\Sh(Y))$ be
  objects that are isomorphic to bounded above complexes of
  vector 
  bundles.
  Let $I$ and $J$ be bounded below complexes of injective
  quasi-coherent sheaves on $X$ and $Y,$ respectively.
  Let $\tau\colon I \boxtimes J \ra T$ be a quasi-isomorphism
  with $T$ a 
  bounded below complex of injective quasi-coherent sheaves.
  Then the composition
  \begin{equation*}
    {\leftidx{_{C(\Sh(X))}}{(E, I)}{} 
      \otimes 
      \leftidx{_{C(\Sh(Y))}}{(F,J)}{}}
    \xra{\boxtimes}
    {\leftidx{_{C(\Sh(X \times Y))}}{(E \boxtimes F, I \boxtimes J)}{}}
    \xra{\tau_*}
    {\leftidx{_{C(\Sh(X \times Y))}}{(E \boxtimes F, T)}{}}
  \end{equation*}
  is a quasi-isomorphisms of dg modules. Here
  we abbreviate
  $\leftidx{_?}{(-,-)}{}=\Hom_?(-,-).$
\end{proposition}

\begin{proof}
  Since an injective quasi-coherent sheaf on a Noetherian scheme
  is also an injective sheaf, 
  by Theorem~\ref{t:injective-in-qcoh-vs-all-OX}.\ref{enum:inj-qcoh-equal-inj-OX-that-qcoh}, 
  we 
  can assume
  that $E$ and $F$ are bounded above complexes of
  vector bundles.
  
  Then $\sheafHom(E,I),$ $\sheafHom(F,J),$ and
  $\sheafHom(E 
  \boxtimes F, T)$ are bounded below complexes of injective
  quasi-coherent sheaves. In particular their components are
  $\Gamma$-acyclic.
  We will prove below that 
  the composition
  \begin{equation}
    \label{eq:boxtimes-sheafhom-noch-neuer}
    \tau_* \comp \can \colon
    {\sheafHom(E,I) \boxtimes \sheafHom(F,J)} 
    \xra{\can}
    {\sheafHom(E \boxtimes F, I \boxtimes J)}
    \xra{\tau_*}
    {\sheafHom(E \boxtimes F, T)}
  \end{equation}
  is a quasi-isomorphism.
  Assuming this for a moment, Proposition~\ref{p:kuenneth-formula}
  shows that the composition 
  $\Gamma(\tau_* \comp \can) \comp \boxtimes$ in the
  commutative diagram
  \begin{equation*}
    \xymatrix{
      {\Gamma(\sheafHom(E,I)) \otimes \Gamma(\sheafHom(F,J))}
      \ar[r]^-{\boxtimes} 
      \ar[d]_-{\boxtimes} 
      & 
      {\Gamma(\sheafHom(E,I) \boxtimes \sheafHom(F,J))} 
      \ar[dl]^-{\Gamma(\can)}
      \ar[d]^-{\Gamma(\tau_* \comp \can)}
      \\
      {\Gamma(\sheafHom(E \boxtimes F, I \boxtimes J))}
      \ar[r]_-{\Gamma(\tau_*)}
      &
      {\Gamma(\sheafHom(E \boxtimes F, T))}
    }
  \end{equation*}
  is a quasi-isomorphism, and this implies the proposition.

  Now let us prove that \eqref{eq:boxtimes-sheafhom-noch-neuer}
  is a quasi-isomorphism. 
  The claim is local on $X$ and $Y$,
  by  
  Theorem~\ref{t:injective-in-qcoh-vs-all-OX}.\ref{enum:inj-qcoh-restrict-to-inj-qcoh}.
  Hence we can assume that
  $X=\Spec A$
  and $Y=\Spec B$ are affine. 
  Then $E$ and $F$ are bounded above
  complexes of finitely generated projective modules over $A$ and
  $B,$ respectively.  
  When testing whether \eqref{eq:boxtimes-sheafhom-noch-neuer}
  induces an 
  isomorphism on cohomology in a fixed degree, only finitely many
  components of $E$ and $F$ are involved. 
  Now use that 
  \begin{equation*}
    \Hom_A(A, I) \otimes \Hom_B(B,J)
    \xsira{\can}
    \Hom_{A \otimes B}(A \otimes B, I \otimes J)
    \ra
    \Hom_{A \otimes B}(A \otimes B, T)
  \end{equation*}
  is obviously a quasi-isomorphism since it identifies with $I
  \otimes J = I \otimes J \ra T.$
\end{proof}


\subsection{Some preparations}
\label{sec:some-preparations}

\begin{lemma}
  \label{l:sheafHom-to-inj}
  Let $I$ be an injective sheaf on a scheme $X.$
  If $F$ is a flat sheaf on $X,$ then $\sheafHom(F,I)$ is
  an injective sheaf.
  If $P$ is a vector bundle on $X$ and $U \subset X$ an open
  subscheme, 
  then $\leftidx{^U}{P}{}$ is flat and
  $\sheafHom(\leftidx{^U}{P}{},I)$ is an injective sheaf. 
  If $X$ is Noetherian and $I$ is injective
  quasi-coherent, then 
  $\sheafHom(\leftidx{^U}{P}{},I)$ is injective quasi-coherent.
\end{lemma}

\begin{proof}
  The functor
  $\Hom_{\Sh(X)}(-, \sheafHom(F,I)) \cong
  \Hom_{\Sh(X)}(- \otimes F, I)$
  is exact. This proves first and second claim since 
  $\leftidx{^U}{P}{} = u_!u^*P$ is certainly flat, where
  $u \colon U \hra X$ is the open immersion.

  We have
  $\sheafHom(u_!u^* P,I) \cong  
  u_*\sheafHom(u^*P,u^*I) \cong 
  \sheafHom(P,u_*u^*I) \cong
  P^\cek \otimes u_*u^*I,$
  reproving injectivity since $u^*=u^!$ and $u_*$ preserve
  injective sheaves.
  Let $X$ be Noetherian and $I$ injective
  quasi-coherent. Then $u^*I$ is injective quasi-coherent by 
  Theorem~\ref{t:injective-in-qcoh-vs-all-OX}.\ref{enum:inj-qcoh-restrict-to-inj-qcoh},
  and $u_*$ preserves quasi-coherence 
  by Lemma~\ref{l:qcqs-preserves-qcoh}
  since $X$ and $U$ are Noetherian.
  Hence $u_*u^*I$ and 
  $P^\cek \otimes u_*u^*I$
  are injective quasi-coherent.
\end{proof}

\begin{corollary}
  \label{c:cplx-sheafhom-flat-inj-is-inj}
  Let $X$ be a scheme.
  Let $F \in C(\Sh(X))$ be a bounded above complex of
  flat sheaves
  and $I \in C(\Sh(X))$ a bounded below complex of injective
  sheaves.
  Then $\sheafHom(F,I)$ is a bounded below complex of injective
  sheaves. In particular, it is h-injective and fibrant as an
  object of 
  $Z^0(C(\Sh(X))).$ 
\end{corollary}

\begin{proof}
  The boundedness condition is obvious, and each component of
  $\sheafHom(F,I)$ is injective as a finite product of sheaves 
  $\sheafHom(\mathcal{F}, \mathcal{I})$ where $\mathcal{F}$ is a
  flat sheaf and $\mathcal{I}$ is an injective
  sheaf (use Lemma~\ref{l:sheafHom-to-inj}).
  The last claim follows from
  Remark~\ref{rem:fibrant-explained}.
\end{proof}

\begin{corollary}
  \label{c:cplx-sheafhom-!-ext-to-inj-qcoh}
  Let $X$ be a Noetherian scheme,
  $I \in C(\Sh(X))$ a bounded below complex of injective
  quasi-coherent sheaves  
  and
  $F \in C(\Sh(X))$ a bounded above complex of
  sheaves whose components are finite products of sheaves of the
  form $\leftidx{^U}{P}{}$
  where $P$ is a vector bundle on $X$ and $U \subset X$ is an open
  subscheme.
  Then $\sheafHom(F,I)$ is a bounded below complex
  of injective sheaves and injective quasi-coherent
  sheaves and fibrant and h-injective as an object
  of $Z^0(C(\Sh(X)))$ or
  $Z^0(C(\Qcoh(X))).$
\end{corollary}

\begin{proof}
  The boundedness condition is obvious, and each component of
  $\sheafHom(F,I)$ is 
  injective quasi-coherent as 
  a finite product of sheaves 
  $\sheafHom(\leftidx{^U}{P}{}, J)$ where $P$ and $U$
  are as above and $J$ is an injective quasi-coherent
  sheaf (use Lemma~\ref{l:sheafHom-to-inj}).
  Now use
  Theorem~\ref{t:injective-in-qcoh-vs-all-OX}.\ref{enum:inj-qcoh-equal-inj-OX-that-qcoh} 
  and 
  Remark~\ref{rem:fibrant-explained}.
\end{proof}

\begin{lemma}
  \label{l:obtain-homotopy-equiv}
  Let $X$ be a 
  scheme over a field $k.$ 
  Let $\Delta \colon X \ra X \times X$
  be the diagonal
  immersion
  and let
  $p,$ $q \colon  X \times X \ra X$ be first and second
  projection. 
  Let $\mathcal{I}$ be an h-injective complex of sheaves on $X$
  and  
  $\Delta_*(\mathcal{I}) \ra \mathcal{K}$ a quasi-isomorphism
  with $\mathcal{K}$ an h-injective complex of sheaves on $X
  \times X.$ 
  Let $j \colon  U \ra X$ be the immersion of an 
  open subscheme
  and $P$ a vector bundle on $U.$
  Then the obvious morphism
  \begin{equation}
    \label{eq:obtain-homotopy-equiv}
    p_*(\sheafHom(q^* j_! P, \Delta_*\mathcal{I})) 
    \ra
    p_*(\sheafHom(q^* j_! P, \mathcal{K}))
  \end{equation}
  in $C(\Sh(X))$
  is a quasi-isomorphism between h-injective objects and hence a
  homotopy equivalence (= an isomorphism in $[C(\Sh(X))]$).
  If $\mathcal{I}$ and $\mathcal{K}$ are even fibrant,
  the same is
  true for both complexes in
  \eqref{eq:obtain-homotopy-equiv}. 
\end{lemma}

\begin{proof}
  The following commutative diagram with two cartesian squares
  explains our notation. 
  \begin{equation*}
    \xymatrix{
      {U} \ar[r]^-{j} \ar[d]^-{\Delta'} & 
      {X} \ar[d]^-{\Delta}\\
      {X \times U} \ar[r]^-{j'} \ar[d]^-{q'} 
      & 
      {X \times X} \ar[d]^-{q} \ar[r]^-{p} & 
      {X} \\
      {U} \ar[r]^-{j} &
      {X}       
    }
  \end{equation*}
  Define $p':= p \comp j'\colon X \times U \ra X.$
  Observe that 
  \begin{align*}
    p_*(\sheafHom(q^* j_! P, -) 
    \sila &
    p_*(\sheafHom(j'_! q'^*P,
    -)
    &&
    \text{(by Lemma~\ref{l:base-change-ringed-spaces-open}.\ref{enum:proper-base-change-open})}\\ 
    \sila &
    p'_*(\sheafHom(q'^*P,
    j'^*(-))
    && \text{(by adjunction)}\\
    \sila &
    p'_*(Q \otimes
    j'^*(-))
  \end{align*}
  where we abbreviate $Q:=(q'^*P)^\cek.$
  Hence the morphism \eqref{eq:obtain-homotopy-equiv}
  is identified with the morphism
  \begin{equation}
    \label{eq:obtain-homotopy-equiv-simplified}
    p'_*(Q \otimes
    j'^*\Delta_*\mathcal{I})
    \ra
    p'_*(Q \otimes
    j'^*\mathcal{K}).
  \end{equation}
  This morphism is obtained by applying $p'_*$ to the 
  morphism
  \begin{equation*}
    Q \otimes
    j'^*\Delta_*\mathcal{I}
    \ra
    Q \otimes
    j'^*\mathcal{K}
  \end{equation*}
  which is certainly a quasi-isomorphism.
  To see that \eqref{eq:obtain-homotopy-equiv-simplified}
  is a quasi-isomorphism it is hence enough to show that
  $Q \otimes j'^*\Delta_*\mathcal{I}$
  is h-limp (:= K-limp) and that
  $Q \otimes j'^*\mathcal{K}$
  is h-injective, by \cite[Cor.~5.17]{spaltenstein}.
  
  The claim for $Q \otimes j'^*\mathcal{K}$ is clear because
  $j'_!$ and $(-\otimes Q^\cek)$ are exact and left 
  adjoint to $j'^*$ and $\sheafHom(Q^\cek,-) \sila (Q \otimes -)$,
  respectively.
  From 
  Lemma~\ref{l:base-change-ringed-spaces-open}.\ref{enum:push-and-restrict-to-open}
  and 
  the projection formula \cite[Exercise~II.5.1.(d)]{Hart}
  we obtain
  \begin{equation}
    \label{eq:rewrite-Q-jdeltaI}
    Q \otimes j'^*\Delta_*\mathcal{I} 
    \sira Q \otimes \Delta'_* j^*\mathcal{I}
    \sira \Delta'_*((\Delta'^*Q) \otimes j^*\mathcal{I}).
  \end{equation}
  Similarly as above, 
  $(\Delta'^*Q) \otimes
  j^*\mathcal{I}$ is h-injective, 
  and in particular h-limp. Hence 
  $\Delta'_*((\Delta'^*Q) \otimes j^*\mathcal{I})$ is h-limp by 
  \cite[Prop.~5.15.(b)]{spaltenstein}. This proves that 
  $Q \otimes j'^*\Delta_*\mathcal{I}$
  is h-limp.
  We conclude that 
  \eqref{eq:obtain-homotopy-equiv-simplified}
  is a quasi-isomorphism.
  
  Flatness of $p'$ implies that $p'^*:\Sh(X) \ra \Sh(X \times
  U)$ is exact, so 
  $p'_*(Q \otimes j'^*\mathcal{K})$ is h-injective.
  On the other hand,
  \eqref{eq:rewrite-Q-jdeltaI} yields
  \begin{equation*}
    p'_*(Q \otimes
    j'^*\Delta_*\mathcal{I}) \sira
    p'_*\Delta'_*((\Delta'^*Q) \otimes j^*\mathcal{I})
    =j_*((\Delta'^*Q) \otimes j^*\mathcal{I})
  \end{equation*}
  and exactness of $j^*$ 
  and h-injectivity of $(\Delta'^*Q) \otimes
  j^*\mathcal{I}$ show that this object is h-injective.

  Recall from Remark~\ref{rem:fibrant-explained} that a
  fibrant complex of sheaves is the same thing as an h-injective
  complex of injective sheaves.  If $\mathcal{I}$ and
  $\mathcal{K}$ are componentwise injective, the exact left
  adjoint functors used above also show that $p'_*(Q \otimes
  j'^*\mathcal{K})$ and $j_*((\Delta'^*Q) \otimes j^*\mathcal{I})
  \sila p'_*(Q \otimes j'^*\Delta_*\mathcal{I})$ are
  componentwise injective. This proves the last claim.
\end{proof}

\begin{corollary}
  \label{c:obtain-homotopy-equiv}
  Let $X,$ $p,$ $q,$ $\Delta,$ $\mathcal{I},$ $\mathcal{K}$ and
  $\Delta_*\mathcal{I} \ra \mathcal{K}$ be as in 
  Lemma~\ref{l:obtain-homotopy-equiv} and assume in addition that
  $X$ is quasi-compact separated and that
  $\mathcal{I}$ and $\mathcal{K}$ are bounded below complexes of
  injective sheaves.
  Let $P$ be a bounded above complex of vector bundles on $X.$
  We also fix an ordered
  finite affine open covering $\mathcal{U}=(U_s)_{s \in S}$
  of $X.$
  Then the obvious morphism
  \begin{equation}
    \label{eq:obtain-homotopy-equiv-2}
    p_*(\sheafHom(q^*(\mathcal{C}_!(P)), \Delta_*\mathcal{I})) 
    \ra
    p_*(\sheafHom(q^*(\mathcal{C}_!(P)), \mathcal{K}))
  \end{equation}
  in $C(\Sh(X))$
  is a quasi-isomorphism between bounded below 
  complexes of injective sheaves and hence a
  homotopy equivalence.
\end{corollary}

\begin{proof}
  Both complexes in \eqref{eq:obtain-homotopy-equiv-2}
  are bounded below
  since
  $\mathcal{C}_!(P)$ is bounded above and both
  $\Delta_*\mathcal{I}$ and $\mathcal{K}$ are bounded below.
  Recall from Remark~\ref{rem:fibrant-explained} that a
  bounded below complex of injective sheaves is the same thing as
  a bounded below fibrant complex.
  If $P$ is bounded, brutal truncation and passing to
  direct summands reduces the statement of the corollary to that
  of Lemma~\ref{l:obtain-homotopy-equiv}.

  
  Now assume that $P$ is bounded above. Fix $d \in \DZ$ and
  limit attention to the components of
  the morphism \eqref{eq:obtain-homotopy-equiv-2}
  in degrees $d-1,$ $d$ and $d+1.$
  These components coincide
  with those of 
  \begin{equation*}
    p_*(\sheafHom(q^*(\mathcal{C}_!(\sigma^{\geq a}(P))),
    \Delta_*\mathcal{I})  
    \ra
    p_*(\sheafHom(q^*(\mathcal{C}_!(\sigma^{\geq a}(P))),
    \mathcal{K}) 
  \end{equation*}
  if $a \ll 0$ is sufficiently small (where $\sigma^{\geq a}(P)$
  denotes 
  the brutal truncation of $P$), and we already know that this
  morphism is a quasi-isomorphism between (bounded below)
  complexes of injective sheaves.
  This implies that
  \eqref{eq:obtain-homotopy-equiv-2} 
  induces an isomorphism on the $d$-th cohomology sheaves, and
  that the degree $d$ components of both sides of
  \eqref{eq:obtain-homotopy-equiv-2} are injective sheaves.
\end{proof}


\subsection{Dualities and $!$-\v{C}ech enhancements}
\label{sec:dual-cech-enhanc}

Let $X$ be a scheme. Given $\mathcal{J} \in C(\Sh(X))$ we define
the dg functor
\begin{equation*}
  \mathcal{D}_\mathcal{J} :=
  \sheafHom(-,\mathcal{J}) \colon 
  C(\Sh(X))^\opp \ra C(\Sh(X)).
\end{equation*}
If $\mathcal{J}$ is an h-injective complex of sheaves, $\mathcal{D}_\mathcal{J}$
induces the functor
\begin{equation}
  \label{eq:RHom-J-D-Sh}
  \mathcal{D}_\mathcal{J}=\sheafHom(-,\mathcal{J})=\bR\sheafHom(-,\mathcal{J})
  \colon  
  D(\Sh(X))^\opp \ra D(\Sh(X))
\end{equation}
of triangulated categories. 
There are two important cases where this functor restricts to an
equivalence on $\mfPerf'(X)$ or $D^b_\Coh(\Sh(X)).$

\subsubsection{Duality for perfect complexes}
\label{sec:dual-perf-compl}

Let $\mathcal{J}$ be a bounded below complex of
injective sheaves on $X$ which is isomorphic to $\mathcal{O}_X$
in $D(\Sh(X)).$  
In this case the functor
\eqref{eq:RHom-J-D-Sh} is isomorphic to 
$\bR\sheafHom(-,\mathcal{O}_X)$
and induces
an equivalence
\begin{equation}
  \label{eq:duality-perf}
  \mathcal{D}_\mathcal{J} \colon 
  \mfPerf'(X)^\opp \sira \mfPerf'(X)
\end{equation}
of triangulated categories
which
is a duality since 
$\id \sira \mathcal{D}_\mathcal{J}^2$ naturally
(\cite[\href{http://stacks.math.columbia.edu/tag/08DQ}{Lemma~08DQ}]{stacks-project}).

\begin{lemma}
  \label{l:duality-and-!-Cech-perf}
  Let $X$ be a quasi-compact separated scheme, 
  let $\mathcal{J}$ be a bounded below complex of
  injective sheaves on $X$ which is isomorphic to $\mathcal{O}_X$
  in $D(\Sh(X)),$ and  
  fix an ordered
  finite affine open covering
  $\mathcal{U}=(U_s)_{s \in S}$ of $X.$
  Then 
  the dg functor
  \begin{equation*}
    \mathcal{D}_\mathcal{J} \comp \mathcal{C}_! \colon  \Cech_!(X)
    \xra{\mathcal{C}_!} C(\Sh(X)) 
    \xra{\mathcal{D}_\mathcal{J}} 
    C(\Sh(X))^\opp 
  \end{equation*}
  lands in the full dg subcategory of bounded below complexes of
  injective sheaves 
  and 
  is quasi-fully faithful, i.\,e.\ the morphism
  \begin{equation}
    \label{eq:D-C!-qiso-perf}
    \mathcal{D}_\mathcal{J} \comp \mathcal{C}_!
    \colon  \Hom_{\Cech_!(X)}(P,Q) \ra
    \Hom_{C(\Sh(X))}(\mathcal{D}_\mathcal{J}(\mathcal{C}_!(Q)),
    \mathcal{D}_\mathcal{J}(\mathcal{C}_!(P))) 
  \end{equation}
  is a quasi-isomorphism for all $P, Q \in \Cech_!(X).$
\end{lemma}

\begin{proof}
  For $P \in \Cech_!(X)$, the object
  $\mathcal{D}_\mathcal{J}(\mathcal{C}_!(P)) =
  \sheafHom(\mathcal{C}_!(P), 
  \mathcal{J})$ is a bounded below complex of injective sheaves
  and h-injective by
  Corollary~\ref{c:cplx-sheafhom-flat-inj-is-inj} since all
  sheaves $\leftidx{^{U_I}}{(P^p)}{}$ are flat.

  Let $P, Q \in \Cech_!(X).$ 
  For $m \in \DZ$ consider the commutative diagram
  \begin{equation*}
    \xymatrix{
      {\Hom_{[\Cech_!(X)]}([m]P,Q)}
      \ar[r]^-{\mathcal{D}_\mathcal{J} \comp \mathcal{C}_!}
      \ar[d]^{\mathcal{C}_!} & 
      {\Hom_{[C(\Sh(X))]}([m]\mathcal{D}_\mathcal{J}(\mathcal{C}_!(Q)),
        \mathcal{D}_\mathcal{J}(\mathcal{C}_!(P)))} 
      \ar[d]^-{\can} \\
      {\Hom_{D(\Sh(X))}([m]\mathcal{C}_!(P),\mathcal{C}_!(Q))}
      \ar[r]^-{\mathcal{D}_\mathcal{J}} &
      {\Hom_{D(\Sh(X))}
        ([m]\mathcal{D}_\mathcal{J}(\mathcal{C}_!(Q)),  
        \mathcal{D}_\mathcal{J}(\mathcal{C}_!(P)))}   
    }
  \end{equation*}
  whose arrows are induced by the indicated functors.
  The left vertical arrow $\mathcal{C}_!$ is an isomorphism
  by Proposition~\ref{p:abstract-cech-!-bounded-above-full-faithful},
  the lower horizontal arrow $\mathcal{D}_\mathcal{J}$ is an
  isomorphism by the 
  duality 
  equivalence~\eqref{eq:duality-perf},
  and the right vertical
  arrow $\can$ is an isomorphism since 
  $\mathcal{D}_\mathcal{J}(\mathcal{C}_!(P))$
  is h-injective.
  Hence the upper horizontal arrow is an isomorphism, and this
  just means that \eqref{eq:D-C!-qiso-perf} is a
  quasi-isomorphism. 
\end{proof}

\subsubsection{Duality for bounded derived categories of coherent sheaves}
\label{sec:dual-bound-deriv}

Let $X$ be a locally Noetherian scheme that has a dualizing
complex in the sense of \cite[V.2,
  p.~258]{hartshorne-residues-duality}.
Let $\mathcal{J}$ be a dualizing
complex.
We can 
(by \cite[V.2, p.~257]{hartshorne-residues-duality}, using 
Theorem~\ref{t:injective-in-qcoh-vs-all-OX}.\ref{enum:inj-qcoh-equal-inj-OX-that-qcoh})
and will assume that $\mathcal{J}$
is a bounded complex of injective
quasi-coherent sheaves on $X.$
Then the functor \eqref{eq:RHom-J-D-Sh} 
induces an equivalence
\begin{equation}
  \label{eq:duality-D-b-Coh}
  \mathcal{D}_\mathcal{J} \colon 
  D^b_\Coh(\Sh(X))^\opp \sira D^b_\Coh(\Sh(X))
\end{equation}
of triangulated categories
(see \cite[V.2]{hartshorne-residues-duality}) satisfying
$\id \sira \mathcal{D}_\mathcal{J}^2.$



\begin{lemma}
  \label{l:duality-and-!-Cech-D-b-Coh}
  Let $X$ be a Noetherian separated scheme having a dualizing
  complex $\mathcal{J}.$ Assume that $\mathcal{J}$
  is a bounded complex of injective
  quasi-coherent sheaves on $X.$
  Fix
  an ordered
  finite affine open covering 
  $\mathcal{U}=(U_s)_{s \in S}$ of $X.$
  Then the dg functor 
  \begin{equation*}
    \mathcal{D}_\mathcal{J} \comp \mathcal{C}_! \colon  \Cech_!^{-,b}(X)
    \xra{\mathcal{C}_!} C(\Sh(X)) 
    \xra{\mathcal{D}_\mathcal{J}} 
    C(\Sh(X))^\opp 
  \end{equation*}
  lands in the full dg subcategory of bounded below complexes of
  injective quasi-coherent sheaves 
  and 
  is quasi-fully faithful, i.\,e.\ the morphism
  \begin{equation*}
    \mathcal{D}_\mathcal{J} \comp \mathcal{C}_!
    \colon  \Hom_{\Cech_!^{-,b}(X)}(P,Q) \ra
    \Hom_{C(\Sh(X))}(\mathcal{D}_\mathcal{J}(\mathcal{C}_!(Q)),
    \mathcal{D}_\mathcal{J}(\mathcal{C}_!(P))) 
  \end{equation*}
  is a quasi-isomorphism for all $P, Q \in \Cech_!^{-,b}(X).$
\end{lemma}

\begin{proof}
  This is proved in the same way as
  Lemma~\ref{l:duality-and-!-Cech-perf},
  using 
  Corollary~\ref{c:cplx-sheafhom-!-ext-to-inj-qcoh}
  instead of
  Corollary~\ref{c:cplx-sheafhom-flat-inj-is-inj} 
  and
  the duality equivalence~\eqref{eq:duality-D-b-Coh}
  instead of \eqref{eq:duality-perf}.
\end{proof}



\subsection{Homological smoothness and the structure sheaf of the
  diagonal}
\label{sec:homol-smoothn-struct-II}


\begin{proof}[Proof of
  Theorem~\ref{t:mfPerf-abstract-Cechobj-smooth-vs-diagonal-sheaf-perfect}]
  Let $X$ be a Noetherian \ref{enum:GSP}-scheme over a field $k$
  and assume that $X \times X$ is also Noetherian.
  Fix an ordered finite affine open covering 
  $\mathcal{U}=(U_s)_{s \in S}$ of $X.$ 
  Then 
  $k$-smoothness of $\mfPerf(X) \sira \mfPerf'(X)$ is equivalent
  to $k$-smoothness 
  of its enhancement $\Cech_!(X)$ 
  (Proposition~\ref{p:abstract-cech-!-object-enhancement},
  Remark~\ref{rem:Perf-DbCoh-smooth}). 

  Let $E$ be a classical generator of $\mfPerf'(X)$ (which exists
  by \cite[2.1, 3.1]{bondal-vdbergh-generators}).
  We can and will assume that $E$ is a bounded complex of vector
  bundles, 
  by condition~\ref{enum:GSP}. (We can even assume that $E$ is
  a vector bundle
  by replacing $E$ by the direct sum of its
  components.) Then $E \in \Cech_!(X).$ 
  By
  \cite[Prop.~2.18]{valery-olaf-matrix-factorizations-and-motivic-measures},
  $k$-smoothness of $\Cech_!(X)$ is equivalent to
  $k$-smoothness of the dg algebra
  \begin{equation*}
    A:=\End_{\Cech_!(X)}(E),
  \end{equation*}
  i.\,e.\ to the condition $A \in \per(A \otimes A^{\opp}).$
  We also know that
  $E^\cek$ is a classical generator of $\mfPerf'(X)$, by
  \eqref{eq:duality-perf}, 
  and that 
  $E^\cek \boxtimes E$ is a classical
  generator of $\mfPerf'(X \times X)$
  and a compact
  generator of 
  $D_\Qcoh(\Sh(X \times X)),$
  by \cite[Lemma~3.4.1, 3.1,
  2.1]{bondal-vdbergh-generators}.

  Let $\mathcal{O}_X \ra \mathcal{J}$ be a resolution by
  injective quasi-coherent sheaves. Note that $\mathcal{J}$
  consists of 
  injective sheaves (by
  Theorem~\ref{t:injective-in-qcoh-vs-all-OX}.\ref{enum:inj-qcoh-equal-inj-OX-that-qcoh}).
  As before we abbreviate
  $\mathcal{D}_\mathcal{J}=\sheafHom(-, \mathcal{J}).$
  Lemma~\ref{l:duality-and-!-Cech-perf} shows that
  \begin{equation}
    \label{eq:A-D-lower-shriek-qiso-dga-perf}
    \mathcal{D}_\mathcal{J} \comp \mathcal{C}_!
    \colon
    A =
    \End_{\Cech_!(X)}(E)
    \ra
    \End_{C(\Sh(X))}(\mathcal{D}_\mathcal{J}(\mathcal{C}_!(E)))^\opp
  \end{equation}
  is a quasi-isomorphism of dg algebras.
  Note that $\mathcal{C}_!(E) \ra E$ and 
  $\mathcal{D}_\mathcal{J}(E) \ra
  \mathcal{D}_\mathcal{J}(\mathcal{C}_!(E))$
  are quasi-isomorphisms. In particular, we
  find a quasi-isomorphism $i:\mathcal{C}_!(E) \ra I$ 
  with $I$
  a bounded below 
  complex
  of injective quasi-coherent sheaves on $X.$
  The object $\mathcal{D}_\mathcal{J}(\mathcal{C}_!(E))$ is
  already a bounded below complex of injective quasi-coherent
  sheaves, by 
  Corollary~\ref{c:cplx-sheafhom-!-ext-to-inj-qcoh}.
  Let
  $\tau\colon \mathcal{D}_\mathcal{J}(\mathcal{C}_!(E))
  \boxtimes I
  \ra T$ be a quasi-isomorphism with $T$ a
  bounded below complex of injective quasi-coherent
  sheaves. 
  Consider the object 
  \begin{equation*}
    P:=\mathcal{D}_\mathcal{J}(\mathcal{C}_!(E)) \boxtimes
    \mathcal{C}_!(E)
  \end{equation*}
  and the commutative diagram
  \begin{equation*}
    \xymatrix{
      {A^\opp \otimes A} \ar[d]^-{(\mathcal{D}_\mathcal{J} \comp
        \mathcal{C}_!) \otimes 
        \mathcal{C}_!}
      \ar[rd]^-{(\mathcal{D}_\mathcal{J} \comp
      \mathcal{C}_!) \boxtimes 
      \mathcal{C}_!} 
      \\
      {
        \leftidx{_X}{(\mathcal{D}_\mathcal{J}(\mathcal{C}_!(E)),
          \mathcal{D}_\mathcal{J}(\mathcal{C}_!(E)))}{} 
        \otimes
        \leftidx{_X}{(\mathcal{C}_!(E), \mathcal{C}_!(E))}{}
      }
      \ar[r]^-{\boxtimes}
      \ar[d]^-{\id \otimes i_*}
      &
      {\leftidx{_{X \times X}}{(P,P)}{}} 
      \ar[d]^-{(\id \boxtimes i)_*}
      \ar[r]^-{(\tau \comp (\id \boxtimes i))_*}
      &
      {\leftidx{_{X \times X}}{(P, T)}{}}
      \ar@{}[d]|-{\verteq}
      \\
      {\leftidx{_X}{(\mathcal{D}_\mathcal{J}(\mathcal{C}_!(E)),
          \mathcal{D}_\mathcal{J}(\mathcal{C}_!(E)))}{}  
        \otimes 
        \leftidx{_X}{(\mathcal{C}_!(E),I)}{}}
      \ar[r]^-{\boxtimes}
      &
      {\leftidx{_{X \times X}}{(P, 
          \mathcal{D}_\mathcal{J}(\mathcal{C}_!(E)) \boxtimes
          I)}{}}  
      \ar[r]^-{\tau_*}
      &
      {\leftidx{_{X \times X}}{(P, T)}{}}
    }
  \end{equation*}
  where 
  we abbreviate
  $\leftidx{_?}{(-,-)}{}:=\Hom_{C(\Sh(?))}(-,-).$
  The left vertical composition $(\mathcal{D}_\mathcal{J} \comp \mathcal{C}_!)
  \otimes (i_* \comp \mathcal{C}_!)$ is a
  quasi-isomorphism because it is the tensor product over the
  field 
  $k$ of the quasi-isomorphism
  \eqref{eq:A-D-lower-shriek-qiso-dga-perf}
  with the quasi-isomorphism
  $i_* \comp \mathcal{C}_! \colon A \ra
  \Hom_{C(\Sh(X))}(\mathcal{C}_!(E),I)$ (use
  Proposition~\ref{p:abstract-cech-!-object-enhancement}
  and   
  the fact that the complex $I$ of sheaves is h-injective
  as a 
  bounded below complex of 
  injective 
  (quasi-coherent) sheaves
  (Theorem~\ref{t:injective-in-qcoh-vs-all-OX}.\ref{enum:inj-qcoh-equal-inj-OX-that-qcoh}
  and Remark~\ref{rem:fibrant-explained})). 
  The composition in the lower row is a quasi-isomorphism by
  Proposition~\ref{p:tensor-product-of-endos-vs-endos-of-boxproduct}.
  Hence the composition of the morphism 
  \begin{equation}
    \label{eq:AoppoA-EndP-perf}
    A^\opp \otimes A \xra{(\mathcal{D}_\mathcal{J} \comp
      \mathcal{C}_!) \boxtimes 
      \mathcal{C}_!} 
    \End_{C(\Sh(X \times X))}(P)
  \end{equation}
  of dg algebras
  with the upper right horizontal map $(\tau \comp (\id \boxtimes
  i))_*$ in the above diagram is a 
  quasi-isomorphism. 

  Recall the enhancement $C^\hinj_\Qcoh(\Sh(X \times X))$ of
  $D(\Qcoh(X \times X)) \sira D_\Qcoh(\Sh(X \times X))$ from
  section~\ref{sec:inject-enhanc}.
  Note that $\tau \comp (\id \boxtimes i)\colon P \ra T$ is a
  quasi-isomorphism and that $T$ is an h-injective complex of
  sheaves on the Noetherian scheme $X \times X$
  (Theorem~\ref{t:injective-in-qcoh-vs-all-OX}.\ref{enum:inj-qcoh-equal-inj-OX-that-qcoh} 
  and Remark~\ref{rem:fibrant-explained})
  and a compact
  generator of 
  $[C^\hinj_\Qcoh(\Sh(X \times X))] \sira
  D_\Qcoh(\Sh(X \times X))$
  because it is isomorphic to $E^\cek
  \boxtimes E$ (use Remark~\ref{rem:boxtimes-exact}).
  We apply
  Proposition~\ref{p:homotopy-categories-triang-via-dg-algebras}.\ref{hoI-via-B}
  to the dg subcategory $C^\hinj_\Qcoh(\Sh(X \times X))$ of
  $C(\Sh(X \times X))$ 
  and $\beta$ there the morphism
  \eqref{eq:AoppoA-EndP-perf} and obtain an equivalence
  \begin{equation*}
    F:=\Hom_{C(\Sh(X \times X))}(P,-) \colon 
    [C^\hinj_\Qcoh(\Sh(X \times X))] \sira D(A^\opp \otimes A)
  \end{equation*}
  of triangulated categories.
  The category on the left identifies with
  $D_\Qcoh(\Sh(X \times X)),$
  and $F$ induces an equivalence 
  $\mfPerf'(X \times X) \sira \per(A^\opp \otimes A)$ on the
  subcategories of compact objects.
  Recall that $A$ is k-smooth if and only if $A^\opp$ is
  $k$-smooth (see e.\,g.\
  \cite[Remark~3.11]{valery-olaf-smoothness-equivariant}).  
  We claim that $F$ maps (an h-injective lift of)
  $\Delta_*(\mathcal{O}_X)$ to (an object isomorphic to) $A^\opp.$

  Let 
  $\Delta_*(\mathcal{J}) \ra \mathcal{K}$ be a quasi-isomorphism
  with $\mathcal{K}$ a bounded below complex of injective
  quasi-coherent sheaves on $X \times X.$
  Then $\mathcal{K}$
  consists of 
  injective sheaves (by
  Theorem~\ref{t:injective-in-qcoh-vs-all-OX}.\ref{enum:inj-qcoh-equal-inj-OX-that-qcoh}).
  Let $p,$ $q \colon X \times X \ra X$ be first and second
  projection. 
  We have canonical identifications and quasi-isomorphisms of dg
  $A^\opp \otimes A$-modules
  \begin{align*}
    F(\mathcal{K}) & =
    \Hom_{C(\Sh(X \times
      X))}(\mathcal{D}_\mathcal{J}(\mathcal{C}_!(E)) \boxtimes  
    \mathcal{C}_!(E), \mathcal{K}) \\
    & = \Hom_{C(\Sh(X \times
      X))} (p^*\mathcal{D}_\mathcal{J}(\mathcal{C}_!(E)) \otimes
    q^*\mathcal{C}_!(E), \mathcal{K})\\
    & = \Hom_{C(\Sh(X \times X))}(p^*\mathcal{D}_\mathcal{J}(\mathcal{C}_!(E)),
    \sheafHom(q^*\mathcal{C}_!(E), \mathcal{K})) && \text{(by
      adjunction)}\\
    & = \Hom_{C(\Sh(X))}(\mathcal{D}_\mathcal{J}(\mathcal{C}_!(E)),
    p_*(\sheafHom(q^*\mathcal{C}_!(E), \mathcal{K}))) && \text{(by
      adjunction)}\\
    & 
    \la
    \Hom_{C(\Sh(X))}(\mathcal{D}_\mathcal{J}(\mathcal{C}_!(E)),
    p_*(\sheafHom(q^*(\mathcal{C}_!(E)), \Delta_*\mathcal{J})))
    && \text{(homotopy equiv.\ by Cor.~\ref{c:obtain-homotopy-equiv})}\\ 
    & = \Hom_{C(\Sh(X))}(\mathcal{D}_\mathcal{J}(\mathcal{C}_!(E)),
    p_*\Delta_*(\sheafHom(\Delta^*q^*(\mathcal{C}_!(E)),
    \mathcal{J})))
    && \text{(by adjunction)}\\
    & = \Hom_{C(\Sh(X))}(\mathcal{D}_\mathcal{J}(\mathcal{C}_!(E)),
    \sheafHom(\mathcal{C}_!(E),
    \mathcal{J})) && \\
    & =
    \Hom_{C(\Sh(X))}(\mathcal{D}_\mathcal{J}(\mathcal{C}_!(E)),
    \mathcal{D}_\mathcal{J}(\mathcal{C}_!(E))) 
    && \text{(by definition)}\\
    & 
    \la
    A^\opp
    && \text{(quasi-isom.\ by 
      \eqref{eq:A-D-lower-shriek-qiso-dga-perf}).}
  \end{align*}
  This proves our claim and shows the equivalence of 
  \ref{enum:mfPerf-smooth} 
  and \ref{enum:diagonal-cpt}
  in
  Theorem~\ref{t:mfPerf-abstract-Cechobj-smooth-vs-diagonal-sheaf-perfect}. 

  If $X$ is in addition of finite type
  over $k,$ the equivalence of 
  \ref{enum:diagonal-cpt} and
  \ref{enum:geometrically-smooth} follows from 
  Proposition~\ref{p:glatt-diagonaleperfekt-regular}.
  Finally, recall from
  Remark~\ref{rem:onGSP}.\ref{enum:resolution-property} that 
  any separated scheme of finite type over $k$ 
  having
  the resolution
  property satisfies
  condition~\ref{enum:GSP}.
\end{proof}

\subsection{Geometric smoothness and the structure sheaf of the
  diagonal}
\label{sec:geom-smoothn-diagonal}



\begin{proposition}
  \label{p:glatt-diagonaleperfekt-regular}
  Let $X$ be a separated scheme locally
  of finite type over a field $k,$ and let $\Delta\colon X \ra X
  \times X$ be the diagonal (closed) immersion.
  Consider the following conditions.
  \begin{enumerate}
  \item 
    \label{enum:glatt}
    $X$ is smooth over $k.$
  \item
    \label{enum:diagonaleperfekt}
    $\Delta_*(\mathcal{O}_X) \in \mfPerf'(X \times X).$ 
  \item 
    \label{enum:regular}
    $X$ is regular.
  \end{enumerate}
  Then \ref{enum:glatt} and 
  \ref{enum:diagonaleperfekt} are equivalent, and they imply
  \ref{enum:regular}. If the field $k$ is perfect, these
  three conditions are equivalent.
\end{proposition}

\begin{proof}
  Conditions~\ref{enum:glatt} and \ref{enum:regular}
  are obviously local on $X,$ and the same is true for 
  condition~\ref{enum:diagonaleperfekt}: if $(U_i)_{i \in I}$ is an
  open covering of $X,$ then the sets $(U_i \times U_i)_{i \in
    I}$ together 
  with $(X \times X)  \setminus \Delta(X)$
  form
  an open covering
  of $X \times X,$ and the restriction of
  $\Delta_*(\mathcal{O}_X)$ to 
  $X \times X \setminus \Delta(X)$ is zero.

  Hence we can assume that $X=\Spec R$ for $R$ a finitely
  generated $k$-algebra. Then the condition 
  $\Delta_*(\mathcal{O}_X) \in
  \mfPerf'(X \times X)$ means precisely that $R \in \per(R
  \otimes R),$ by Example~\ref{exam:perfect-on-affine}.
  Now all claims follow from the following
  Proposition~\ref{p:affine-glatt-diagonaleperfekt-regular}. 
\end{proof}

\begin{proposition}
  \label{p:affine-glatt-diagonaleperfekt-regular}
  Let $R$ be a finitely generated algebra over a field $k$ and
  put $X=\Spec R.$ 
  Consider the following conditions.
  \begin{enumerate}[label=(\roman*)]
  \item 
    \label{enum:affine-glatt}
    $X$ is smooth over $k.$
  \item
    \label{enum:affine-diagonaleperfekt}
    $R \in \per(R \otimes R).$
  \item
    \label{enum:affine-projdim-finite}
    $\projdim_{R \otimes R} R < \infty.$
  \item 
    \label{enum:affine-HH}
    There are a positive even integer $i$ and a positive odd
    integer $j$ such that $HH_i(R,R)=0$ and $HH_j(R,R)=0.$
    Here $HH_*(R,R)$ denotes the Hochschild homology of the
    $k$-algebra $R$ with
    values in $R.$ 
  \item 
    \label{enum:affine-regular}
    $R$ is regular.
  \end{enumerate}
  Then the four conditions \ref{enum:affine-glatt},
  \ref{enum:affine-diagonaleperfekt}, 
  \ref{enum:affine-projdim-finite} and
  \ref{enum:affine-HH} 
  are equivalent, and they imply
  \ref{enum:affine-regular}. If the field $k$ is perfect, all
  five conditions are equivalent.
\end{proposition}

We give condition~\ref{enum:affine-HH} just for curiosity and
treat it separately in the following proof.

\begin{proof}
  \ref{enum:affine-glatt} $\Rightarrow$
  \ref{enum:affine-diagonaleperfekt}: Then $X \times X$ is smooth over
  $k$ and in particular regular
  (\cite[Cor.~6.32]{goertz-wedhorn-AGI}. 
  So $R \otimes R$
  is regular or, equivalently, of 
  finite global dimension. This means that the diagonal bimodule
  $R$ has a finite resolution by finitely generated projective $R
  \otimes R$-modules (since $R \otimes R$ is Noetherian).

  \ref{enum:affine-diagonaleperfekt} $\Rightarrow$
  \ref{enum:affine-projdim-finite}: trivial.

  \ref{enum:affine-projdim-finite} $\Rightarrow$ \ref{enum:affine-regular}:
  Let $P \ra R$
  be a finite resolution by
  projective $R \otimes R$-modules.
  Let $M$ be any $R$-module. Then $M \otimes_R P \ra M \otimes_R R
  = M$ is a
  quasi-isomorphism
  since both $P$ and $R$ consist of flat left $R$-modules (they
  are h-flat as dg left $R$-modules). 
  Each $M \otimes_R P^i$ is a projective $R$-module
  as a direct summand of a direct sum of
  modules of the form $M \otimes_R (R \otimes R) = M \otimes R.$
  Hence $M \otimes_R P \ra M$ is a finite projective resolution,
  so the projective dimension of $M$ over $R$ is smaller than the
  length of the resolution $P.$
  Hence $R$ has finite global dimension and is regular.

  \ref{enum:affine-projdim-finite} $\Rightarrow$
  \ref{enum:affine-glatt}:
  Let $\ol{k}$ be an
  algebraic closure of $k$ and put $\ol{R}=R \otimes \ol{k}.$
  Applying $(-\otimes \ol{k})$ to a finite projective resolution
  of $R$ 
  shows that $\projdim_{\ol{R} \otimes_{\ol{k}} \ol{R}} \ol{R} < \infty.$
  We already know that
  \ref{enum:affine-projdim-finite} implies
  \ref{enum:affine-regular}. Hence
  $\ol{R}$ is regular, so
  $R$ is smooth over $k$ (by \cite[Cor.~6.32]{goertz-wedhorn-AGI}).  
  


    


  \ref{enum:affine-projdim-finite} $\Rightarrow$
  \ref{enum:affine-HH}: trivial.

  \ref{enum:affine-HH} $\Rightarrow$
  \ref{enum:affine-glatt}:
  This is the main theorem of
  \cite{avramov-vigue-poirrier-hochschild-homology-criteria-for-smoothness}.

  Finally, if $k$ is perfect, it is well known that regularity
  and $k$-smoothness are equivalent (see e.\,g.\
  \cite[Cor.~3.33]{Liu}). 
\end{proof}

\subsection{Smoothness of bounded derived categories of coherent sheaves}
\label{sec:smoothn-bound-deriv}

We will need the following two results for the proof of 
Theorem~\ref{t:D-b-Coh-smooth}.

\begin{theorem}
  \label{t:generator-of-DbCoh-finite-type-any-field}
  Let $X$ be 
  a separated scheme of finite type over a
  field.
  Then the category $D^b(\Coh(X))$ has a classical generator.
\end{theorem}

\begin{proof}
  The proof of
  \cite[Prop.~6.8]{lunts-categorical-resolution}
  (where the field is assumed to be perfect)
  also works over an arbitrary field because
  the 
  regular locus of 
  a scheme that is locally of finite type
  over a field is open,
  by \cite[Rem.~6.25.(4)]{goertz-wedhorn-AGI}.
\end{proof}

\begin{theorem}
  \label{t:generator-of-DbCoh-product}
  Let $X$ and $Y$ be separated schemes of finite type over a
  perfect 
  field $k.$ 
  If
  $S$ and $T$ are classical generators of $D^b(\Coh(X))$
  and $D^b(\Coh(Y)),$ respectively,
  then $S \boxtimes T$ is a
  classical generator of $D^b(\Coh(X \times Y)).$
\end{theorem}

\begin{remark}
  The assumption that $k$ is perfect is necessary:
  Let $k \subset K$ be a finite purely inseparable field
  extension (e.\,g.\ $\DF_2(t^2) \subset \DF_2(t)$) 
  and consider $X=Y=\Spec K.$
  Since $K \otimes K$ is not reduced (i.\,e.\ it contains a
  non-zero nilpotent element), $X
  \times X$ is not regular and $\mfPerf(X \times X)
  \subsetneq D^b(\Coh(X \times X)),$
  by Proposition~\ref{p:regular-vs-singularity-cat}.
  Consider the classical generators $E=F=K$ 
  of $D^b(\Coh(X))=\mfPerf(X).$
  Then $E \boxtimes F= K \otimes K$
  is a
  classical 
  generator of 
  $\mfPerf(X \times X)$ but not of
  $D^b(\Coh(X \times X)).$
\end{remark}

\begin{proof}
  The proof of
  \cite[Thm.~6.3]{lunts-categorical-resolution} shows that there
  are classical generators $S'$ and $T'$ of $D^b(\Coh(X))$
  and $D^b(\Coh(Y)),$ respectively, such that $S' \boxtimes
  T'$ is a classical generator of $D^b(\Coh(X \times Y)).$
  From $S' \in \thick(S)$ we obtain $S' \boxtimes T' \in
  \thick(S \boxtimes T'),$ so $S \boxtimes T'$ is a
  classical generator of $D^b(\Coh(X \times Y)).$
  Similarly, we see that $S \boxtimes T$ is a
  classical generator of $D^b(\Coh(X \times Y)).$
\end{proof}

We now have all ingredients for the proof of
Theorem~\ref{t:D-b-Coh-smooth}; it is similar to that of
Theorem~\ref{t:mfPerf-abstract-Cechobj-smooth-vs-diagonal-sheaf-perfect}.

\begin{proof}[Proof of Theorem~\ref{t:D-b-Coh-smooth}]
  Let $X$ be a separated scheme of finite type over a perfect
  field $k$ that has the resolution property.
  The category
  $D^b(\Coh(X)) \sira D^b_\Coh(\Qcoh(X)) \sira
  D^b_\Coh(\Sh(X))$ 
  has a classical generator $G$
  (Theorem~\ref{t:generator-of-DbCoh-finite-type-any-field}).
  Since $X$ is a \ref{enum:RES}-scheme
  we can and will assume that $G$ is a bounded above
  complex 
  of vector bundles with bounded 
  coherent cohomology
  (Proposition~\ref{p:D-minus-Coh(Sh)-strict-coherent}). 

  Fix an ordered finite affine open covering
  $\mathcal{U}=(U_s)_{s \in S}$ of $X$ and consider
  $G$ as an object of the enhancement
  $\Cech_!^{-,b}(X)$
  of $D^b_\Coh(\Sh(X))$
  (Proposition~\ref{p:shriek-cech-enhancement-D-minus-and-b-Coh}). 
  By
  \cite[Prop.~2.18]{valery-olaf-matrix-factorizations-and-motivic-measures},
  $k$-smoothness of 
  $\Cech_!^{-,b}(X)$
  is equivalent to
  $k$-smoothness of the dg algebra
  \begin{equation*}
    A:=\End_{\Cech_!^{-,b}(X)}(G).  
  \end{equation*}

  Since any scheme of finite type over a field
  has a dualizing complex 
  (\cite[V.10, p.~299]{hartshorne-residues-duality}),
  there is a bounded
  complex $\mathcal{J}$ of injective quasi-coherent sheaves on
  $X$ that is a dualizing complex. Abbreviate
  $\mathcal{D}_\mathcal{J}=\sheafHom(-, \mathcal{J}).$
  Lemma~\ref{l:duality-and-!-Cech-D-b-Coh} shows that
  \begin{equation}
    \label{eq:A-D-lower-shriek-qiso-dga}
     A =
    \End_{\Cech_!^{-,b}(X)}(G)
    \ra
    \End_{C(\Sh(X))}(\mathcal{D}_\mathcal{J}(\mathcal{C}_!(G)))^\opp
  \end{equation}
  is a quasi-isomorphism of dg algebras.
  The object $\mathcal{D}_\mathcal{J}(G)$ 
  is a classical generator of $D^b_\Coh(\Sh(X))$ by the duality
  equivalence~\eqref{eq:duality-D-b-Coh}, and 
  $\mathcal{D}_\mathcal{J}(G) \boxtimes G$ is a classical
  generator of 
  $D^b(\Coh(X \times X)) \sira 
  D^b_\Coh(\Sh(X \times X)),$ by
  Theorem~\ref{t:generator-of-DbCoh-product}.
  Note that $\mathcal{C}_!(G) \ra G$ and 
  $\mathcal{D}_\mathcal{J}(G) \ra
  \mathcal{D}_\mathcal{J}(\mathcal{C}_!(G))$
  are quasi-isomorphisms. In particular, we
  find a quasi-isomorphism $i:\mathcal{C}_!(G) \ra I$ 
  with $I$
  a bounded below 
  complex
  of injective quasi-coherent sheaves on $X.$
  Note that $\mathcal{D}_\mathcal{J}(\mathcal{C}_!(G))$ is
  already a bounded below complex of injective quasi-coherent
  sheaves, by 
  Corollary~\ref{c:cplx-sheafhom-!-ext-to-inj-qcoh}.
  Let
  $\tau\colon \mathcal{D}_\mathcal{J}(\mathcal{C}_!(G))
  \boxtimes I
  \ra T$ be a quasi-isomorphism with $T$ a
  bounded below complex of injective quasi-coherent
  sheaves. 
  The object 
  \begin{equation*}
    P:=\mathcal{D}_\mathcal{J}(\mathcal{C}_!(G)) \boxtimes
    \mathcal{C}_!(G) 
  \end{equation*}
  is a classical generator 
  of $D^b_\Coh(\Sh(X \times X))$ 
  because it is isomorphic
  to 
  $\mathcal{D}_\mathcal{J}(G) \boxtimes G.$
  As in the proof of
  Theorem~\ref{t:mfPerf-abstract-Cechobj-smooth-vs-diagonal-sheaf-perfect}
  we see that the composition
  of the morphism 
  \begin{equation*}
    A^\opp \otimes A \xra{(\mathcal{D}_\mathcal{J} \comp
      \mathcal{C}_!) \boxtimes 
      \mathcal{C}_!} 
    \End_{C(\Sh(X \times X))}(P)
  \end{equation*}
  of dg algebras
  with the morphism
  \begin{equation*}
    {(\tau \comp (\id \boxtimes i))_*}
    \colon
    \End_{C(\Sh(X \times X))}(P),
    \ra
    \Hom_{C(\Sh(X \times X))}(P, T)
  \end{equation*}
  is a quasi-isomorphism. 
  Recall the
  enhancement
  $C_\Coh^{+,b}(\InjSh(X \times X))$ 
  of $D^b(\Coh(X \times X)) \sira D^b_\Coh(\Sh(X \times X))$ 
  from 
  section~\ref{sec:inject-enhanc}.
  Note that $T \in C_\Coh^{+,b}(\InjSh(X \times X))$ 
  by 
  Theorem~\ref{t:injective-in-qcoh-vs-all-OX}.\ref{enum:inj-qcoh-equal-inj-OX-that-qcoh}
  and that
  $D^b(\Coh(X \times X))$
  is
  Karoubian by \cite{le-chen-karoubi-trcat-bdd-t-str}. 
  We are allowed to apply
  Proposition~\ref{p:homotopy-categories-triang-via-dg-algebras}.\ref{enum:hoI-via-B-small}
  in this setting and see that
  \begin{equation*}
    F:=\Hom_{C(\Sh(X \times X))}(P,-) \colon
    [C_\Coh^{+,b}(\InjSh(X \times X))] 
    \ra 
    \per(A^\opp 
    \otimes A)
  \end{equation*} 
  is an equivalence of triangulated categories. 
  Certainly $\Delta_*(\mathcal{J}) \in D^b_\Coh(\Sh(X \times
  X)).$
  We claim that
  $F$ maps 
  $\Delta_*(\mathcal{J})$ 
  (or rather a lift of this object to 
  $C_\Coh^{+,b}(\InjSh(X \times X))$)
  to (an object isomorphic to) $A^\opp.$

  Let 
  $\Delta_*(\mathcal{J}) \ra \mathcal{K}$ be a quasi-isomorphism
  with $\mathcal{K}$ a bounded below complex of injective
  quasi-coherent sheaves on $X \times X,$
  so $\mathcal{K} \in C_\Coh^{+,b}(\InjSh(X \times X))$
  by 
  Theorem~\ref{t:injective-in-qcoh-vs-all-OX}.\ref{enum:inj-qcoh-equal-inj-OX-that-qcoh}.
  Let $p,$ $q \colon X \times X \ra X$ be first and second
  projection. 
  We have canonical identifications and quasi-isomorphisms of dg
  $A^\opp \otimes A$-modules
  \begin{align*}
    F(\mathcal{K}) 
    & = \Hom_{C(\Sh(X \times
      X))} (p^*\mathcal{D}_\mathcal{J}(\mathcal{C}_!(G)) \otimes
    q^*\mathcal{C}_!(G), \mathcal{K})\\
    & = \Hom_{C(\Sh(X \times X))}(p^*\mathcal{D}_\mathcal{J}(\mathcal{C}_!(G)),
    \sheafHom(q^*\mathcal{C}_!(G), \mathcal{K})) && \text{(by
      adjunction)}\\
    & = \Hom_{C(\Sh(X))}(\mathcal{D}_\mathcal{J}(\mathcal{C}_!(G)),
    p_*(\sheafHom(q^*\mathcal{C}_!(G), \mathcal{K}))) && \text{(by
      adjunction)}\\
    & 
    \la
    \Hom_{C(\Sh(X))}(\mathcal{D}_\mathcal{J}(\mathcal{C}_!(G)),
    p_*(\sheafHom(q^*(\mathcal{C}_!(G)), \Delta_*\mathcal{J})))
    && \text{(homotopy equiv.\ by Cor.~\ref{c:obtain-homotopy-equiv})}\\ 
    & = \Hom_{C(\Sh(X))}(\mathcal{D}_\mathcal{J}(\mathcal{C}_!(G)),
    p_*\Delta_*(\sheafHom(\Delta^*q^*(\mathcal{C}_!(G)),
    \mathcal{J})))
    && \text{(by adjunction)}\\
    & = \Hom_{C(\Sh(X))}(\mathcal{D}_\mathcal{J}(\mathcal{C}_!(G)),
    \sheafHom(\mathcal{C}_!(G),
    \mathcal{J})) && \\
    & =
    \Hom_{C(\Sh(X))}(\mathcal{D}_\mathcal{J}(\mathcal{C}_!(G)),
    \mathcal{D}_\mathcal{J}(\mathcal{C}_!(G))) 
    && \text{(by definition)}\\
    & 
    \la
    A^\opp
    && \text{(quasi-isom.\ by 
      \eqref{eq:A-D-lower-shriek-qiso-dga}).}
  \end{align*}
  This proves our claim and shows that $A^\opp \in \per(A^\opp
  \otimes A).$ Hence $A^\opp$ is smooth over $k.$ This is
  equivalent to $A$ being smooth over $k$ (see e.\,g.\
  \cite[Remark~3.11]{valery-olaf-smoothness-equivariant}).  
\end{proof}

\section{Properness of categories and schemes}
\label{sec:prop-categ-schem}

\begin{lemma}
  \label{l:affine-curve}
  Let $X$ be a separated scheme of finite type over a field $k$
  which is not proper over $k.$ Then there is 
  an 
  affine closed curve $C \subset X.$
  (By a curve we mean a separated scheme of finite type over the
  field $k$ which is integral and of dimension one.)
  In particular, the $k$-vector space
  $\Hom_{\mathcal{O}_X}(\mathcal{O}_X,
  \mathcal{O}_C)=\Gamma(C;\mathcal{O}_C)$ is
  infinite-dimensional. 
\end{lemma}

\begin{proof}
  If $C$ is an affine curve then certainly $\dim_k
  \Gamma(C; \mathcal{O}_C)=\infty$ because otherwise $\Gamma(C,
  \mathcal{O}_C)$ would be an Artinian ring.

  By
  \cite[Prop.~12.58.(5), Prop.~12.59]{goertz-wedhorn-AGI}
  we can assume that $X$ is integral (and not proper over $k$).

  The case $\dim X=0$ can not occur:
  if $\dim X=0$ we can view $X$ as a
  closed subscheme of some $\DA^n_k,$ 
  by \cite[Prop.~5.20]{goertz-wedhorn-AGI},
  and all points of $X$ are
  closed points of $\DA^n_k$; then the composition 
  $X \subset \DA^n_k \subset \DP^n_k$ is a closed immersion so
  that $X$ is even projective over $k$ in contradiction to our
  assumption. 


  If $\dim X=1$ take $C=X$ which is an affine closed curve
  by
  \cite[\href{http://stacks.math.columbia.edu/tag/0A28}{Lemma~0A28}]{stacks-project}. 

  Assume that $\dim X > 1.$
  By Nagata's compactification theorem
  (see \cite{luetkeboehmert-compactification-nagata,
    conrad-compactification-nagata},
  \cite[Thm.~12.70]{goertz-wedhorn-AGI}) we can assume that $X$
  is an open dense subscheme of a scheme $Y$ which is proper over
  $k.$ In particular, $Y$ is separated and of finite type over
  $k.$ By replacing $Y$ with $Y_\reduced$ (which still contains
  $X$ as an open dense subscheme) we can assume that $Y$ is
  integral.  

  Since $X$ is not proper over $k$ we have $X \subsetneq Y,$
  so there is a closed point
  $y \in Y \setminus X,$ by
  \cite[\href{http://stacks.math.columbia.edu/tag/005E}{Lemma~005E}]{stacks-project}. 

  If $\{y\}$ is an irreducible component of $Y \setminus X$ then
  $y$ is contained in an irreducible closed subset $D
  \subset Y$ of dimension one. Then 
  $D \cap X$ is a non-empty irreducible closed subset of $X,$
  and we can take $C$ to be $D \cap X$ equipped with the induced
  reduced scheme 
  structure, by
  \cite[\href{http://stacks.math.columbia.edu/tag/0A28}{Lemma~0A28}]{stacks-project},
  because $C$ is not proper over $k$ (otherwise,
  since $D$ is
  separated over $k,$
  the inclusion $C \subset D$ 
  would be proper, and in particular closed, 
  so that $C$ would be a non-empty closed and open strict subset
  of the 
  irreducible set $D$).

  Otherwise all irreducible components $C_1, \dots, C_n$ of $Y
  \setminus X$ containing $y$ are strictly bigger than $\{y\}.$
  Let $\Spec A = U \subset Y$ be an affine open neighborhood
  of $y,$ and let $\mfm \subset A$ be the maximal ideal
  corresponding to $y.$ Let $\mfp_i \in \Spec A$ be the prime
  ideal corresponding to $C_i \cap U.$
  Then $\mfp_i \subsetneq \mfm$ for all
  $i=1,\dots, n.$ By \cite[Prop.~1.11.(i)]{atiyah-macdonald} there
  is an element $f \in \mfm$ with $f \not\in \bigcup_{i=1}^n
  \mfp_i.$ Then $\mathcal{V}(f) \subset U$ is equi-codimensional
  of codimension one in $U,$ by
  \cite[Thm.~5.32]{goertz-wedhorn-AGI}, and contains $y.$
  Let $D$ be an irreducible component of $\mathcal{V}(f)$
  containing $y.$
  Then $\dim D=\dim U-1=\dim Y -1=\dim X-1,$ and $D$ contains a
  point of 
  $X$: otherwise $D \subset Y 
  \setminus X,$ so $D \subset C_i$ for some $i,$ so
  $D=C_i \cap U$ because $\dim D=\dim X-1 \geq \dim (Y \setminus
  X) \geq 
  \dim C_i=\dim (C_i \cap U),$ hence $C_i \cap U=D \subset \mathcal{V}(f)$ and we obtain
  the contradiction $f \in \mfp_i.$

  Let $\ol{D}$ be the closure of $D$ in $Y$ equipped with the
  induced 
  reduced scheme structure. Then $\ol{D} \cap X$ is non-empty and
  closed in $X,$
  has dimension 
  $\dim \ol{D} = \dim D =\dim X-1,$ and is not proper
  over $k$ 
  because otherwise the open immersion $\ol{D} \cap X
  \subsetneq \ol{D}$ would be closed.
  By induction we find an affine closed curve $C
  \subset \ol{D} \cap X.$
\end{proof}

\begin{definition}
  \label{d:tricat-proper}
  Let $k$ be a field and $\mathcal{T}$ a triangulated $k$-linear
  category.  
  We say that $\mathcal{T}$ is
  \define{proper over $k$} if $\mathcal{T}$ has a classical
  generator and 
  \begin{equation*}
    \dim_k \Big(\bigoplus_{n \in \DZ} \Hom_\mathcal{T}(E,[n]F)\Big) < \infty
  \end{equation*}
  for all objects $E,$ $F \in \mathcal{T}.$
  We say that $\mathcal{T}$ is
  \define{locally proper over $k$} if 
  the second condition holds.
\end{definition}

\begin{remark}
  \label{rem:tricat-proper-enhancement}
  The terminology of Definition~\ref{d:tricat-proper} is
  motivated by the corresponding terminology for dg
  $k$-categories, see e.\,g.\
  \cite[Def.~2.11]{valery-olaf-matrix-factorizations-and-motivic-measures}.  If $k$
  is a field and $\mathcal{T}$ a triangulated $k$-linear category
  admitting a $k$-linear enhancement, then $\mathcal{T}$ is
  locally proper (resp.\ proper) over $k$ if and only if a/any
  $k$-linear enhancement of $\mathcal{T}$ is locally proper
  (resp.\ proper) over $k$ as a dg $k$-category (use
  \cite[Prop.~2.18]{valery-olaf-matrix-factorizations-and-motivic-measures},
  the fact that a dg $k$-algebra is locally proper if and only if
  it is proper, and that, given a pretriangulated dg category
  $\mathcal{E}$ having a compact generator, its homotopy category
  $[\mathcal{E}]$ has a classical generator).
\end{remark}

\begin{remark}
  \label{rem:perf-proper}
  Let $X$ be a quasi-compact separated scheme over a field $k.$
  Then $\mfPerf(X)$ has a classical generator,
  by \cite[2.1, 3.1]{bondal-vdbergh-generators}, so
  $\mfPerf(X)$ is proper over $k$ if and only if 
  it is locally proper over $k.$
  Similarly, 
  if $X$ is 
  a separated scheme of finite type over a
  field $k,$ 
  then $D^b(\Coh(X))$ has a
  classical generator by
  Theorem~\ref{t:generator-of-DbCoh-finite-type-any-field},
  so $D^b(\Coh(X))$ is proper over $k$ if and only if it is
  locally proper over $k.$
\end{remark}



\begin{theorem} 
  [{Homological versus geometric properness}]
  \label{t:scheme-proper-iff-Perf-proper}
  %
  Let $X$ be a separated scheme of finite type over a field
  $k.$
  If $X$ is proper over $k$, then $\mfPerf(X)$ is proper over
  $k.$ 
  If $X$ has the resolution property, the converse is also true.
\end{theorem}

\begin{proof}
  Assume that $X$ is proper over $k.$ 
  For $E, F \in \mfPerf(X)$ (or even $F \in D^b(\Coh(X)) \sira
  D^b_\Coh(\Qcoh(X))$) 
  and $n 
  \in \DZ,$ the dimension of
  \begin{equation*}
    \Hom_{D(\Qcoh(X))}(E,[n]F) \cong
    \Hom_{D(\Qcoh(X))}(\mathcal{O}_X, [n]\bR\sheafHom(E, F))
    \cong H^n(X; \bR\sheafHom(E, F))
  \end{equation*}
  is finite since $X$ is proper over $k$
  and $\bR\sheafHom(E, F) \in D^b_\Coh(\Qcoh(X))$,
  by \cite[Thm.~3.2.1]{EGAIII-i}. Moreover, since $X$ is
  quasi-compact, 
  $H^n(X; \bR\sheafHom(E, F))$ is nonzero for at most finitely
  many $n \in \DZ$
  (by Lemma~\ref{l:fin-coho-dim}).   
  This implies that $\mfPerf(X)$
  is proper over $k.$

  Assume now that $X$ has the resolution property and is not
  proper over $k.$ 
  Lemma~\ref{l:affine-curve} shows that $X$ contains an affine
  closed curve $C,$ and in particular $\dim_k \Gamma(C;
  \mathcal{O}_C)=\infty.$ 
  Let $i \colon C \subset X$ be the inclusion. Since $X$ has the
  resolution property, $i_*\mathcal{O}_C$ has a resolution $\dots
  \ra P^{-1} \ra P^0 \ra i_* \mathcal{O}_C \ra 0$ by vector
  bundles. Let $\mathcal{U}=(U_s)_{s \in S}$ be an ordered
  finite affine open covering of $X.$
  Given any complex $A$ of quasi-coherent sheaves on $X$ we have
  $\Gamma(X;\mathcal{C}_*(A)) \cong \bR\Gamma(X; A)$ by
  \cite[Lemma~12.4.(b)]{wolfgang-olaf-locallyproper}
  because all components of $\mathcal{C}_*(A)$ are acyclic with
  respect to the functor $\Gamma(X;-)\colon \Qcoh(X) \ra
  \Qcoh(\Spec k)$
  of finite cohomological
  dimension
  (by Lemmata~\ref{l:u'-ls-qls-acyclic} and
  \ref{l:fin-coho-dim}).   
  Let $Q$ be obtained from $P$ by replacing all components in
  degrees $< -|S|$ by zero. Then the complexes
  $\mathcal{C}_*(Q)$ and $\mathcal{C}_*(P)$ coincide in all degrees
  $\geq -1$, and we obtain
  \begin{multline*}
    \Hom_{D(\Qcoh(X))}(\mathcal{O}_X, Q)
    \cong
    H^0 (\bR\Gamma(X; Q))
    \cong
    H^0 (\Gamma(X; \mathcal{C}_*(Q))
    =
    H^0 (\Gamma(X; \mathcal{C}_*(P))\\
    \cong
    H^0 (\bR\Gamma(X; P))
    \cong
    H^0 (\bR\Gamma(X; i_*\mathcal{O}_C)
    \cong \Gamma(X; i_*\mathcal{O}_C)
    \cong \Gamma(C; \mathcal{O}_C).
  \end{multline*}
  Since $\mathcal{O}_X$ and $Q$ are in $\mfPerf(X)$ and
  $\dim_k \Gamma(C; \mathcal{O}_C)=\infty$ we see that
  $\mfPerf(X)$ is not proper over $k.$
\end{proof}

\begin{theorem} 
  \label{t:scheme-proper+reg-iff-DbCoh-proper}
  Let $X$ be a separated scheme $X$ of finite type over a field
  $k.$
  Then $D^b(\Coh(X))$ is proper over $k$ if and only if $X$ is
  proper over $k$ and regular.
\end{theorem}

\begin{proof}
  If $X$ is regular and proper, then $D^b(\Coh(X))=\mfPerf(X)$ 
  by Proposition~\ref{p:regular-vs-singularity-cat},
  and Theorem~\ref{t:scheme-proper-iff-Perf-proper} shows that
  this category is proper over $k.$

  Let $X$ be 
  not proper. 
  Then,  by Lemma~\ref{l:affine-curve}, 
  there is an affine closed curve $i\colon C \subset X$ with
  $\dim_k \Hom_{D(\Qcoh(X))}(\mathcal{O}_X, i_*\mathcal{O}_C)=
  \infty.$
  Since $\mathcal{O}_X$ and $i_*\mathcal{O}_C$ are coherent
  sheaves on $X$ this
  shows that $D^b(\Coh(X))$ is not proper.

  Now assume that $X$ is not regular. Then there is a non-regular
  closed point $x \in X.$ Then
  $\kappa(x)=\mathcal{O}_{X,x}/\mfm_{x}$ has infinite
  projective dimension as an $\mathcal{O}_{X,x}$-module. This
  implies that
  $\Ext_{\mathcal{O}_{X,x}}^n(\kappa(x), \kappa(x))\not=0$ for
  all $n \in \DN$ (apply $\Hom_{\mathcal{O}_{X,x}}(-,\kappa(x))$
  to a minimal free resolution of $\kappa(x)$).
  Equip $\{x\}=\ol{\{x\}}$ with the induced reduced scheme
  structure and let 
  $i\colon \{x\} \ra X$ be the closed embedding.
  It factors as the composition $\{x\}
  \xra{c} \Spec \mathcal{O}_{X,x} \xra{l} 
  X$ of affine morphisms.
  We can view $\kappa(x)=\mathcal{O}_{X,x}/\mfm_x$ as a coherent
  sheaf on 
  $\Spec \mathcal{O}_{X,x}$ and as a coherent sheaf on 
  $\{x\}.$ Hence
  $l_*(\kappa(x))=l_*(c_*(\kappa(x)))=i_*(\kappa(x)) \in \Coh(X).$
  For $n \in \DZ$ the adjunction $(l^*=\bL l^*, l_*=\bR l_*)$
  yields
  \begin{equation*}
    \Hom_{D(\Qcoh(X))}(l_*(\kappa(x)), [n]l_*(\kappa(x)))
    \cong \Ext_{\mathcal{O}_{X,x}}^n(\kappa(x), \kappa(x)).
  \end{equation*}
  This shows that 
  $D^b(\Coh(X))$ is not proper over $k.$
\end{proof}

\section{Fourier-Mukai functors}
\label{sec:four-mukai-funct}

Our aim in this section is to prove
Theorem~\ref{t:translate-fm-to-dg}.
Recall that for a quasi-compact separated scheme $X$ the category
$D(\Qcoh(X)) \sira D_\Qcoh(\Sh(X))$ is generated by a single
perfect object (see
\cite[Thm.~3.1.1]{bondal-vdbergh-generators}) and that
the
dg subcategory $C^\hinj(\Qcoh(X))$ of $C(\Qcoh(X))$ consisting of
h-injective objects is an enhancement of $D(\Qcoh(X))$ (see
section~\ref{sec:inject-enhanc}).  

\begin{proposition}
  \label{p:DQcohX-via-!-and-*-dg-algebra}
  Let $X$ be a \ref{enum:GSP}-scheme. Fix 
  an 
  ordered finite affine open covering 
  $\mathcal{U}=(U_s)_{s \in S}$ of $X.$
  Let
  $E \in C(\Qcoh(X))$ be a bounded
  complex of vector bundles that is a generator of
  $D(\Qcoh(X)).$ Consider the dg ($\DZ$-)algebras
  $A=\End_{\Cech_*(X)}(E)$ and $B=\End_{\Cech_!(X)}(E).$ 
  Then: 
  \begin{enumerate}
  \item 
    \label{enum:star-equiv}
    The functor
    \begin{equation*}
      \Hom_{C(\Qcoh(X))}(\mathcal{C}_*(E),-) \colon
      [C^\hinj(\Qcoh(X))] \sira D(A)
    \end{equation*}
    is an equivalence of triangulated categories.
  \item 
    \label{enum:shriek-equiv}
    The triangulated functor  
    \begin{equation}
      \label{eq:shriek-equiv-precursor}
        \Hom_{C(\Sh(X))}(\mathcal{C}_!(E),-) \colon [C(\Qcoh(X))]
        \ra D(B)    
    \end{equation}
    maps acyclic objects to zero and factors uniquely to an
    equivalence 
    \begin{equation}
      \label{eq:shriek-equiv}
      \Hom_{C(\Sh(X))}(\mathcal{C}_!(E),-) \colon
      D(\Qcoh(X)) 
      \sira D(B) 
    \end{equation}
    of triangulated categories.
  \end{enumerate}
\end{proposition}

\begin{proof}
  %
  \ref{enum:star-equiv}
  Note that $\mathcal{C}_*(E)$ is a complex of
  quasi-coherent sheaves, and let
  $z\colon \mathcal{C}_*(E) \ra I$ be a quasi-isomorphism with $I
  \in C^\hinj(\Qcoh(X)).$
  Then the first map in  
  \begin{equation}
    \label{eq:A-qiso-via-hinj}
    A=\End_{\Cech_*(X)}(E)
    \xra{\mathcal{C}_*} \End_{C(\Qcoh(X))}(\mathcal{C}_*(E)) 
    \xra{z_*} \Hom_{C(\Qcoh(X))}(\mathcal{C}_*(E), I)
  \end{equation}
  is a morphism of dg algebras and the composition 
  is a quasi-isomorphism: its $n$-th cohomology appears as the
  first row in the following commutative diagram
  \begin{equation*}
    \xymatrix{
      {\Hom_{[\Cech_*(X)]}(E,[n]E)]} \ar[r]^-{\mathcal{C}_*}
      \ar[rd]^-{\mathcal{C}_*}_-{\sim} 
      & 
      {\Hom_{[C(\Qcoh(X))]}(\mathcal{C}_*(E),[n]\mathcal{C}_*(E))}
      \ar[r]^-{z_*}
      \ar[d]
      &
      {\Hom_{[C(\Qcoh(X))]}(\mathcal{C}_*(E), [n]I)}
      \ar[d]^-{\sim}\\
      &
      {\Hom_{D(\Qcoh(X))}(\mathcal{C}_*(E),[n]\mathcal{C}_*(E))}
      \ar[r]^-{z_*}_-{\sim}
      &
      {\Hom_{D(\Qcoh(X))}(\mathcal{C}_*(E), [n]I)}
    }
  \end{equation*}
  whose diagonal arrow is an isomorphism by
  Proposition~\ref{p:abstract-cech-*-object-enhancement}. 
  Since $E$ is a compact generator of 
  $D(\Qcoh(X))$
  (see
  \cite[Thm.~3.1.1]{bondal-vdbergh-generators})
  we see that $I$ is a compact generator of 
  $[C^\hinj(\Qcoh(X))]$
  and can apply 
  Proposition~\ref{p:homotopy-categories-triang-via-dg-algebras}.\ref{hoI-via-B}.

  \ref{enum:shriek-equiv}
  Since $\mathcal{C}_!(E)$ is a
  bounded 
  complex with components finite products of objects
  $\leftidx{_{U_J}}{E^p}{}$, where $\emptyset \not= J \subset T$ 
  and $p \in \DZ$, 
  Lemma~\ref{l:hom-u!-qcoh} implies that
  \begin{equation}
    \label{eq:C!-to-qcoh}
    \Hom_{[C(\Sh(X))]}(\mathcal{C}_!(E), G) \sira
    \Hom_{D(\Sh(X))}(\mathcal{C}_!(E), G)        
  \end{equation}
  is an isomorphism for every $G \in C(\Qcoh(X))$.  This shows
  that the 
  functor
  \eqref{eq:shriek-equiv-precursor} maps acyclic objects to zero
  and hence
  factors uniquely to the functor~\eqref{eq:shriek-equiv} we
  claim to be an equivalence.

  Let $E \ra I$ be a quasi-isomorphism with $I \in
  C^\hinj(\Qcoh(X))$ and let $z$ be the composition
  $\mathcal{C}_!(E) \ra E \ra I$ in $C(\Sh(X))$.
  Then the composition
  \begin{equation}
    \label{eq:B-qiso-via-hinj}
    B=\End_{\Cech_!(X)}(E)
    \xra{\mathcal{C}_!} \End_{C(\Sh(X))}(\mathcal{C}_!(E)) 
    \xra{z_*} \Hom_{C(\Sh(X))}(\mathcal{C}_!(E), I)
  \end{equation}
  is a quasi-isomorphism: this is proved as the fact that
  \eqref{eq:A-qiso-via-hinj} is a quasi-isomorphism, using
  Proposition~\ref{p:abstract-cech-!-object-enhancement} and
  \eqref{eq:C!-to-qcoh}
  (we only use here that $I$ is a 
  complex of quasi-coherent sheaves).
  Similarly, 
  using the equivalence
  $D(\Qcoh(X)) \sira D_\Qcoh(\Sh(X))$ ($X$ being quasi-compact
  separated) and
  \eqref{eq:C!-to-qcoh} again, one proves that
  \begin{equation*}
    z^*\colon
    \Hom_{C(\Qcoh(X))}(I, J)=
    \Hom_{C(\Sh(X))}(I, J) \ra \Hom_{C(\Sh(X))}(\mathcal{C}_!(E), J)   
  \end{equation*}
  is a
  quasi-isomorphism for all $J \in C^\hinj(\Qcoh(X)).$
  Hence
  Proposition~\ref{p:homotopy-categories-triang-via-dg-algebras}.\ref{hoI-via-B}
  shows that the restriction of 
  \eqref{eq:shriek-equiv-precursor} to 
  $[C^\hinj(\Qcoh(X))]$
  is an equivalence.
  This obviously implies that \eqref{eq:shriek-equiv}
  is an equivalence.
\end{proof}

\begin{proposition}
  \label{p:DQcoh-XY-via-!-dg-algebra}
  Let $X$ and $Y$ be Noetherian \ref{enum:GSP}-schemes over a
  field $k$ and assume that $Y \times X$ is also Noetherian.  Let
  $E \in C(\Qcoh(X))$ and $F \in C(\Qcoh(Y))$ be bounded
  complexes of vector bundles that are generators of
  $D(\Qcoh(X))$ and $D(\Qcoh(Y)),$ respectively.  
  Fix ordered finite affine open coverings $\mathcal{U}=(U_s)_{s
    \in S}$ of 
  $X$ and $\mathcal{V}=(V_t)_{t \in T}$ of $Y.$ Consider the 
  dg ($k$-)algebras
  $A=\End_{\Cech_*(X)}(E)$ and $B=\End_{\Cech_!(Y)}(F).$ 
  Then there is a triangulated functor
  \begin{equation}
    \label{eq:shriek-box-shriek-equiv-precursor}
    \Hom_{C(\Sh(Y \times X))}(\mathcal{C}_!(F) \boxtimes
    \mathcal{C}_!(E^\cek),-) \colon [C(\Qcoh(Y \times X))]
    \ra D(B \otimes A^\opp)
  \end{equation}
  which factors uniquely to an equivalence
  \begin{equation}
    \label{eq:shriek-box-shriek-equiv}
    \Hom_{C(\Sh(Y \times X))}(\mathcal{C}_!(F) \boxtimes
    \mathcal{C}_!(E^\cek),-) \colon D(\Qcoh(Y \times X))
    \sira D(B \otimes A^\opp)
  \end{equation}
  of triangulated categories.
\end{proposition}

\begin{proof}
  Recall the duality isomorphism $(-)^\cek\colon \Cech_!(X)^\opp
  \sira \Cech_*(X)$ of dg categories from   
  \eqref{eq:duality-cech-enhancements}.
  It yields an isomorphism 
  \begin{equation*}
    \alpha \colon \End_{\Cech_!(X)}(E^\cek) \sira
    A^\opp=\End_{\Cech_*(X)}(E)^\opp 
  \end{equation*}
  of dg algebras. 
  Let
  $P:=\mathcal{C}_!(F) \boxtimes \mathcal{C}_!(E^\cek).$
  The morphism of dg algebras
  \begin{equation}
    \label{eq:BotimesAopp-EndP}
    \mathcal{C}_! \boxtimes (\mathcal{C}_! \comp \alpha\inv)
    \colon 
    B \otimes A^\opp
    \ra
    \End_{C(\Sh(Y \times X))}(P)
  \end{equation}
  then defines the triangulated functor
  \eqref{eq:shriek-box-shriek-equiv-precursor}.
  Lemmata~\ref{l:hom-u!-qcoh} 
  and
  \ref{l:boxtimes-restriction-extension}.\ref{enum:boxtimes-!-extension}
  show that
  \begin{equation*}
    \Hom_{[C(\Sh(Y \times X))]}(P, G) \sira 
    \Hom_{D(\Sh(Y \times X))}(P, G)        
  \end{equation*}
  is an isomorphism for every $G \in C(\Qcoh(Y \times X))$.  
  This implies that the 
  functor
  \eqref{eq:shriek-box-shriek-equiv-precursor}
  factors uniquely to the
  functor~\eqref{eq:shriek-box-shriek-equiv} 
  we
  claim to be an equivalence.

  Recall that $F \boxtimes E^\cek$ is a compact 
  generator of $D_\Qcoh(\Sh(Y \times X))$ by
  \cite[Lemma~3.4.1]{bondal-vdbergh-generators},
  and so is the quasi-isomorphic object $P.$  
  Let $E^\cek \ra I$ and $F \ra J$ be 
  quasi-isomorphisms
  with $I$
  and $J$ 
  bounded below 
  complexes
  of injective quasi-coherent sheaves on $X$ and $Y$,
  respectively. Let $i$ be the composition $\mathcal{C}_!(E^\cek)
  \ra 
  E^\cek \ra I$ and $j$ the composition $\mathcal{C}_!(F) \ra F
  \ra J.$  
  Then $j \boxtimes i \colon P \ra J \boxtimes I$ is a
  quasi-isomorphism.
  Let
  $\tau\colon J \boxtimes I
  \ra T$ be a quasi-isomorphism with $T$ a
  bounded below complex of injective quasi-coherent
  sheaves. 

  Consider the commutative diagram
  \begin{equation*}
    \xymatrix{
      {B \otimes A^\opp} 
      \ar[d]_-{\mathcal{C}_! \otimes 
        (\mathcal{C}_! \comp \alpha\inv)}
      \ar[rd]^-{\mathcal{C}_! \boxtimes 
        (\mathcal{C}_! \comp \alpha\inv)} 
      \\
      {
        \leftidx{_X}{(\mathcal{C}_!(F),
          \mathcal{C}_!(F))}{} 
        \otimes
        \leftidx{_X}{(\mathcal{C}_!(E^\cek), \mathcal{C}_!(E^\cek))}{}
      }
      \ar[r]_-{\boxtimes}
      \ar[d]^-{j_* \otimes i_*}
      &
      {\leftidx{_{Y \times X}}{(P,P)}{}} 
      \ar[d]^-{(j \boxtimes i)_*}
      \ar[r]^-{(\tau \comp (j \boxtimes i))_*}
      &
      {\leftidx{_{Y \times X}}{(P, T)}{}}
      \ar@{}[d]|-{\verteq}
      \\
      {\leftidx{_X}{(\mathcal{C}_!(F), J)}{}  
        \otimes 
        \leftidx{_X}{(\mathcal{C}_!(E^\cek),I)}{}}
      \ar[r]^-{\boxtimes}
      &
      {\leftidx{_{Y \times X}}{(P, J \boxtimes I)}{}}  
      \ar[r]^-{\tau_*}
      &
      {\leftidx{_{Y \times X}}{(P, T)}{}}
    }
  \end{equation*}
  where 
  we abbreviate
  $\leftidx{_?}{(-,-)}{}:=\Hom_{C(\Sh(?))}(-,-).$
  The proof that the composition in
  \eqref{eq:B-qiso-via-hinj} is a quasi-isomorphism
  implies that both $j_* \comp \mathcal{C}_!$ and $i_* \comp
  \mathcal{C}_!$ are quasi-isomorphisms.
  Hence the left vertical composition $(j_* \comp \mathcal{C}_!)
  \otimes (i_* \comp \mathcal{C}_! \comp \alpha\inv)$ is a
  quasi-isomorphism.
  The composition in the lower row is a quasi-isomorphism by
  Proposition~\ref{p:tensor-product-of-endos-vs-endos-of-boxproduct}. 
  Hence, the composition of the diagonal arrow
  $\mathcal{C}_! \boxtimes (\mathcal{C}_! \comp \alpha\inv)$,
  our morphism \eqref{eq:BotimesAopp-EndP}
  of dg algebras,
  with the upper right horizontal map $(\tau \comp (j \boxtimes
  i))_*$ is a 
  quasi-isomorphism.
  We then proceed as 
  in the proof of
  Proposition~\ref{p:DQcohX-via-!-and-*-dg-algebra}.\ref{enum:shriek-equiv}.
  %
  %
\end{proof}

Let $X$ and $Y$ be quasi-compact separated schemes over a field
$k.$ Let $p \colon Y \times X \ra X$ and $q \colon Y \times X
\ra Y$ be second 
and first projection. 
For $K \in C(\Qcoh(Y \times X))$
let $\phi_K$ be the composition 
\begin{equation*}
  \phi_K\colon 
  C(\Qcoh(X)) \xra{p^*}
  C(\Qcoh(Y \times X)) \xra{- \otimes K}
  C(\Qcoh(Y \times X)) \xra{q_*}
  C(\Qcoh(Y))
\end{equation*}
of dg functors (which is well-defined by
Lemma~\ref{l:qcqs-preserves-qcoh}).
Recall that on a quasi-compact separated scheme 
any complex of quasi-coherent sheaves admits a quasi-isomorphism
from an h-flat complex of quasi-coherent sheaves
\cite[Lemma~8]{murfet-der-cat-qcoh-sheaves}.
The Fourier-Mukai functor with kernel $K$ is defined as the 
composition
\begin{equation*}
  \Phi_K \colon D(\Qcoh(X)) \xra{p^*=\bR p^*}
  D(\Qcoh(Y \times X)) \xra{-\otimes^\bL K}
  D(\Qcoh(Y \times X)) \xra{\bR q_*}
  D(\Qcoh(Y))
\end{equation*}
of triangulated functors.


\begin{proposition}
  \label{p:compute-FM-via-star-cech}
  Let $X$ and $Y$ be quasi-compact separated schemes over a field
  $k.$
  Let $K' \ra K$ be a quasi-isomorphism in $Z^0(C(\Qcoh(Y \times
  X)))$ 
  with $K'$ an h-flat complex
  of quasi-coherent sheaves.
  Let $\mathcal{U}=(U_s)_{s \in S}$
  be an ordered finite affine open covering 
  of $X.$ 
  Then the composition $\can \comp \phi_{K'} \comp \mathcal{C}_*$
  of triangulated functors in the diagram
  \begin{equation*}
    \xymatrix{
      {[C(\Qcoh(X))]} \ar[r]^-{\mathcal{C}_*} \ar[d]^-{\can} &
      {[C(\Qcoh(X))]} \ar[r]^-{\phi_{K'}} &
      {[C(\Qcoh(Y))]} \ar[d]^-{\can} \\
      {D(\Qcoh(X))}
      \ar@{..>}[rr]^-{\ol{\phi_{K'} \comp \mathcal{C}_*}}
      &&
      {D(\Qcoh(Y))} 
    }
  \end{equation*}
  maps acyclic
  complexes to zero and hence factors uniquely
  to the indicated 
  triangulated 
  functor 
  $\ol{\phi_{K'} \comp \mathcal{C}_*}$ making the diagram
  commutative. This functor 
  $\ol{\phi_{K'} \comp \mathcal{C}_*}$ is isomorphic to the
  Fourier-Mukai functor $\Phi_K=\bR q_*(p^*(-) \otimes^\bL K).$
  Moreover, 
  all functors in the above diagram preserve coproducts.
\end{proposition}

\begin{proof}
  Let $p \colon Y \times X \ra X$ and $q \colon Y \times X
  \ra Y$ be the projections. 
  Let $A \in C(\Qcoh(X)).$
  We claim that all components of $p^*(\mathcal{C}_*(A)) \otimes
  K'$ are acyclic with respect to the functor
  $q_*\colon \Qcoh(Y \times X) \ra \Qcoh(Y)$.

  Let $B \in \Qcoh(X)$ and $M \in \Qcoh(Y \times X).$
  If $u\colon U \subset X$ is the inclusion of an affine open
  subset, then
  \begin{equation*}
    p^*(\leftidx{_U}{B}{}) \otimes M 
    \sira
    \leftidx{_{Y \times U}}{(p^*B)}{} \otimes M 
    \sira
    \leftidx{_{Y \times U}}{(p^*B \otimes M)}{}
  \end{equation*}
  by Lemmata~\ref{l:push-affine-then-pull}
  and \ref{l:push-from-open-and-tensor} because $u$ and $u' \colon
  Y \times U \ra Y \times X$ are affine. Since also $q \comp u'$
  is affine, 
  Lemma~\ref{l:u'-ls-qls-acyclic} shows that 
  $\leftidx{_{Y \times U}}{(p^*B \otimes M)}{}$ is acyclic with
  respect to $q_*\colon \Qcoh(Y \times X) \ra \Qcoh(Y).$
  Since arbitrary coproducts of $q_*$-acyclic quasi-coherent
  sheaves are again $q_*$-acyclic, by \cite[Lemma~B.6, Cor.~B.9]{thomason-trobaugh-higher-K-theory},
  each component of
  $p^*(\mathcal{C}_*(A)) \otimes K'$ 
  is $q_*$-acyclic.
  
  Note moreover that $q_*\colon \Qcoh(Y \times X) \ra \Qcoh(Y)$
  has finite
  cohomological dimension
  (by Lemma~\ref{l:fin-coho-dim}).
  Then \cite[Lemma~12.4.(b)]{wolfgang-olaf-locallyproper}
  shows that 
  \begin{equation*}
    q_*(p^*(\mathcal{C}_*(A)) \otimes K')
    \ra
    \bR q_*(p^*(\mathcal{C}_*(A)) \otimes K')
  \end{equation*}
  is an isomorphism in $D(\Qcoh(Y)).$
  Note that 
  \begin{equation*}
    p^*(A) \otimes K' 
    \ra
    p^*(\mathcal{C}_*(A)) \otimes K'
  \end{equation*}
  is a quasi-isomorphism because $p$ is flat and $K'$ is h-flat.
  The object on the left is isomorphic to 
  $p^*(A) \otimes^\bL K$ in $D(\Qcoh(Y \times X)).$
  Applying $\bR q_*$ yields isomorphisms
  \begin{equation*}
    \bR q_*(p^*(A) \otimes^\bL K)
    \cong \bR q_*(p^*(A) \otimes K') 
    \sira
    \bR q_*(p^*(\mathcal{C}_*(A)) \otimes K')
  \end{equation*}
  in $D(\Qcoh(Y)).$
  In particular,
  $q_*(p^*(\mathcal{C}_*(A)) \otimes K')$ is acyclic if $A$ is
  acyclic. 
  This shows that $\ol{\phi_{K'} \comp \mathcal{C}_*}$ exists
  and is isomorphic to $\Phi_K.$
  
  The functors $\can$ in the diagram preserve coproducts by
  \cite[Lemma~1.5]{neeman-homotopy-limits}.
  Lemma~\ref{l:qcqs-preserves-qcoh} shows that the functors $q_*$
  and $u_*$ preserves coproducts, where $u \colon U \subset X$ is
  the inclusion of an affine open subset. This implies that
  $\phi_{K'}= q_* (p^*(-) \otimes K')$, $\mathcal{C}_*$, and
  $\ol{\phi_{K'} \comp \mathcal{C}_*}$ preserve coproducts.
\end{proof}

\begin{theorem}
  [{Fourier-Mukai kernels and dg bimodules}]
  \label{t:translate-fm-to-dg}
  Let $X$ and $Y$ be Noetherian \ref{enum:GSP}-schemes over a
  field $k$ and assume that $Y \times X$ is also Noetherian.  Let
  $E \in C(\Qcoh(X))$ and $F \in C(\Qcoh(Y))$ be bounded
  complexes of vector bundles that are generators of
  $D(\Qcoh(X))$ and $D(\Qcoh(Y)),$ respectively.  
  Fix ordered finite affine open coverings $\mathcal{U}=(U_s)_{s
    \in S}$ of 
  $X$ and $\mathcal{V}=(V_t)_{t \in T}$ of $Y$ and let
  $A=\End_{\Cech_*(X)}(E)$ and $B=\End_{\Cech_!(Y)}(F).$ 
  Denote the equivalences of triangulated categories 
  provided by Propositions~\ref{p:DQcohX-via-!-and-*-dg-algebra}
  and \ref{p:DQcoh-XY-via-!-dg-algebra}
  as follows
  \begin{align}
    \notag
    \theta_X:= \Hom_{C(\Qcoh(X))}(\mathcal{C}_*(E),-) \colon &
    [C^\hinj(\Qcoh(X))] 
    \sira D(A),\\
    \notag
    \theta_Y:= \Hom_{C(\Sh(Y))}(\mathcal{C}_!(F),-) \colon &
    D(\Qcoh(Y)) 
    \sira D(B), \\
    \notag
    \theta_{Y \times X} :=
    \Hom_{C(\Sh(Y \times X))}(\mathcal{C}_!(F) \boxtimes
    \mathcal{C}_!(E^\cek),-) \colon & D(\Qcoh(Y \times X))
    \sira D(B \otimes A^\opp),
  \end{align}
  and let $\can\inv$ be a
  quasi-inverse of the equivalence $\can \colon
  [C^\hinj(\Qcoh(X))] \sira D(\Qcoh(X)).$
  Then for any $K \in D(\Qcoh(Y \times X))$ 
  with corresponding $M=\theta_{Y \times X}(K) \in D(B \otimes
  A^\opp)$ 
  the diagram
  \begin{equation*}
    \xymatrix{
      {D(\Qcoh(X))} 
      \ar[d]_-{\theta_X \comp \can\inv}^-{\sim}
      \ar[rrrr]^-{\Phi_K=\bR q_*(p^*(-) \otimes^{\bL} K)} &&&&
      {D(\Qcoh(Y))} \ar[d]_-{\theta_Y}^-{\sim}\\
      {D(A)} 
      \ar[rrrr]^-{- \otimes^{\bL}_A M}
      &&&&
      {D(B)}
    }
  \end{equation*}
  commutes up to an isomorphism $\tau^K$ of triangulated
  functors. 
\end{theorem}

\begin{proof}
  In this proof we abbreviate
  $\Hom_X=\Hom_{C(\Sh(X))}$
  and
  $\End_X=\End_{C(\Sh(X))},$ and similarly for $Y$ and $Y \times X.$

  Let $K \in D(\Qcoh(Y \times X))$ and let $K' \ra K$ be a
  quasi-isomorphism 
  with $K'$ an h-flat complex
  of quasi-coherent sheaves.
  Consider the composition
  \begin{equation*}
    C(\Qcoh(X)) \xra{\mathcal{C}_*} C(\Qcoh(X)) \xra{\phi_{K'}}
    C(\Qcoh(Y)) \xra{\Hom_Y(\mathcal{C}_!(F),-)} C(B)
  \end{equation*}
  of dg functors where we write $C(B)$ for the dg category of dg
  $B$-modules.
  Applying it to $\mathcal{C}_*(E)$ we define
  \begin{equation*}
    M':=
    \Hom_Y(\mathcal{C}_!(F),\phi_{K'}(\mathcal{C}_*(\mathcal{C}_*(E))))   \end{equation*}
  which is naturally a dg $B \otimes A^\opp$-module where the
  $B$-action is obvious and the $A^\opp$-action comes from the
  morphism
  of dg algebras
  \begin{equation*}
    A=\End_{\Cech_*(X)}(E)
    \xra{\mathcal{C}_*} \End_X(\mathcal{C}_*(E))
    \xra{\phi_{K'} \comp \mathcal{C}_*} 
    \End_Y(\phi_{K'}(\mathcal{C}_*(\mathcal{C}_*(E)))). 
  \end{equation*}
  
  Consider the diagram
  \begin{equation*}
    \xymatrix{
      {[C^\hinj(\Qcoh(X))]} \ar[r]^-{\iota}
      \ar[rd]^-{\sim}_-{\can}
      \ar[dd]_-{\theta_X=\Hom_X(\mathcal{C}_*(E),-)}^-{\sim} &
      {[C(\Qcoh(X))]} \ar[r]^-{\mathcal{C}_*} \ar[d] &
      {[C(\Qcoh(X))]} \ar[r]^-{\phi_{K'}} &
      {[C(\Qcoh(Y))]} \ar[d] \\
      & 
      {D(\Qcoh(X))}
      \ar[rr]^-{\ol{\phi_{K'} \comp \mathcal{C}_*}}
      &&
      {D(\Qcoh(Y))} 
      \ar[d]_-{\theta_Y=\Hom_Y(\mathcal{C}_!(F),-)}^-{\sim} \\
      {D(A)} \ar[rrr]^-{-\otimes^\bL_A M'} 
      &&& {D(B)}
    }
  \end{equation*}
  of triangulated functors.
  The triangle and the little rectangle in this diagram are
  commutative, by Proposition~\ref{p:compute-FM-via-star-cech}.
  Let 
  \begin{equation*}
    \sigma^{K'} \colon
    (- \otimes_A^\bL M') \comp \theta_X \ra
    \theta_Y \comp \ol{\phi_{K'} \comp \mathcal{C}_*} \comp \can
  \end{equation*}
  be the morphism of triangulated
  functors $[C^\hinj(\Qcoh(X))] \ra D(B)$  
  between the two outer
  paths from the top left corner to the bottom right corner
  given on objects as follows: for $J \in [C^\hinj(\Qcoh(X))]$ let
  $\sigma^{K'}_J$ be the composition
  \begin{equation*}
    \Hom_X(\mathcal{C}_*(E), J) \otimes_A^\bL M'
    \ra
    \Hom_X(\mathcal{C}_*(E), J) \otimes_A M'  
    \ra
    \Hom_Y(\mathcal{C}_!(F),
    \phi_{K'}(\mathcal{C}_*(J)))
  \end{equation*}
  with obvious first map and second map given by $f
  \otimes m \mapsto \phi_{K'}(\mathcal{C}_*(f)) \comp m.$

  We claim that $\sigma^{K'}$ is an isomorphism of functors.
  Note that 
  source $(- \otimes_A^\bL M') \comp \theta_X$ and target
  $\theta_Y \comp \ol{\phi_{K'} \comp \mathcal{C}_*} \comp \can$
  of $\sigma^{K'}$ commute with coproducts, by
  Proposition~\ref{p:compute-FM-via-star-cech}
  (however, this is not true for the functor $\iota$ in the above
  diagram, cf.\ Remark~\ref{rem:coprod-h-inj}).
  Hence it is enough to show that $\sigma^{K'}$ evaluates to an
  isomorphism at an arbitrary compact generator of 
  $[C^\hinj(\Qcoh(X))].$

  As in the proof of Proposition~\ref{p:DQcohX-via-!-and-*-dg-algebra}.\ref{enum:star-equiv},
  the complex $\mathcal{C}_*(E)$ of
  quasi-coherent sheaves admits a quasi-isomorphism
  $z\colon \mathcal{C}_*(E) \ra I$ to an object $I
  \in C^\hinj(\Qcoh(X)).$ Then $I$ 
  is a compact generator
  of
  $[C^\hinj(\Qcoh(X))]$ and we need to show that $\sigma^{K'}_I$
  is an isomorphism.
  Recall that the composition in \eqref{eq:A-qiso-via-hinj}
  is a quasi-isomorphism. If we use it to compute the left
  derived tensor product, the morphism $\sigma^{K'}_I$ is given by the
  map
  \begin{equation*}
    \sigma_I^{K'} \colon \Hom_Y(\mathcal{C}_!(F),
    \phi_{K'}(\mathcal{C}_*(\mathcal{C}_*(E))))=M'
    =
    A \otimes_A M' 
    \ra
    \Hom_Y(\mathcal{C}_!(F),
    \phi_{K'}(\mathcal{C}_*(I)))
  \end{equation*}
  which maps an element $m \in M'$ to $\phi_{K'}(\mathcal{C}_*(z))
  \comp m.$
  Note that 
  \begin{equation*}
    \phi_{K'}(\mathcal{C}_*(z))\colon
    \phi_{K'}(\mathcal{C}_*(\mathcal{C}_*(E))) \ra
    \phi_{K'}(\mathcal{C}_*(I))     
  \end{equation*}
  is a quasi-isomorphism between complexes of quasi-coherent
  sheaves, by
  Proposition~\ref{p:compute-FM-via-star-cech}, and that
  $\Hom_Y(\mathcal{C}_!(F),-)$ maps such quasi-isomorphisms 
  to quasi-isomorphisms of dg $B$-modules, by
  Proposition~\ref{p:DQcohX-via-!-and-*-dg-algebra}.\ref{enum:shriek-equiv}. This
  shows that $\sigma_I^{K'}$ is an isomorphism in $D(B)$
  and proves that $\sigma^{K'}$ is an isomorphism of functors.
  
  Since we already know that
  $\ol{\phi_{K'} \comp \mathcal{C}_*}$ and $\Phi_K=\bR q_*(p^*(-)
  \otimes^\bL K)$ are isomorphic, by 
  Proposition~\ref{p:compute-FM-via-star-cech},
  it remains to show that
  $\theta_{Y \times X}(K) \cong M'$ in $D(B \otimes A^\opp).$

  Observe that we have isomorphisms of dg $B
  \otimes A^\opp$-modules
  \begin{align*}
    \theta_{Y \times X}(K')= &
    \Hom_{Y \times X}(\mathcal{C}_!(F) \boxtimes
    \mathcal{C}_!(E^\cek),K')\\
    = &
    \Hom_{Y \times X}(q^*(\mathcal{C}_!(F)) \otimes p^*(\mathcal{C}_!(E^\cek)),K') 
    \\  
    = &
    \Hom_{Y \times X}(q^*(\mathcal{C}_!(F)),
    \sheafHom(p^*(\mathcal{C}_!(E^\cek)),K')) & \text{(by
      adjunction)}\\  
    = &
    \Hom_{Y}(\mathcal{C}_!(F),
    q_*(\sheafHom(p^*(\mathcal{C}_!(E^\cek)),K'))) & \text{(by
      adjunction)}\\  
    \sila &
    \Hom_{Y}(\mathcal{C}_!(F),
    q_*(p^*(\mathcal{C}_*(E)) \otimes K')) & \text{(by
      Lemma~\ref{l:C*-tensor-isom-sHomC!dual})}\\  
    = &
    \Hom_{Y}(\mathcal{C}_!(F),
    \phi_{K'}(\mathcal{C}_*(E))).
  \end{align*}
  Let $\rho_{\mathcal{C}_*(E)} \colon \mathcal{C}_*(E) \ra
  \mathcal{C}_*(\mathcal{C}_*(E))$ be the canonical
  quasi-isomorphism 
  from $\mathcal{C}_*(E)$ to its $*$-\v{C}ech resolution (see
  \eqref{eq:F-*-Cech-resolution}). 
  The proof of Proposition~\ref{p:compute-FM-via-star-cech}
  shows that all components of $p^*(\mathcal{C}_*(E)) \otimes K'$
  and
  of $p^*(\mathcal{C}_*(\mathcal{C}_*(E))) \otimes K'$
  are acyclic with respect to the functor $q_*\colon \Qcoh(Y
  \times X) \ra 
  \Qcoh(Y)$ of finite cohomological dimension. 
  Since 
  $p^*(\rho_{\mathcal{C}_*(E)}) \otimes \id_{K'}$
  is a quasi-isomorphism,
  \cite[Lemma~12.4.(b)]{wolfgang-olaf-locallyproper}
  shows that its $q_*$-image
  $\phi_{K'}(\rho_{\mathcal{C}_*(E)})$ is a quasi-isomorphism
  between complexes of quasi-coherent sheaves. As above, 
  $\Hom_Y(\mathcal{C}_!(F),-)$ preserves such
  quasi-isomorphisms, so that we get a quasi-isomorphism
  \begin{equation*}
    \Hom_{Y}(\mathcal{C}_!(F),
    \phi_{K'}(\mathcal{C}_*(E)))
    \xra{\phi_{K'}(\rho_{\mathcal{C}_*(E)})
      \comp ?}  
    \Hom_{Y}(\mathcal{C}_!(F),
    \phi_{K'}(\mathcal{C}_*(\mathcal{C}_*(E)))) 
    =M'
  \end{equation*}
  of dg $B$-modules. Let us argue that it is in fact a
  morphism of dg $B 
  \otimes A^\opp$-modules
  where the $A^\opp$-action on the left-hand side comes from the
  morphism 
  of dg algebras
  \begin{equation*}
    A=\End_{\Cech_*(X)}(E)
    \xra{\mathcal{C}_*} \End_X(\mathcal{C}_*(E))
    \xra{\phi_{K'}} 
    \End_Y(\phi_{K'}(\mathcal{C}_*(E))). 
  \end{equation*}
  Given any 
  $a \in A,$ applying the morphism $\rho:\id \ra
  \mathcal{C}_*$ of functors to the morphism $\mathcal{C}_*(a) \colon
  \mathcal{C}_*(E) \ra 
  \mathcal{C}_*(E)$ shows that
  $\rho_{\mathcal{C}_*(E)} \comp \mathcal{C}_*(a) =
  \mathcal{C}_*(\mathcal{C}_*(a)) \comp \rho_{\mathcal{C}_*(E)}.$
  We then obtain for any
  $f \in \Hom_{Y}(\mathcal{C}_!(F), \phi_{K'}(\mathcal{C}_*(E)))$
  that
  \begin{multline*}
    \phi_{K'}(\rho_{\mathcal{C}_*(E)}) \comp \phi_{K'}(\mathcal{C}_*(a))
    \comp f  
    =\phi_{K'}(\rho_{\mathcal{C}_*(E)} \comp \mathcal{C}_*(a)) \comp 
    f  
    =\phi_{K'}(\mathcal{C}_*(\mathcal{C}_*(a)) \comp
    \rho_{\mathcal{C}_*(E)}) \comp f\\
    =\phi_{K'}(\mathcal{C}_*(\mathcal{C}_*(a))) \comp
    \phi_{K'}(\rho_{\mathcal{C}_*(E)}) \comp f.  
  \end{multline*}
  
  Combining the above (quasi-)isomorphisms of dg $B \otimes
  A^\opp$-modules yields that $\theta_{Y \times X}(K) \cong
  \theta_{Y \times X}(K')$ and 
  $M'$ are isomorphic in $D(B \otimes A^\opp)$.
\end{proof}

\begin{lemma}
  \label{l:C*-tensor-isom-sHomC!dual}
  Let $X$ and $Y$ be schemes over a field $k.$ 
  Assume that $X$ is
  quasi-compact separated
  and fix an ordered finite affine open
  covering $\mathcal{U}=(U_s)_{s 
    \in S}$ of 
  $X.$ 
  Let $L$ be a complex of 
  quasi-coherent 
  sheaves on $Y \times X.$
  Then there is an isomorphism
  \begin{equation*}
    p^*(\mathcal{C}_*((-)^\cek))
    \otimes
    L
    \sira
    \sheafHom(p^*(\mathcal{C}_!(-)),L) 
  \end{equation*}
  of dg functors
  \begin{equation*}
    \Cech_!(X)^\opp \ra C(\Sh(Y \times X)).
  \end{equation*}
  In particular, if 
  $E$ is a bounded complex of vector bundles on $X,$ plugging
  in $E^\cek$ yields an isomorphism
  \begin{equation*}
    \gamma \colon p^*(\mathcal{C}_*(E))
    \otimes
    L
    \sira
    \sheafHom(p^*(\mathcal{C}_!(E^\cek)),L) 
  \end{equation*}
  in $Z^0(C(\Qcoh(Y \times X)))$
  which is compatible with the left action of the dg algebra
  \begin{equation*}
    \End_{\Cech_*(X)}(E) \xsila{(-)^\cek} \End_{\Cech_!(X)}(E^\cek)^\opp  
  \end{equation*}
  in the sense that the diagram
  \begin{equation*}
    \xymatrix{
      {\End_{\Cech_*(X)}(E)}
      \ar[d]^-{p^*(\mathcal{C}_*(-)) \otimes L} &&
      {\End_{\Cech_!(X)}(E^\cek)^\opp}
      \ar[ll]_-{(-)^\cek}^-{\sim}
      \ar[d]^-{\sheafHom(p^*(\mathcal{C}_!(-)),L)}\\
      {\End_{C(\Sh(Y \times X))}(p^*(\mathcal{C}_*(E)) \otimes
        L)} \ar[rr]^-{\gamma \comp ? \comp \gamma\inv}_-{\sim} &&
      {\End_{C(\Sh(Y \times X))}(\sheafHom(p^*(\mathcal{C}_!(E^\cek)),L))} 
    }
  \end{equation*}
  of dg algebras commutes.
\end{lemma}

\begin{proof}
  This follows from
  Corollary~\ref{c:compatible} because $X$ is separated and all
  the inclusions $U_I \subset X$ for $I \subset S$ are affine.
\end{proof}

\begin{remark}
  \label{rem:coprod-h-inj}
  Let $X$ be a Noetherian separated scheme.
  Since $\can \colon [C^\hinj(\Qcoh(X))] \ra D(\Qcoh(X))$ is
  an equivalence, 
  $[C^\hinj(\Qcoh(X))]$ has arbitrary coproducts. 
  In general, however, it is 
  not true
  that the inclusion functor $\iota \colon [C^\hinj(\Qcoh(X))]
  \ra [C(\Qcoh(X))]$ commutes with coproducts.
  Surprisingly, this is related to regularity of $X$ as explained
  in (the proof of) the following
  Proposition~\ref{p:coprod-h-inj}. 
\end{remark}

\begin{proposition}
  \label{p:coprod-h-inj}
  Let $X$ be a Noetherian separated scheme. Then the following
  conditions are equivalent where coproducts are formed in 
  $Z^0(C(\Qcoh(X)))$.
  \begin{enumerate}
  \item 
    \label{enum:coprod-hinj}
    all coproducts of h-injective complexes of
    quasi-coherent sheaves 
    are h-injective; 
  \item 
    \label{enum:coprod-fibrant}
    all coproducts of fibrant complexes of
    quasi-coherent sheaves 
    are fibrant;
  \item 
    \label{enum:Perf-DbCoh}
    $\mfPerf(X)=D^b(\Coh(X)).$
  \end{enumerate}
  These conditions imply regularity of $X.$ If $X$ is of finite
  dimension, they are equivalent 
  to regularity of $X.$
\end{proposition}

The main argument of the following proof is due to Henning Krause.

\begin{proof}
  Let $C^\fib(\Qcoh(X))$ denote the full dg subcategory of
  $C(\Qcoh(X))$ of fibrant objects. 
  We use results explained in Remark~\ref{rem:fibrant-explained}.

  \ref{enum:coprod-hinj} $\Rightarrow$ \ref{enum:coprod-fibrant}:
  Let $(F_l)_{l \in L}$ be a family of fibrant complexes of
  quasi-coherent sheaves. Since all $F_l$ are h-injective and
  degreewise injective quasi-coherent,
  $\bigoplus F_l$ is h-injective by assumption and degreewise
  injective quasi-coherent by
  Theorem~\ref{t:injective-in-qcoh-vs-all-OX}.\ref{enum:inj-and-inj-qcoh-arbitrary-sums}. Hence
  $\bigoplus F_l$ is fibrant.

  \ref{enum:coprod-fibrant} $\Rightarrow$ \ref{enum:coprod-hinj}:
  Let $(I_l)_{l \in L}$ be a family of h-injective complexes of
  quasi-coherent sheaves. 
  Choose quasi-isomorphisms $I_l \ra F_l$ with $F_l$
  fibrant 
  complexes of quasi-coherent sheaves. 
  These morphisms are homotopy
  equivalences because all $I_l$ and $F_l$ are h-injective. Hence
  $\bigoplus I_l 
  \ra \bigoplus F_l$ is a homotopy equivalence. Since $\bigoplus
  F_l$ is fibrant by assumption, in particular h-injective,
  $\bigoplus I_l$ is h-injective.

  \ref{enum:coprod-hinj} $\Leftrightarrow$ \ref{enum:Perf-DbCoh}:
  We first reformulate \ref{enum:coprod-hinj}.
  Let $\iota \colon [C^\hinj(\Qcoh(X))] \ra [C(\Qcoh(X))]$ and
  $\iota' \colon [C^\fib(\Qcoh(X))] \ra [C(\Qcoh(X))]$ denote the
  inclusions. Since $Z^0(C(\Qcoh(X))) \ra [C(\Qcoh(X)]$ preserves
  coproducts, condition~\ref{enum:coprod-hinj} holds if and only
  if $\iota$ preserves coproducts, and this is the case if and
  only if $\iota'$ preserves coproducts because
  $[C^\fib(\Qcoh(X))] \sira [C^\hinj(\Qcoh(X))]$ is an
  equivalence.  Moreover, $[C(\InjQcoh(X))]$ is cocomplete 
  (by
  Theorem~\ref{t:injective-in-qcoh-vs-all-OX}.\ref{enum:inj-and-inj-qcoh-arbitrary-sums}) and
  the obvious functor $[C(\InjQcoh(X))] \ra [C(\Qcoh(X))]$
  preserves coproducts 
  and obviously detects them. This implies that 
  $\iota'$ preserves coproducts if and only if $\iota'' \colon
  [C^\fib(\Qcoh(X))] \ra [C(\InjQcoh(X))]$ preserves coproducts.

  Recall from \cite[Thm.~1.1]{krause-stable-derived}
  that $[C(\InjQcoh(X))]$ is compactly generated, and
  that the functor $Q\colon [C(\InjQcoh(X))] \ra
  D(\Qcoh(X))$ 
  induces an equivalence 
  \begin{equation}
    \label{eq:DbCoh-renormalized}
    [C(\InjQcoh(X))]^\cpt \sira D^b(\Coh(X))  
  \end{equation}
  and has a right adjoint functor.
  If we identify
  $D(\Qcoh(X)) \cong
  [C^\fib(\Qcoh(X))]$ this right adjoint is
  given by $\iota''.$ 
  This implies that the functor $\iota''$
  preserves coproducts if and only if 
  $Q$
  preserves compact objects (use the adjunction $(Q, \iota'')$
  and that  
  $[C(\InjQcoh(X))]$ is compactly generated) if and only if
  $\mfPerf(X) = 
  D^b(\Coh(X))$ (use \eqref{eq:DbCoh-renormalized} and
  $D(\Qcoh(X))^\cpt=\mfPerf(X) \subset D^b(\Coh(X))$).
  Thus \ref{enum:coprod-hinj} $\Leftrightarrow$
  \ref{enum:Perf-DbCoh}. 
  The last claim follows from 
  Proposition~\ref{p:regular-vs-singularity-cat}.
\end{proof}
  
\appendix

\section{Sheaf homomorphisms and external
  tensor products} 
\label{sec:some-results-sheaf-hom-external-prod}

In this appendix we use calligraphic letters like $\mathcal{E}$,
$\mathcal{F}$ for sheaves and quasi-coherent sheaves, and
ordinary letters like $P$, $Q$ for vector bundles.

\subsection{Some base change isomorphisms}
\label{sec:some-base-change}

\begin{lemma}
  \label{l:base-change-ringed-spaces-open}
  Let $f\colon (Y, \mathcal{O}_Y) \ra (X, \mathcal{O}_X)$ be a
  morphism of ringed spaces. Let $U \subset X$ be an open subset
  and $V:= f\inv(U)$. Consider $U$ and $V$ as ringed spaces with
  structure sheaves $\mathcal{O}_X|_U$ and $\mathcal{O}_Y|_V$ so
  that we have a cartesian diagram
  \begin{equation*}
    \xymatrix{
      {V} \ar[r]^-{v} \ar[d]^-{f'} & {Y} \ar[d]^-{f}\\
      {U} \ar[r]^-{u} & {X}
    }
  \end{equation*}
  of ringed spaces. Then:
  \begin{enumerate}
  \item 
    \label{enum:push-and-restrict-to-open}
    For $\mathcal{G}$ a sheaf of $\mathcal{O}_Y$-modules,
    there is a natural 
    isomorphism $u^*f_*\mathcal{G} \sira f'_*
    v^* \mathcal{G}$ of sheaves of $\mathcal{O}_U$-modules.
  \item 
    \label{enum:proper-base-change-open}
    For $\mathcal{F}$ a sheaf of $\mathcal{O}_U$-modules,
    there is a natural isomorphism 
    $f^* u_!\mathcal{F} \sira v_! f'^*\mathcal{F}$
    of $\mathcal{O}_Y$-modules.
  \end{enumerate}
\end{lemma}

\begin{proof}
  \ref{enum:push-and-restrict-to-open}:
  The morphism comes from $f'^*u^*f_* =v^*f^*f_* \ra v^*$
  by adjunction.
  For $W \subset U$
  open we have
  $(u^*(f_*\mathcal{G}))(W)=(f_*\mathcal{G})(W)
  =\mathcal{G}(f\inv(W))=
  \mathcal{G}(f'^{-1}(W))
  =
  (v^*\mathcal{G})(f'^{-1}(W))
  =
  (f'_*(v^*\mathcal{G}))(W).$

  \ref{enum:proper-base-change-open}:
  The morphism in the lemma
  is constructed in the usual way, see
  e.\,g.\ \cite{wolfgang-olaf-locallyproper}.
  The stalk of $f^* u_!\mathcal{F}$ 
  at a point $y \in Y$ is 
  \begin{equation*}
    (\mathcal{O}_Y \otimes_{f\inv\mathcal{O}_X} f\inv u_!
    \mathcal{F})_y
    =
    \mathcal{O}_{Y,y} \otimes_{\mathcal{O}_{X,f(y)}} (u_!
    \mathcal{F})_{f(y)}
    =
    \begin{cases}
      \mathcal{O}_{Y,y} \otimes_{\mathcal{O}_{X,f(y)}}
      \mathcal{F}_{f(y)} & \text{if $f(y) \in U,$}\\  
      0 & \text{otherwise.}
    \end{cases}
  \end{equation*}
  On the other hand, we have
  $(v_! f'^*\mathcal{F})_y = \mathcal{O}_{V,y}
  \otimes_{\mathcal{O}_{U,f'(y)}} \mathcal{F}_{f'(y)}$ if $y \in
  V,$ and $(v_! f'^*\mathcal{F})_y =0$ otherwise.   
  One may also deduce \ref{enum:proper-base-change-open}
  from \ref{enum:push-and-restrict-to-open} (and conversely)
  using the Yoneda 
  embedding and the obvious adjunctions. 
\end{proof}

\begin{lemma}
  \label{l:push-affine-then-pull}
  Let
  \begin{equation*}
    \xymatrix{
      {Y'} \ar[d]^-{a'} \ar[r]^-{g'} 
      & {Y} \ar[d]^-{a}\\
      {X'} \ar[r]^-{g} 
      & {X}
    }
  \end{equation*}
  be a cartesian diagram of schemes.
  If $a$ is an affine morphism,
  there 
  is
  an isomorphism $g^*a_* \sira a'_*g'^*$ of
  functors $\Qcoh(Y) \ra \Qcoh(X').$
\end{lemma}

\begin{proof}
  See \cite[Prop.~12.6]{goertz-wedhorn-AGI}.
  Both
  $a_*$ and $a'_*$
  preserve quasi-coherence 
  by Lemma~\ref{l:qcqs-preserves-qcoh}.
\end{proof}

\begin{lemma}
  \label{l:push-from-open-and-tensor}
  Let $a \colon Y \ra X$ be an
  affine morphism of schemes.
  Given $\mathcal{F} \in \Qcoh(Y)$ and $\mathcal{G} \in \Qcoh(X)$
  there is a canonical isomorphism
  $(a_*\mathcal{F}) \otimes \mathcal{G} \sira a_*(\mathcal{F}
  \otimes a^*\mathcal{G})$ of quasi-coherent sheaves.
\end{lemma}

\begin{proof}
  Use \cite[Prop.~12.6]{goertz-wedhorn-AGI}
  and that $a_*$ preserves
  quasi-coherence by
  Lemma~\ref{l:qcqs-preserves-qcoh}.
\end{proof}

\begin{lemma}[Projection formula for an open embedding]
  \label{l:projection-formula-for-open-embedding}
  Let $(X, \mathcal{O}_X)$ be a ringed space and $U \subset X$ an
  open subset, considered as a ringed space with structure
  sheaf $\mathcal{O}_U=\mathcal{O}_X|_U$.
  Let $\mathcal{F}$ be a sheaf of $\mathcal{O}_U$-modules and
  $\mathcal{G}$ a sheaf of $\mathcal{O}_X$-modules. Then there is
  a natural isomorphism
  of sheaves of $\mathcal{O}_X$-modules
  \begin{equation*}
    (u_! \mathcal{F}) \otimes \mathcal{G} \sira
    u_!(\mathcal{F} 
    \otimes u^*\mathcal{G}).
  \end{equation*}
\end{lemma}

\begin{proof}
  The morphism in the lemma
  is constructed in the usual way, see
  e.\,g.\ \cite{wolfgang-olaf-locallyproper}.
  To check that it is an isomorphism we consider it on the stalks 
  at an arbitrary point $x \in X.$ 
  The stalks of both sides are compatibly
  identified with 
  $\mathcal{F}_x \otimes_{\mathcal{O}_{X,x}} \mathcal{G}_x$
  if $x \in U$, and vanish otherwise.
\end{proof}

\subsection{Some results for sheaf homomorphisms}
\label{sec:some-results-sheaf}




Given open subschemes $U$ and $V$ of a scheme $X$ the cartesian
diagram
\begin{equation}
  \label{eq:cart-UV-open-in-X}
  \xymatrix{
    {U \cap V} \ar[r]^-{u'} \ar[d]^-{v'} 
    & {V} \ar[d]^-{v}\\
    {U} \ar[r]^-{u} 
    & {X}
  }
\end{equation}
will be used several times in the following.

Recall that an integral domain $R$ is N-1 if
the integral closure of $R$ in its quotient field is a finite
$R$-module (see
\cite[\href{http://stacks.math.columbia.edu/tag/032F}{Def.~032F}]{stacks-project}). Trivially,
any integrally closed (= normal) integral domain is N-1. 
More interestingly, any Nagata integral domain and in particular
any 
(quasi-)excellent integral domain is N-1.
\cite[\href{http://stacks.math.columbia.edu/tag/07QV}{Lemma~07QV},
\href{http://stacks.math.columbia.edu/tag/035S}{Lemma~035S}]{stacks-project}


\begin{lemma}
  \label{l:Hom-UPVQ-star}
  Let $X$ be a scheme and let $u \colon U \subset X$ and $v
  \colon V \subset X$ be inclusions of open subschemes. Assume
  that $V$ is an affine
  Noetherian integral scheme such that 
  $\mathcal{O}_V(V)$ is N-1, and that $U
  \cap V$ is affine.  Let $P$ be a vector bundle on $U$ and $Q$ a
  vector bundle on $V.$
  Then
  \begin{equation*}
    \Hom_{\mathcal{O}_X}(u_*P, v_*Q)
    =
    \begin{cases}
      \Hom_{\mathcal{O}_V}(P|_V, Q) & \text{if $V \subset
        U,$}\\
      0 & \text{otherwise.}
    \end{cases}
  \end{equation*}
\end{lemma}

In Remark~\ref{rem:Hom-UPVQ-weaker-assumptions} below we show
that this lemma is not true if we replace ``integral'' and
``N-1'' by ``irreducible'', or by ``reduced''.
Lemma~\ref{l:sheafHom-UPVQ} contains a variant of
this lemma.

\begin{proof}
  We work with the cartesian diagram 
  \eqref{eq:cart-UV-open-in-X}.
  The adjunction $(v^*, v_*)$ and
  Lemma~\ref{l:base-change-ringed-spaces-open}.\ref{enum:push-and-restrict-to-open} 
  yield
  \begin{equation*}
    \Hom_{\mathcal{O}_X}(u_*P, v_*Q)
    = \Hom_{\mathcal{O}_V}(v^*u_*P, Q)
    \sila \Hom_{\mathcal{O}_V}(u'_*v'^*P, Q).
  \end{equation*}
  If $V \subset U$ then $U \cap V =V$ and the claim is trivial.
  Assume that $V \not\subset U.$ So $u' \colon  U \cap V
  \subsetneq V.$ 
  If $U \cap V = \emptyset$ the claim is trivial, 
  so assume that $U \cap V \not=\emptyset.$  

  So we have $\emptyset \not= U \cap V = \Spec B \subsetneq V =
  \Spec A.$ The corresponding ring morphism $A \ra B$ is
  injective  
  but not bijective, and we can view $A \subsetneq B$ as subrings
  of the quotient field $Q(A)$ of the Noetherian N-1 integral
  domain $A.$

  Assume that $u'_*(P|_{U \cap V}) \ra Q$ is a
  nonzero 
  morphism. Since its source and target are quasi-coherent
  sheaves on $V,$ by Lemma~\ref{l:qcqs-preserves-qcoh}, taking
  global  
  sections translates this morphism into a nonzero morphism
  $P|_{U \cap V}(U \cap V) \ra Q(V)$ of $A$-modules. 
  Since $P|_{U \cap V}$ and $Q$ are vector bundles
  on $\Spec B$ and $\Spec A$, respectively, they are direct
  summands of 
  free modules of finite rank.
  In particular we deduce that there
  is a nonzero morphism $B \ra A$ of $A$-modules.
  But this cannot happen by Lemma~\ref{l:no-morphisms-BA}.
\end{proof}

\begin{lemma}
  \label{l:no-morphisms-BA}
  Let $A$ be a Noetherian N-1 integral domain with quotient field
  $Q(A),$ and let $A \subsetneq B \subset Q(A)$ be an
  intermediate ring such that $\Spec B \subsetneq \Spec A$ is
  open. 
  Then any morphism $f \colon  B \ra A$ of $A$-modules is zero.
\end{lemma}


\begin{proof}
  Let $\tildew{A}$ (resp.\ $\tildew{B}$) be the integral closure
  of $A$ (resp.\ $B$) in $Q(A).$ 
  Then the multiplication map $B
  \otimes_A \tildew{A} \ra \tildew{B}$ is an isomorphism and
  $\tildew{A} \subsetneq \tildew{B}$ 
  and 
  $\Spec \tildew{B} \subsetneq \Spec \tildew{A}$ is open
  (by
  \cite[Prop.~12.43]{goertz-wedhorn-AGI}).  
  Both vertical arrows in the commutative diagram
  \begin{equation*}
    \xymatrix{
      & {B} \ar[r]^-{f} \ar[d] & {A} \ar[d]\\
      {\tildew{B}} & {B \otimes_A \tildew{A}} \ar[l]_-{\sim}
      \ar[r]^-{f \otimes \id} & {A \otimes_A \tildew{A}}
      \ar[r]^-{\sim} & {\tildew{A}} 
    }
  \end{equation*}
  are injective since they correspond to the inclusions $B
  \subset \tildew{B}$ and $A \subset \tildew{A}.$
  (The isomorphism $B \otimes_A \tildew{A} \sira \tildew{B}$ is
  an easy instance of  
  Lemma~\ref{l:push-affine-then-pull}.)
  Hence 
  it is enough to
  show that the  
  morphism $f \otimes \id$ of $\tildew{A}$-modules is zero.
  Since $A$ is N-1 the ring $\tildew{A}$ is Noetherian, so it
  is enough to prove the lemma under the additional assumption
  that $A$ 
  is integrally closed (= normal).
  
  Let $A$ be normal. (In the rest of the proof we do not need
  that $\Spec B \subset \Spec A$ is open.)
  Let $b \in B$ with $b \not\in A.$
  As a normal Noetherian integral domain, $A$ is the intersection
  of the discrete valuation rings $A_\mfp,$
  where $\mfp$ runs through all
  minimal
  nonzero prime ideals $\mfp \subset A$ (see
  \cite[Thm.~8.10]{reid-undergrad-comm-alg}).  
  Hence there is a nonzero minimal prime ideal $\mfp \subset A$
  such that $b \not\in A_\mfp.$ Let $S:= A \setminus \mfp.$ We
  localize the composition 
  $A \ra B \xra{f} A$ of $A$-linear maps at $S$ and obtain
  $A_\mfp$-linear maps
  \begin{equation*}
    A_\mfp \ra S\inv B \xra{S\inv f} A_\mfp
  \end{equation*}
  The first morphism takes place in $Q(A)$ and is injective
  but
  not bijective since $b \in B \subset S\inv B$ but $b \not\in
  A_\mfp.$ 
  Since $A_\mfp$ is a discrete valuation ring this implies that
  $S\inv B =Q(A).$
  Then $S\inv f \colon  S\inv B =Q(A) \ra A_\mfp$ is $A_\mfp$-linear and then
  certainly zero. (If $R$ is an integral domain which is not a
  field, any
  $R$-linear morphism $\varphi \colon Q(R) \ra R$ is zero:
  let $r:= \varphi(1);$ if $r=0$ then $\varphi=0$;
  otherwise any $0 \not= s \in R$ is invertible in $R$:
  $s\varphi(1/rs)=\varphi(1/r)=1.$)
  Hence the composition $B \xra{f} A \ra A_\mfp$ is
  zero, and so is $f.$
\end{proof}

\begin{remark}
  \label{rem:Hom-UPVQ-weaker-assumptions}
  We give two examples that
  Lemma~\ref{l:Hom-UPVQ-star}
  is not true if $V$ is only required to be irreducible (resp.\
  reduced) instead of integral with $\mathcal{O}_V(V)$ N-1. 
  Both
  examples are simple in the sense that we have $U \subsetneq
  X=V$ and the vector bundles on 
  $U$ and $V$ are the structure sheaves. 
  Let
  $k$ be a field. 
  \begin{enumerate}
  \item
    Let $A:= k[T,\epsilon]/((T^2-1)\epsilon, \epsilon^2)$
    and $B:= A_T=A[T\inv]=k[T, T\inv,
      \epsilon]/((T^2-1)\epsilon, \epsilon^2).$  
    Then $A=k[T] \oplus \frac{k[T]}{(T^2-1)} \epsilon$ and
    $B=k[T,T\inv] \oplus \frac{k[T]}{(T^2-1)} \epsilon$ as
    $k[T]$-modules.
    Then the composition $B \ra k[T,T\inv] \xra{1 \mapsto
      \epsilon}  
    \frac{k[T]}{(T^2-1)} \epsilon \ra A$ is a nonzero morphism
    of $A$-modules. 
    Hence if we define $X=V =\Spec A$ and $U=\Spec B$ and
    $Q=\mathcal{O}_X$ and $P=\mathcal{O}_U$ we see that
    $\Hom_{\mathcal{O}_X}(u_*P, v_*Q) 
    = \Hom_A(B,A)\not=0$
    but $V \not\subset U.$ Note that $X=V$ is an
    affine irreducible (non-reduced) Noetherian excellent scheme
    and 
    $U=U\cap V$ is 
    a strict principal open subset.
  \item 
    Let $A:= k[S,T]/(ST)$ and
    $B:=A[(S+1)\inv]=k[S,T,(S+1)\inv]/(ST).$ 
    Any element of $B$ can be written as
    $\frac{f(S)+g(T)T}{(S+1)^n}$ with $f \in k[S],$ $g \in k[T]$
    and $n \in \DN.$ Mapping such an element to $(f(0)+g(T)T)T^2$
    yields a well-defined nonzero morphism $f \colon B \ra A$ of
    $A$-modules. 
    Then $X=V=\Spec A$ is affine reduced (non-irreducible but
    connected) Noetherian excellent and $U=
    \Spec B$ is a strict
    principal open
    subset but 
    $\Hom_{\mathcal{O}_X}(u_*\mathcal{O}_U,
    v_*\mathcal{O}_V) \not=0.$
  \end{enumerate}
\end{remark}

\begin{lemma}
  \label{l:sheafHom-UPVQ}
  Let $X$ be a scheme and assume that $u \colon U \subset X$ and
  $v \colon V \subset X$ are affine inclusions of open subschemes.
  We assume that $V$ is a
  Nagata
  integral
  scheme.
  Let $u' \colon  U \cap V \subset V$ be the
  open immersion. 
  Let $P$ be a vector bundle on $U$ and $Q$ a
  vector bundle on $V.$
  Then
  \begin{equation}
    \label{eq:sheafHom-UPVQ}
    \sheafHom_{\mathcal{O}_X}(u_*P, v_*Q)
    \sira
    v_* u'_! \sheafHom_{\mathcal{O}_{U \cap V}}(P|_{U \cap V},
    Q|_{U \cap V}), 
  \end{equation}
  and taking global sections yields
  \begin{equation}
    \label{eq:sheafHom-UPVQ-global}
    \Hom_{\mathcal{O}_X}(u_*P, v_*Q)
    =
    \begin{cases}
      \Hom_{\mathcal{O}_V}(P|_V, Q) & \text{if $V \subset
        U,$}\\
      0 & \text{otherwise.}
    \end{cases}
  \end{equation}
\end{lemma}

In particular, $\sheafHom_{\mathcal{O}_X}(u_*P, v_*Q)$ is not
quasi-coherent in general. 

\begin{corollary}
  \label{c:sheafHom-UPVQ}
  If $X$ is a Nagata integral scheme, $u \colon U \subset X$
  is an affine inclusion of an open subscheme, and $P$ is a
  vector bundle on $U,$ then 
  \begin{equation*}
    (u_*P)^\cek
    \sira u_!(P^\cek). 
  \end{equation*}
  If $U \subsetneq X$ then 
  $\Hom_{\mathcal{O}_X}(u_*P, \mathcal{O}_X)=0.$
\end{corollary}


\begin{proof}[Proof of Corollary~\ref{c:sheafHom-UPVQ}]
  Apply Lemma~\ref{l:sheafHom-UPVQ} to $V=X$ and
  $Q=\mathcal{O}_X.$
\end{proof}

\begin{remark}
  If $X$ is a scheme, $u \colon U \subset X$ the inclusion of
  an open subscheme, and $\mathcal{F}$ a sheaf on $U$ we have
  $(u_!\mathcal{F})^\cek=\sheafHom(u_!\mathcal{F}, \mathcal{O}_X)
  \cong u_*\sheafHom(\mathcal{F},
  u^*\mathcal{O}_X)=u_*(\mathcal{F}^\cek).$ Hence in the setting
  of Corollary~\ref{c:sheafHom-UPVQ} the sheaves $u_!P$ and
  $u_*P$ satisfy $u_!P \cong ((u_!P)^\cek)^\cek$
  and $u_*P \cong (u_*P)^\cek)^\cek.$
\end{remark}

\begin{proof}[Proof of Lemma~\ref{l:sheafHom-UPVQ}]
  We work with the cartesian diagram 
  \eqref{eq:cart-UV-open-in-X}.
  Consider the morphism
  $\sheafHom_{\mathcal{O}_X}(u_*P, v_*Q)
  \ra
  \sheafHom_{\mathcal{O}_X}(u_!P, v_*Q)$
  coming from $u_!P \ra u_*P.$
  Its target
  \begin{multline*}
    \sheafHom_{\mathcal{O}_X}(u_!P, v_*Q)
    =
    v_*\sheafHom_{\mathcal{O}_V}(v^*u_!P, Q)
    \sila
    v_*\sheafHom_{\mathcal{O}_V}(u'_!v'^*P, Q)\\
    =
    v_*u'_*\sheafHom_{\mathcal{O}_{U \cap V}}(P|_{U \cap V},
    Q|_{U \cap V})
  \end{multline*}
  (use the adjunctions $(v^*,v_*)$ and $(u'_!, u'^!=u'^*)$ and
  Lemma~\ref{l:base-change-ringed-spaces-open}.\ref{enum:proper-base-change-open})
  contains
  the right-hand side of \eqref{eq:sheafHom-UPVQ} as a subsheaf
  ($v_*$ being left exact). 
  We claim that it induces an isomorphism onto this subsheaf.
  To see this we show that both sides of \eqref{eq:sheafHom-UPVQ}
  have the same evaluation at an arbitrary 
  affine open subset $W \subset X.$
  The commutative diagrams 
  \begin{equation*}
    \xymatrix{
      {U \cap W} \ar[r]^-{u''} \ar[d]^-{w'} & {W} \ar[d]^-{w}\\
      {U} \ar[r]^-{u} & {X,}
    }
    \quad \quad
    \xymatrix{
      {V \cap W} \ar[r]^-{v''} \ar[d]^-{w''} & {W} \ar[d]^-{w}\\
      {V} \ar[r]^-{v} & {X,}
    }
  \end{equation*}
  and 
  Lemma~\ref{l:base-change-ringed-spaces-open}.\ref{enum:push-and-restrict-to-open} 
  imply that
  $(u_*P)|_W
  =w^*u_*P \sira u''_*w'^*P
  =u''_*(P|_{U \cap W})$
  and
  $(v_*Q)|_W
  =w^*v_*Q \sira v''_*w''^*Q=v''_*(Q|_{V \cap W}).$
  This and Lemma~\ref{l:Hom-UPVQ-star}
  (note that both $V \cap W$ and $U \cap V \cap W$ are affine
  since $v$ and $u$ are affine, and that $V \cap W$ is either
  empty or 
  Nagata 
  (by
  \cite[\href{http://stacks.math.columbia.edu/tag/033X}{Lemma~033X}]{stacks-project}),
  in particular
  Noetherian, and
  integral 
  with $\mathcal{O}_{V \cap W}(V \cap W)$ N-1)  
  show that
  \begin{align*}
    \sheafHom_{\mathcal{O}_X}(u_*P, v_*Q) (W)
    & =
    \Hom_{\mathcal{O}_W}(u''_*(P|_{U \cap W}), 
    v''_*(Q|_{V \cap W}))\\
    & =
    \begin{cases}
      \Hom_{\mathcal{O}_{V \cap W}}(P|_{V \cap W}, Q|_{V \cap W})
      & \text{if $V \cap W \subset U,$}\\
      0 & \text{otherwise.}
    \end{cases}
  \end{align*}
  On the other hand we have
  \begin{equation*}
    (v_* u'_! \sheafHom_{\mathcal{O}_{U \cap V}}(P|_{U \cap V},
    Q|_{U \cap V}))(W)  
    =
    (u'_! \sheafHom_{\mathcal{O}_{U \cap V}}(P|_{U \cap V}, Q|_{U
      \cap V}))(V \cap W) 
  \end{equation*}
  Note that $E:=\sheafHom_{\mathcal{O}_{U \cap
      V}}(P|_{U \cap V}, Q|_{U \cap V})$ is a 
  vector bundle on $U \cap V$ satisfying
  \begin{equation*}
    E(U \cap V \cap W)
    = \Hom_{\mathcal{O}_{U \cap V \cap W}}(P|_{U \cap V \cap W},
    Q|_{U \cap V \cap W}),
  \end{equation*}
  and that in case $V \cap W \subset U$ we have $U \cap V \cap W=
  V \cap W.$ Hence it is sufficient to check that
  \begin{equation*}
    (u'_!E)(V \cap W) =
    \begin{cases}
      E(V \cap W)
      & \text{if $V \cap W \subset U,$}\\
      0 & \text{otherwise.}
    \end{cases}
  \end{equation*}
  Since the inclusion morphism $u' \colon U \cap V \hra V$ is
  affine, 
  this follows from
  Lemma~\ref{l:shriek-extension-of-structure-sheaf}, applied to
  $r=u' \colon R= U\cap V \hra Y=V$
  and $T=V \cap W.$ 
  This establishes the isomorphism \eqref{eq:sheafHom-UPVQ}.

  If we evaluate it at $X$
  we obtain 
  $\Hom_{\mathcal{O}_X}(u_*P, v_*Q) = (u'_! E)(V),$
  and Lemma~\ref{l:shriek-extension-of-structure-sheaf} applied
  to the same morphism as above but $T=V$ shows
  \eqref{eq:sheafHom-UPVQ-global} because $V \subset U \cap V$ is
  equivalent to $V \subset U$ and to $V = U \cap V.$
\end{proof}

\begin{lemma}
  \label{l:shriek-extension-of-structure-sheaf}
  Let $r \colon  R \subset Y$ be the
  inclusion of an open subscheme in an integral scheme $Y$ (it is
  sufficient to assume that $R$ is integral and $Y$ is
  irreducible).  
  Assume that $r$ is an affine morphism.
  If 
  $E$ is a vector bundle on $R$ and
  $T \subset Y$ is an arbitrary open subset, then
  \begin{equation*}
    (r_!E)(T) =
    \begin{cases}
      E(T)
      & \text{if $T \subset R,$}\\
      0 & \text{otherwise.}
    \end{cases}
  \end{equation*}
\end{lemma}

\begin{proof}
  By definition of $r_!$ we have
  \begin{equation*}
    (r_!E)(T) = \{s \in E(T \cap R) \mid
    \supp s \subset T \text{ closed}\}.
  \end{equation*}
  So if $T \subset R$ this is equal to $E(T)$ because 
  any section of $E(T)$ has closed support in $T.$

  Now let $T$ be arbitrary and assume that $s \in
  (r_!E)(T)$ is nonzero. Then there is some $t \in T
  \cap R$
  such that the germ $s_t$ is nonzero.
  Let $A \subset T$ be an affine open neighborhood of $t.$
  Since $r$ is affine, $A \cap R$ is affine, and hence 
  $E|_{A \cap R}$ is a direct summand of
  a finite direct sum of copies of $\mathcal{O}_{A \cap R}.$
  In particular,
  there is a morphism $E|_{A \cap R} \ra
  \mathcal{O}_{A \cap R}$
  such that the 
  image $s'$ of $s|_{A \cap R}$ in $\mathcal{O}_{A \cap R}(A \cap
  R)$ is nonzero. Since $A \cap R$ is an integral scheme we have
  $\Supp s' =A \cap R.$ This implies that $\supp (s|_{A \cap R}) =A
  \cap R.$
  Hence $s|_A \in (r_!E)(A)$ when viewed as an element
  of $E(A \cap R)$ has support $A \cap R,$ so $A \cap R
  $ is closed in $A.$

  If $B \subset T$ is an arbitrary non-empty affine open subset,
  then $A \cap B \not= \emptyset$ because $T$ is irreducible. 
  Since the stalk of $s$ at all points of $A \cap B$ is nonzero
  we can repeat the above argument and see that
  $B \cap R$ is closed in $B.$ This means that $T \cap R$ is
  closed in 
  $T.$ Since $T \cap R$ is non-empty and open in $T$ and $T$ is
  irreducible we have $T \cap R=T,$ hence $T \subset R.$ This
  proves the lemma.
\end{proof}


\begin{remark}
  \label{rem:shriek-extension-counterexample-to-lemma} 
  We give two examples showing that
  Lemma~\ref{l:shriek-extension-of-structure-sheaf} 
  is not true if $Y$ is only required to be irreducible (resp.\
  reduced) instead of integral. 
  Let $k$ be any algebraically closed field.
  \begin{enumerate}
  \item 
    \label{enum:shriek-extension}
    Let $A=k[U,V]/(UV,V^2)$ and $a \in k[U] \subset A$ with
    $a \not\in k$ and $a(0)\not=0.$ Let $R:= \Spec A_a \subsetneq
    Y:=\Spec A$ and let 
    $T=Y.$ So $Y$ and $R$ are irreducible but not reduced. 
    The section $s:= V \in \mathcal{O}_R(T \cap
    R)=\mathcal{O}_R(R)=A_a$ is nonzero 
    and its support $\Supp s = \{(U,V)\} \subset R$ is closed in
    $T=Y.$ Hence it defines a nonzero element of
    $(r_!\mathcal{O}_R)(T)$ even though $T \not\subset R.$
    We may rewrite this as
    \begin{equation*}
      \Hom_{\mathcal{O}_Y}(\mathcal{O}_Y,
      (r_!\mathcal{O}_R))\sira 
      (r_!\mathcal{O}_R)(Y) \not=0.
    \end{equation*}
  \item 
    Let $Y=\Spec k[U,V]/(UV)$ which is 
    reduced and connected but not irreducible.
    Let $R := Y_U$ (the $U$-axis without zero, which is integral)
    and $T=Y_{U+V}$
    ($Y$ without the origin).
    Then $(r_!\mathcal{O}_R)(T)=\mathcal{O}_R(R)\not=0$ even
    though $T \not\subset R.$
  \end{enumerate}
\end{remark}

\begin{lemma}
  \label{l:Hom-UPVQ-shriek}
  Let $X$ be a scheme and let $u \colon U \subset X$ and $v
  \colon V \subset X$ be inclusions of open subschemes.  We
  assume that $U$ is an affine integral scheme and that $U \cap
  V$ is affine (it is sufficient to assume that $U$ is affine
  irreducible and that $U \cap V$ is affine integral). Let $P$ be
  a 
  vector bundle on $U$ and $Q$ a vector bundle on $V.$
  Then
  \begin{equation*}
    \Hom_{\mathcal{O}_X}(u_!P, v_!Q)
    =
    \begin{cases}
      \Hom_{\mathcal{O}_U}(P, Q|_U) & \text{if $U \subset
        V,$}\\
      0 & \text{otherwise.}
    \end{cases}
  \end{equation*}
\end{lemma}

Remark~\ref{rem:shriek-extension-counterexample-to-lemma}.\ref{enum:shriek-extension}
shows that this lemma is not true if we replace ``integral''
by ``irreducible''.

\begin{proof}
  We work with the cartesian diagram 
  \eqref{eq:cart-UV-open-in-X}.
  The adjunction $(u_!, u^!=u^*)$ and
  Lemma~\ref{l:base-change-ringed-spaces-open}.\ref{enum:proper-base-change-open} 
  yield
  \begin{equation*}
    \Hom_{\mathcal{O}_X}(u_!P, v_!Q)
    = \Hom_{\mathcal{O}_U}(P, u^*v_!Q)
    \sira 
    \Hom_{\mathcal{O}_U}(P, v'_!u'^*Q).
  \end{equation*}
  If $U \subset V$ then $U \cap V = U$ and the claim is clear.
  
 Assume that $\Hom_{\mathcal{O}_U}(P, v'_!u'^*Q) \not=0.$
 We need to show that $U \subset V.$
 Since $U$ and $U \cap V$ are affine and $P$ and $u'^*Q$ are
 vector bundles we see that
  \begin{equation*}
    0 \not= 
    \Hom_{\mathcal{O}_U}(\mathcal{O}_U, v'_! \mathcal{O}_{U \cap V})
    \sira
    (v'_!(\mathcal{O}_{U \cap V})) (U).
  \end{equation*}
  Pick a nonzero element of the right hand side.
  It is given by an element $s \in
  \mathcal{O}_{U \cap V}(U \cap V)$ whose support is
  closed in $U,$ cf.\ the proof of
  Lemma~\ref{l:shriek-extension-of-structure-sheaf}.
  Since $U \cap V$ is integral we necessarily have $\Supp s = U
  \cap V.$  This set is
  non-empty and open and closed in the irreducible set $U,$ hence
  $U 
  \cap V=U$ and $U \subset V.$
  This proves the lemma.
\end{proof}




\subsection{External tensor product}
\label{sec:extern-tens-prod}

Let $X$ and $Y$ be schemes over a field $k$ with projection
morphisms $X \xla{p} X \times Y \xra{q} Y.$
We explain some properties and compatibilities of the bifunctor
\begin{equation*}
  (-\boxtimes?):= p^*(-) \otimes q^*(?) \colon \Sh(X) \times
  \Sh(Y) \ra \Sh(X \times Y).
\end{equation*}

\begin{lemma}
  \label{l:boxtimes-exact}
  Given schemes $X$ and $Y$ over a field $k$ the bifunctor
  $\boxtimes$
  is exact. 
\end{lemma}

\begin{proof}
  For $\mathcal{F} \in \Sh(X)$ and $\mathcal{G} \in
  \Sh(Y)$ we have
  \begin{equation*}
    \mathcal{F} \boxtimes \mathcal{G} = p\inv \mathcal{F}
    \otimes_{p\inv \mathcal{O}_X} \mathcal{O}_{X \times Y}
    \otimes_{q\inv \mathcal{O}_Y} q\inv
    \mathcal{G} 
    \sira 
    (p\inv\mathcal{F} \otimes_{\ul{k}}
    q\inv\mathcal{G}) \otimes_{p\inv\mathcal{O}_X \otimes_{\ul{k}}
      q\inv \mathcal{O}_Y} \mathcal{O}_{X \times Y}
  \end{equation*}
  (where $\ul{k}$ is the constant sheaf with stalk $k$)
  and the morphism 
  $p\inv\mathcal{O}_X \otimes_{\ul{k}} q\inv \mathcal{O}_Y
  \ra \mathcal{O}_{X \times Y}$ of sheaves of rings is
  flat because for $\Spec A \subset X$ and $\Spec B \subset Y$ open
  it is 
  given at the stalk at
  $\mfr \in \Spec (A \otimes B)$ with $\mfp=p(\mfr)$ and
  $\mfq=q(\mfr)$ by $A_\mfp \otimes B_\mfq \ra (A \otimes B)_\mfr.$
\end{proof}

\begin{lemma}
  \label{l:diagonal-restriction-of-extension}
  Let $X$ be a scheme over a field $k,$ and let
  $u \colon  U \subset X$ and $v \colon  V \subset X$ be
  affine immersions of open subschemes.
  Then
  $\Delta^*(\leftidx{_{U \times V}}{(\mathcal{E} \boxtimes \mathcal{F})}{})
  \sira
  \leftidx{_{U \cap V}}{(\mathcal{E} \otimes \mathcal{F})}{}$    
  canonically,
  for $\mathcal{E}$ and $\mathcal{F}$ quasi-coherent sheaves
  on $X.$
\end{lemma}

\begin{proof}
  Consider
  the pullback
  diagram
  \begin{equation*}
    \xymatrix{
      {U \cap V} \ar[r]^-{s} \ar[d]^-{\delta} &
      {X} \ar[d]^-{\Delta}\\
      {U \times V} \ar[r]^-{u \times v} &
      {X \times X.}
    }
  \end{equation*}
  If $\mathcal{G}$ is a 
  quasi-coherent sheaf on $U \times V,$ there
  is
  a natural isomorphism 
  $\Delta^* (u \times v)_*(\mathcal{G}) \sira s_*\delta^*(\mathcal{G}),$
  by Lemma~\ref{l:push-affine-then-pull}.
  In particular, we obtain
  \begin{equation*}
    \Delta^*(u \times v)_* (u \times v)^*(\mathcal{E} \boxtimes \mathcal{F})
    \sira
    s_*\delta^*(u \times v)^*(\mathcal{E} \boxtimes \mathcal{F})
    =
    s_*s^*\Delta^*(\mathcal{E} \boxtimes \mathcal{F})
    \sira
    s_*s^*(\mathcal{E} \otimes \mathcal{F}).
  \end{equation*}
\end{proof}

\begin{lemma}
  \label{l:boxtimes-restriction-extension}
  Let $u \colon U \ra X$ and $v \colon V \ra Y$ be morphisms of
  schemes over a field $k.$ Then
  \begin{enumerate}
  \item 
    \label{enum:boxtimes-restriction}
    Let $\mathcal{E} \in \Sh(X)$ and $\mathcal{F} \in \Sh(Y).$
    Then
    $u^*\mathcal{E} \boxtimes v^*\mathcal{F} \sira (u \times v)^*(\mathcal{E} \boxtimes \mathcal{F})$
    naturally. 
  \end{enumerate}
  Assume in addition that
  $U$ and $V$ are open subschemes of $X$ and $Y,$
  respectively.
  \begin{enumerate}[resume]
  \item
    \label{enum:boxtimes-!-extension}
    Let $\mathcal{E}' \in \Sh(U)$ and $\mathcal{F}' \in \Sh(V).$
    Then 
    $u_!\mathcal{E}' \boxtimes v_!\mathcal{F}' \sira (u \times v)_!(\mathcal{E}' \boxtimes \mathcal{F}')$
    naturally. In particular, 
    $\leftidx{^U}{\mathcal{E}}{} \boxtimes \leftidx{^V}{\mathcal{F}}{}
    \sira
    \leftidx{^{U \times V}}{(\mathcal{E} \boxtimes \mathcal{F})}{}$ 
    for $\mathcal{E} \in \Sh(X)$ and $\mathcal{F} \in \Sh(Y).$
  \end{enumerate}
  Assume in addition that $u$ and $v$ are
  affine morphisms.
  \begin{enumerate}[resume]
  \item 
    \label{enum:boxtimes-*-extension}
    Let $\mathcal{E}' \in \Qcoh(U)$ and $\mathcal{F}' \in \Qcoh(V).$
    Then
    $u_*\mathcal{E}' \boxtimes v_*\mathcal{F}' \sira (u \times v)_*(\mathcal{E}' \boxtimes \mathcal{F}')$
    naturally. 
    In particular, $\leftidx{_U}{\mathcal{E}}{} \boxtimes \leftidx{_V}{\mathcal{F}}{}
    \sira
    \leftidx{_{U \times V}}{(\mathcal{E} \boxtimes \mathcal{F})}{}$
    for $\mathcal{E} \in \Qcoh(X)$ and $\mathcal{F} \in \Qcoh(Y).$
  \end{enumerate}
\end{lemma}

\begin{proof}
  Part \ref{enum:boxtimes-restriction} follows from the
  usual isomorphism encoding compatibility of pullback and tensor
  product.
  For part \ref{enum:boxtimes-!-extension}, apply 
  Lemma~\ref{l:base-change-ringed-spaces-open}.\ref{enum:proper-base-change-open} 
  and Lemma~\ref{l:projection-formula-for-open-embedding}
  twice.
  The morphism in part \ref{enum:boxtimes-*-extension} is
  constructed using the adjunction $((u \times v)^*, (u \times
  v)_*).$ That it is an isomorphism
  can be checked locally on $X \times Y.$ Since $u$ and $v$
  are affine we can assume that
  $u\colon U=\Spec A' \ra X=\Spec A$ and $v \colon V=\Spec B' \ra
  Y=\Spec B.$ In this case the claim is obvious.
\end{proof}



\begin{corollary}
  \label{c:combined}
  Let $X$ be a scheme over a field $k,$ and let
  $u \colon  U \subset X$ and $v \colon  V \subset X$ be
  affine immersions of open subschemes.
  Then
  \begin{equation*}
    \Delta^*(\leftidx{_U}{\mathcal{E}}{} \boxtimes \leftidx{_V}{\mathcal{F}}{})
    \sira
    \Delta^*(\leftidx{_{U \times V}}{(\mathcal{E} \boxtimes \mathcal{F})}{})
    \sira
    \leftidx{_{U \cap V}}{(\mathcal{E} \otimes \mathcal{F})}{}
  \end{equation*}
  for $\mathcal{E} \in \Qcoh(X)$ and $\mathcal{F} \in \Qcoh(X).$
\end{corollary}

\begin{proof}
  Use Lemmata~\ref{l:diagonal-restriction-of-extension}
  and \ref{l:boxtimes-restriction-extension}.\ref{enum:boxtimes-*-extension}.
\end{proof}

\begin{remark}
  \label{rem:boxtimes-exact}
  If $X$ and $Y$ are schemes over a field $k$ we have seen in 
  Lemma~\ref{l:boxtimes-exact} that the bifunctor $\boxtimes \colon
  \Sh(X) \times \Sh(Y) \ra \Sh(X \times Y)$  
  is exact.
  Hence we use the same symbol for the induced functor
  \begin{equation*}
    (-\boxtimes ?) \colon D(\Sh(X)) \times
    D(\Sh(Y)) \ra D(\Sh(X \times Y)).
  \end{equation*}
  It is easy to see that 
  this functor is isomorphic to 
  $\bL p^*(-) \otimes^\bL \bL q^*(?)$ which is isomorphic to  
  $p^*(-) \otimes^\bL q^*(?)$ because $p$ and $q$
  are flat.
\end{remark}

\section{Triangulated categories in terms of dg
endomorphism algebras }
\label{sec:descr-triang-categ}

We work with dg categories over an arbitrary commutative ground
ring.

\begin{proposition}
  \label{p:homotopy-categories-triang-via-dg-algebras} 
  Let $\mathcal{C}$ be a dg category with a full pretriangulated
  dg subcategory $\mathcal{I}.$ Let $z \colon P \ra I$ be a
  closed degree zero morphism 
  in
  $\mathcal{C}$ 
  with $I \in
  \mathcal{I}$ such that $z^* \colon \Hom_{\mathcal{C}}(I, J) \ra
  \Hom_{\mathcal{C}}(P, J)$ is a quasi-isomorphism for all $J \in
  \mathcal{I},$ and let $B$ be a dg algebra together with a
  morphism $\beta \colon B \ra \End_\mathcal{C}(P)$ of dg
  algebras such that the composition $B
  \xra{\beta} \End_\mathcal{C}(P) \xra{z_*}
  \Hom_\mathcal{C}(P,I)$ is a quasi-isomorphism.
  \begin{enumerate}
  \item 
    \label{enum:hoI-via-B-small}
    If $[\mathcal{I}]$ is Karoubian 
    and $I$ is a
    classical generator of $[\mathcal{I}]$ then the functor
    \begin{equation*}
      \res^{\End_\mathcal{C}(P)}_{B} \comp \Hom_\mathcal{C}(P,-) \colon 
      [\mathcal{I}] \ra \per(B)
    \end{equation*}
    is an equivalence of triangulated categories.
  \item 
    \label{hoI-via-B}
    If $[\mathcal{I}]$ has all coproducts and $I$ is a 
    compact generator of $[\mathcal{I}]$ then the functor
    \begin{equation*}
    \res^{\End_\mathcal{C}(P)}_{B} \comp \Hom_\mathcal{C}(P,-) \colon 
    [\mathcal{I}] \ra D(B)
    \end{equation*}
    is an equivalence of triangulated categories.
  \end{enumerate}
\end{proposition}

\begin{proof}
  \ref{enum:hoI-via-B-small}:
  The upper horizontal arrow in the following commutative diagram
  is 
  an isomorphism since all the other maps are isomorphisms.
  \begin{equation}
    \label{eq:HomP-res}
    \xymatrix{
      {\Hom_{[\mathcal{I}]}(I,[m]I)} 
      \ar@{}[d]|-{\verteq}
      \ar[rrr]^-{\res^{\End_\mathcal{C}(P)}_{B} \comp
        \Hom_\mathcal{C}(P,-)} &&& 
      {\Hom_{D(B)}(\Hom_\mathcal{C}(P,I),[m]\Hom_\mathcal{C}(P,I))}
      \ar[dd]_-{(z_* \comp \beta)^*}^-{\sim} \\ 
      {H^m(\Hom_{\mathcal{C}}(I,I))} 
      \ar[d]_-{H^m(z^*)}^-{\sim} &&&
      \\ 
      {H^m(\Hom_{\mathcal{C}}(P,I))} 
      &&&
      {\Hom_{D(B)}(B,[m]\Hom_\mathcal{C}(P,I)).}
      \ar[lll]_-{\can}^-{\sim} 
    }
  \end{equation}
  Hence our functor is full and faithful. 
  The quasi-isomorphism
  $z_* \comp \beta \colon  B \ra \Hom_\mathcal{C}(P,I)$ of dg
  $B$-modules shows that it maps $I$ to a classical generator of
  the Karoubian category $\per(B).$ Since $[\mathcal{I}]$ is
  Karoubian we see that our functor is essentially surjective.

  \ref{hoI-via-B}:
  We first claim that the functor
  $\Hom_\mathcal{I}(I,-) \colon  [\mathcal{I}] \ra D(A)$
  preserves all coproducts where $A=\End_\mathcal{I}(I).$
  Let $(M_r)_{r \in R}$ be a family of objects of
  $[\mathcal{I}],$ and let 
  $\bigoplus M_r$ together with morphisms $s_r \colon  M_r \ra
  \bigoplus M_r$ be their coproduct in $[\mathcal{I}].$
  Lifts of the $s_r$ to closed degree zero morphisms in
  $\mathcal{I}$ induce a morphism
  \begin{equation*}
    \bigoplus \Hom_\mathcal{I}(I,M_r) \ra
    \Hom_\mathcal{I}(I,\bigoplus M_r) 
  \end{equation*}
  in $D(A).$ Taking the $m$-th cohomology 
  we obtain
  an isomorphism since $I$ is compact in
  $[\mathcal{I}].$ 

  Using the quasi-isomorphisms
  $z^* \colon  \Hom_{\mathcal{C}}(I, J) \ra \Hom_{\mathcal{C}}(P,
  J)$ for $J=M_r$ and $J=\bigoplus M_r$
  one shows that $\Hom_\mathcal{C}(P,-) \colon [\mathcal{I}] \ra
  D(B)$ preserves all
  coproducts, and this is clear for
  $\res^{\End_\mathcal{C}(P)}_{B} \colon D(\End_\mathcal{C}(P)) \ra
  D(B).$
  By assumption we have a quasi-isomorphism
  $z_* \comp \beta \colon  B \ra \Hom_\mathcal{C}(P,I)$ of dg
  $B$-modules and hence 
  $\Hom_\mathcal{C}(P,I)$ is a compact generator of $D(B).$
  The commutative diagram~\eqref{eq:HomP-res}
  from the proof of \ref{enum:hoI-via-B-small}
  shows that the functor $\res^{\End_\mathcal{C}(P)}_{B} \comp
  \Hom_\mathcal{C}(P,-)$ is full and 
  faithful on all shifts of $I.$
  Then one proceeds as in the proof of
  \cite[Lemma~4.2]{Keller-deriving-dg-cat}.  
\end{proof}


\section{\v{C}ech enhancements for locally integral schemes}
\label{sec:vcech-enhanc-loc-integral}



The aim of this section is to provide some results which are
used in \cite{valery-olaf-matrix-factorizations-and-motivic-measures}.
We employ them for proving
Theorem~\ref{t:mfPerf-Cechobj-smooth-vs-diagonal-sheaf-perfect}. 

In this section we asume that $X$ is a scheme 
satisfying the following 
condition~\ref{enum:GSP+}.
\begin{enumerate}[label=(GSP+)]
\item
  \label{enum:GSP+}
  $X$ is 
  Nagata,
  locally integral, 
  i.\,e.\ all local rings $\mathcal{O}_{X,x}$ are integral
  domains,
  and satisfies condition~\ref{enum:GSP}.
\end{enumerate}
This condition implies that the irreducible and
connected components of $X$ coincide; in particular, $X$ is the
finite disjoint union of integral schemes
(\cite[Ex.~3.16, 
p.~90]{goertz-wedhorn-AGI}).

The condition ``Nagata'' in \ref{enum:GSP+} can be replaced by 
``Noetherian'' if $X$ is excellent, for example if $X$ is of
finite type over a field or the integers.




\subsection{\v{C}ech enhancements}
\label{sec:vcech-enhancements}

We repeat 
(in a more general setting)
the definition
of the enhancement explained in
\cite[Lemma~6.7]{bondal-larsen-lunts-grothendieck-ring}.

Fix an ordered finite affine open covering $\mathcal{U}=(U_s)_{s \in S}$ 
of $X.$
Denote by $\mfPerf_\Cechobjstar(X)$
the smallest full dg subcategory of $C(\Qcoh(X))$ that
contains all  
objects $\mathcal{C}_*(P),$ for $P$ a vector bundle on $X,$ is closed under shifts in
both directions, under cones of
closed degree zero morphisms and under taking homotopy equivalent
objects. Then $\mfPerf_\Cechobjstar(X)$ is strongly pretriangulated.
If $R$ is a bounded complex of vector bundles we have 
$\mathcal{C}_*(R) \in \mfPerf_\Cechobjstar(X)$
(use brutal truncation on $R$).


\begin{proposition}
  [{cf.\ \cite[Lemma 6.7]{bondal-larsen-lunts-grothendieck-ring}}]
  \label{p:cech-*-object-enhancement}
  The canonical functor 
  \begin{equation*}
    \epsilon \colon  [\mfPerf_\Cechobjstar(X)] \ra \mfPerf(X)
  \end{equation*}
  is an equivalence of triangulated categories.
  Hence the dg category $\mfPerf_\Cechobjstar(X)$
  is naturally an enhancement of $\mfPerf(X).$ 
  We call it the \define{(object oriented) $*$-\v{C}ech
    enhancement}.
\end{proposition}

The proof of this proposition is similar to that of Proposition~\ref{p:abstract-cech-!-bounded-above-full-faithful}.

\begin{proof}
  By condition~\ref{enum:GSP}, any object of $\mfPerf(X)$ is
  isomorphic to a bounded complex $R$ of vector bundles, and then
  to $\mathcal{C}_*(R).$ This shows that $\epsilon$ is
  essentially surjective.

  Let $P$ and $Q$ be vector bundles on $X.$
  In order to prove that $\epsilon$ is full and faithful it is
  enough to show that 
  \begin{equation*}
    \Hom_{[C(\Qcoh(X))]}(\mathcal{C}_*(P),
    [m]\mathcal{C}_*(Q))
    \ra
    \Hom_{D(\Qcoh(X))}(\mathcal{C}_*(P), [m]\mathcal{C}_*(Q))
  \end{equation*}
  is an isomorphism, for any $m \in \DZ.$
  Applying brutal truncation to 
  $\mathcal{C}_*(Q)$ and passing 
  to direct summands reduces this to the 
  claim of the following Lemma~\ref{l:CD-isom-CP-VQ}.
\end{proof}

\begin{lemma}
  \label{l:CD-isom-CP-VQ}
  Let $P$ and $Q$ be vector bundles on $X$ and let $n \in \DZ.$
  Assume that $V=U_I$ for some non-empty
  subset $I \subset S.$
  Then 
  \begin{equation*}
    \Hom_{[C(\Qcoh(X))]}(\mathcal{C}_*(P), [n]\leftidx{_V}Q)
    \ra
    \Hom_{D(\Qcoh(X))}(\mathcal{C}_*(P), [n]\leftidx{_V}Q)
  \end{equation*}
  is an isomorphism.
\end{lemma}

\begin{proof}
  Let $p \colon P \ra \mathcal{C}_*(P)$ be the \v{C}ech
  resolution and consider the commutative diagram
  \begin{equation}
    \label{eq:diagram-for-CD-isom-CP-VQ}
    \xymatrix{
      {\Hom_{[C(\Qcoh(X))]}(\mathcal{C}_*(P), [n]\leftidx{_V}{Q}{})}
      \ar[r] \ar[d]_-{p^*} &
      {\Hom_{D(\Qcoh(X))}(\mathcal{C}_*(P), [n]\leftidx{_V}{Q}{})}
      \ar[d]_-{p^*}^-{\sim} \\
      {\Hom_{[C(\Qcoh(X))]}(P, [n]\leftidx{_V}{Q}{})}
      \ar[r]^-{\epsilon'} &
      {\Hom_{D(\Qcoh(X))}(P, [n]\leftidx{_V}{Q}{}).}
    }
  \end{equation}
  The right vertical map $p^*$ is an isomorphism since $p$ is a
  quasi-isomorphism. We will prove that $\epsilon'$ and the left
  vertical map $p^*$ are isomorphisms.
 
  Let us show first that $\epsilon'$ is an isomorphism.
  Note that $j \colon V \ra X$ is affine and open, hence 
  $\bR j_*=j_*$ and $\bL j^*=j^*,$ and recall that
  $\leftidx{_V}{Q}{}=j_*j^*(Q).$ Hence the adjunction
  $(j^*, j_*)$ show that it is enough to prove that
  \begin{equation*}
    \Hom_{[C(\Qcoh(V))]}(j^*(P), [n]j^*(Q))
    \ra
    \Hom_{D(\Qcoh(V))}(j^*(P), [n]j^*(Q))
  \end{equation*}
  is an isomorphism. But this is clear since $V$ is affine and
  $j^*(P)$ corresponds 
  to a projective module over $\mathcal{O}_V(V).$

  Note that $V$ is the finite disjoint union of its connected 
  components which are integral schemes. 
  Let $V'$ be a connected component of $V.$
  (Alternatively we could also assume without loss of generality
  that $X$ and hence $V$ are integral.)
  In order to show that the left vertical map $p^*$ in
  \eqref{eq:diagram-for-CD-isom-CP-VQ} is an
  isomorphism it is sufficient to prove that
  \begin{equation}
    \label{eq:CP-UIQ}
    p^* \colon  \Hom_{C(\Qcoh(X))}(\mathcal{C}_*(P),
    \leftidx{_{V'}}{Q}{}) \ra 
    \Hom_{C(\Qcoh(X))}(P, \leftidx{_{V'}}{Q}{})
  \end{equation}
  is a quasi-isomorphism.
  The right hand side is equal to 
  $\Hom_{C(\Qcoh(V'))}(P|_{V'}, Q|_{V'}).$ Its degree zero
  component is
  \begin{equation*}
    H:=\Hom_{\mathcal{O}_{V'}}(P|_{V'}, Q|_{V'}),
  \end{equation*}
  and all other components vanish.
  The graded components of the left hand side are direct sums
  of objects
  \begin{equation*}
    \Hom_{\mathcal{O}_X}(\leftidx{_{U_J}}{P}{},
    \leftidx{_{V'}}{Q}{})
  \end{equation*}
  for non-empty $J \subset S.$
  Lemma~\ref{l:Hom-UPVQ-star} 
  says that this is equal to $H$ if $V' \subset U_J,$ and zero
  otherwise.
  By assumption, $M(V'):=\{s \in S \mid V' \subset U_s\}$ is
  non-empty, and we have $V' \subset U_J$ if and only if $J
  \subset M(V').$ 
  Hence the 
  left hand side of \eqref{eq:CP-UIQ} is the chain complex
  \begin{equation*}
    \ldots \ra 
    \prod_{s_0, s_1 \in M(V'), \; s_0 < s_1} H \ra
    \prod_{s_0 \in M(V')} H \ra 0 \ra 
    \dots
  \end{equation*}
  of a (non-empty) simplex with coefficients in $H.$ The map
  \eqref{eq:CP-UIQ} is the augmentation map to $H$ which is
  a homotopy equivalence and in particular a quasi-isomorphism.
  This proves the lemma.
\end{proof}

\begin{remark}
  \label{rem:objects-of-cech-object-enhancement}
  If $R$ is a bounded complex of vector bundles on $X$ then
  we have observed above that $\mathcal{C}_*(R) \in
  \mfPerf_\Cechobjstar(X).$ Conversely, any object of 
  $\mfPerf_\Cechobjstar(X)$ is homotopy equivalent to an object of
  this form. Indeed, given $P \in \mfPerf_\Cechobjstar(X),$ there is
  a finite complex $R$ of vector bundles such that $P$ and ($R$
  and) $\mathcal{C}_*(R)$ are isomorphic in $\mfPerf(X),$
  by condition~\ref{enum:GSP}. 
  But then
  Proposition~\ref{p:cech-*-object-enhancement} shows that $P$ and
  $\mathcal{C}_*(R)$ are already isomorphic in
  $[\mfPerf_\Cechobjstar(X)].$
\end{remark}

\begin{remark}
  \label{rem:realization-vs-locally-integral}
  The realization functor \eqref{eq:*-cech-realization}
  defines a 
  quasi-equivalence $\mathcal{C}_* \colon \Cech_*(X) \ra
  \mfPerf_\Cechobjstar(X)$ because both dg categories enhance
  $\mfPerf(X),$ by
  Propositions~\ref{p:abstract-cech-*-object-enhancement} and
  \ref{p:cech-*-object-enhancement} (use
  \cite[Lemma~2.5]{valery-olaf-matrix-factorizations-and-motivic-measures}). 
  In
  particular, these two enhancements of $\mfPerf(X)$ are
  equivalent. 
  Remark~\ref{rem:objects-of-cech-object-enhancement} says that
  any object of $\mfPerf_\Cechobjstar(X)$ is homotopy equivalent
  to an object in the image of this realization functor.  If $X$
  is assumed to be integral the realization functor
  is not only full but also faithful, as already observed in
  Remark~\ref{rem:realization-full}.
\end{remark}





\subsubsection{Version for arbitrary sheaves}
\label{sec:vers-arbitr-sheav}

In the following section~\ref{sec:lifting-duality} we need a
small generalization of the previous constructions and results
(cf.\ Remark~\ref{rem:why-not-Qcoh}).
%
%

Recall
the equivalence $\mfPerf(X)
\sira \mfPerf'(X).$
Denote by $\mfPerf'_\Cechobjstar(X)$ the smallest full dg subcategory of
$C(\Sh(X))$ that contains all objects of $\mfPerf_\Cechobjstar(X)$
and is 
closed under taking homotopy equivalent objects;  
then the inclusion $\mfPerf_\Cechobjstar(X) \ra \mfPerf'_\Cechobjstar(X)$
of dg categories is a quasi-equivalence and 
$\mfPerf'_{\Cechobjstar}$ 
is an enhancement of $\mfPerf'(X).$ 

\subsection{Lifting the duality}
\label{sec:lifting-duality}

Keep the ordered finite affine open covering
$\mathcal{U}=(U_s)_{s \in S}$ from above and abbreviate
$\mathcal{O}=\mathcal{O}_X.$ The functor
\begin{equation*}
  \bR \sheafHom(-, \mathcal{O}) \colon D(\Sh(X))^\opp \ra
  D(\Sh(X))  
\end{equation*}
induces an
auto-equivalence $D \colon  \mfPerf'(X)^\opp \ra \mfPerf'(X)$
satisfying $\id \sira D^2.$
If $E$ is a bounded complex of vector bundles we have $D(E) \cong
E^\cek =\sheafHom(E,\mathcal{O})$ and $E=(E^\cek)^\cek.$
Our aim is to show that we can lift the auto-duality $D$ to the
enhancement $\mfPerf'_\Cechobjstar(X).$
For this we consider the dg functor
\begin{equation*}
  \tildew{D}:= \sheafHom(-, \mathcal{C}_*(\mathcal{O})) \colon 
  C(\Sh(X))^\opp \ra C(\Sh(X)).
\end{equation*}

\begin{remark}
  \label{rem:why-not-Qcoh}
  In general, $\tildew{D}$ does not preserve
  $C(\Qcoh(X))$ (Lemma~\ref{l:sheafHom-UPVQ}).
  This forces us to leave the quasi-coherent world and to work
  with $\mfPerf'_\Cechobjstar(X)$ instead of $\mfPerf_\Cechobjstar(X).$ 
\end{remark}

\begin{remark}
  \label{rem:Cech-of-dual}
  If $E$ and $F$ are vector bundles on $X$ and $v \colon  V \subset X$ is
  an open subscheme, then
  $v_*v^*\sheafHom(E,F)
  = v_*\sheafHom(v^*(E),v^*(F))
  = \sheafHom(E,v_*v^*(F))$
  where the first equality is
  obvious (see Lemma~\ref{l:pullback-sheafHom})
  and the second equality is
  the adjunction. 
  Hence we can and will identify
  $\mathcal{C}_*(E^\cek)=\sheafHom(E,
  \mathcal{C}_*(\mathcal{O})).$
\end{remark}

\begin{lemma} 
  \label{l:sheafHomCC-vs-CsheafHom} 
  Let $E$ be a vector bundle (or a bounded complex of vector
  bundles) on $X.$ Consider the canonical quasi-isomorphism
  $\alpha \colon E \ra \mathcal{C}_*(E).$ Then the induced morphism
  \begin{equation*}
    \tildew{D}(\alpha) \colon 
    \tildew{D}(\mathcal{C}_*(E)) =
    \sheafHom(\mathcal{C}_*(E),\mathcal{C}_*(\mathcal{O})) 
    \ra 
    \tildew{D}(E)
    =
    \sheafHom(E,\mathcal{C}_*(\mathcal{O}))
    = \mathcal{C}_*(E^\cek)
  \end{equation*}
  is a homotopy equivalence.
\end{lemma}

\begin{proof}
  Using brutal truncation we can assume that $E$ is
  a vector bundle.
  Write $\alpha^*:= [-1]\tildew{D}(\alpha).$
  We will show that the complex
  $\Cone(\alpha^*)=\sheafHom(\Cone(\alpha),
  \mathcal{C}_*(\mathcal{O}))$  is
  contractible (this certainly implies the lemma).
  We can also assume that $X$ is irreducible: 
  decompose $E$ into a direct sum of vector bundles according
  to the decomposition of $X$ into irreducible components.

  First we filter $\Cone(\alpha^*)$ according to the
  ``target'', so that the subquotients are the complexes
  $\Cone(\alpha^*)_I:=\sheafHom(\Cone(\alpha), \leftidx{_{U_I}}{\mathcal{O}})$ 
  labeled by non-empty subsets
  $I\subset S.$ 
  Fix such an $I.$
  It suffices to prove that
  each $\Cone(\alpha^*)_I$ is contractible.
  Note that $\Cone(\alpha^*)_I$ as a graded sheaf
  is the direct sum of all
  $\sheafHom(\leftidx{_{U_J}}{E}{},\leftidx{_{U_I}}{\mathcal{O}})$
  (shifted to the obvious degree)
  where 
  $J\subset S$ (here $J = \emptyset$ is allowed, then
  $U_\emptyset=X$ and $\leftidx{_{U_\emptyset}}{E}{}=E$).

  Observe that the sheaf
  $\sheafHom(\leftidx{_{U_J}}{E}{},\leftidx{_{U_I}}{\mathcal{O}})$
  depends only on (the fixed set $U_I$ and) the intersection $U_J
  \cap U_I=U_{I \cup J}.$ Indeed, let $i \colon U_I\ra X$
  and $l \colon U_I\cap U_J\ra U_I$ be the open immersions. Then
  \begin{equation}
    \label{eq:individual-term}
    \sheafHom(\leftidx{_{U_J}}{E}{},
    \leftidx{_{U_I}}{\mathcal{O}})
    \sira
    i_*l_!\sheafHom(E|_{U_I\cap U_J},\mathcal{O}_{U_I\cap U_J}) 
    \sila
    \sheafHom(\leftidx{_{U_{I \cup J}}}{E}{},
    \leftidx{_{U_I}}{\mathcal{O}})
  \end{equation}
  by Lemma~\ref{l:sheafHom-UPVQ} (note that the inclusion
  morphism $U_J \subset X$ is 
  affine, even for $J=\emptyset,$ and that $U_I$ is
  Nagata integral (if non-empty) since $X$ is Nagata integral) or
  trivially if 
  $U_I =\emptyset.$ So it is 
  natural to consider all subsets $J \subset S$ for which $I \cup
  J$ is constant.

  Accordingly, we filter the complex $\Cone(\alpha^*)_I$ 
  so that the subquotients $\Cone(\alpha^*)_I^K$ are
  labeled by subsets $K \subset S \setminus I$ (possibly empty)
  and are (as graded sheaves) direct sums of all
  $\sheafHom(\leftidx{_{U_J}}{E}{},\leftidx{_{U_I}}{\mathcal{O}})$
  (shifted to the obvious degree) 
  where $K \subset J \subset (I \cup K)$ (so $I \cup J$ is
  constant and equal to $I \cup K$). Fix such a $K.$

  It follows from \eqref{eq:individual-term} that the complex
  $\Cone(\alpha^*)_I^K$ is isomorphic to the augmented chain
  complex of a (non-empty) simplex with coefficients in
  $\sheafHom(\leftidx{_{U_{I \cup K}}}{E}{},
  \leftidx{_{U_I}}{\mathcal{O}}),$ and hence
  contractible. This implies that $\Cone(\alpha^*)_I$ 
  is contractible.
\end{proof}



\begin{corollary}
  \label{c:duality-lift-well-defined}
  The dg functor $\tildew{D}$ induces a dg functor 
  \begin{equation*}
    \tildew{D} = \sheafHom(-, \mathcal{C}_*(\mathcal{O})) \colon 
    \mfPerf'_\Cechobjstar(X)^\opp \ra \mfPerf'_\Cechobjstar(X)
  \end{equation*}
  which lifts
  the duality $D$ in the sense that the diagram
  \begin{equation*}
    \xymatrix{
      {[\mfPerf'_\Cechobjstar(X)]^\opp} \ar[r]^-{[\tildew{D}]}
      \ar[d]^-\sim &
      {[\mfPerf'_\Cechobjstar(X)]}
      \ar[d]^-\sim \\
      {\mfPerf'(X)^\opp} \ar[r]^-{D} & {\mfPerf'(X)}
    }
  \end{equation*}
  commutes up to an isomorphism of functors.
\end{corollary}

\begin{proof}
  Let $E$ be a bounded complex of vector bundles. 
  Lemma~\ref{l:sheafHomCC-vs-CsheafHom} shows that   
  $\tildew{D}(\mathcal{C}_*(E))$
  is homotopy equivalent to $\mathcal{C}_*(E^\cek)$
  and hence in $\mfPerf'_\Cechobjstar(X).$
  This implies the first claim.

  For the second claim we need to define an isomorphism between
  the two compositions in the above diagram. On
  $\mathcal{C}_*(E)$ for $E$ as above we define it to be the
  isomorphisms $D(\mathcal{C}_*(E)) \sira D(E) \cong E^\cek$
  followed by the isomorphism induced by the
  quasi-equivalence $E^\cek \ra
  \mathcal{C}_*(E^\cek)$ followed by the inverse of the
  isomorphism induced by the
  homotopy equivalence $\tildew{D}(\mathcal{C}_*(E)) \ra
  \tildew{D}(E)=\mathcal{C}_*(E^\cek)$ from
  Lemma~\ref{l:sheafHomCC-vs-CsheafHom}.  This is easily checked
  to be compatible with
  morphisms, and sufficient by
  Remark~\ref{rem:objects-of-cech-object-enhancement}.
\end{proof}

The canonical morphism
\begin{align}
  \label{eq:theta-F}
  \theta_F  \colon  F &
  \ra
  \tildew{D}^2(F)=\sheafHom(\sheafHom(F,
  \mathcal{C}_*(\mathcal{O})),\mathcal{C}_*(\mathcal{O})),\\
  \notag
  f & \mapsto (\lambda \mapsto \lambda(f)),
\end{align}
(for $F \in C(\Sh(X))$) defines a morphism
$\theta \colon  \id \ra \tildew{D}^2$ of dg functors 
$C(\Sh(X)) \ra C(\Sh(X)),$
and, by Corollary~\ref{c:duality-lift-well-defined}, 
also of dg functors
$\mfPerf'_\Cechobjstar(X) \ra \mfPerf'_\Cechobjstar(X).$

  

\begin{lemma} 
  \label{l:map-to-double-dual-homotopy-equi}
  For each $F \in \mfPerf'_\Cechobjstar(X),$ the morphism
  $\theta_F$ 
  in \eqref{eq:theta-F}
  is a homotopy equivalence.
\end{lemma}

\begin{proof}
  Assume that
  $F=\mathcal{C}_*(E)$ for $E$ a vector bundle on $X.$
  It is certainly enough to show the claim in this special case.
  Note that $\theta_F$ is a closed degree
  zero morphism in $\mfPerf'_\Cechobjstar(X).$ Hence 
  (the $\mfPerf'$-version of)
  Proposition~\ref{p:cech-*-object-enhancement} 
  shows that it is enough to prove
  that $\theta_F$ is a quasi-isomorphism.

  The canonical quasi-isomorphism $\alpha_E  \colon E\ra
  \mathcal{C}_*(E)$ induces the 
  commutative diagram
  \begin{equation*}
    \xymatrix{
      {E} \ar[r]^-{\theta_E} \ar[d]^-{\alpha_E} &
      {\tildew{D}^2(E)
        =\sheafHom(\sheafHom(E,\mathcal{C}_*(\mathcal{O})),
        \mathcal{C}_*(\mathcal{O}))}
      \ar[d]^-{\tildew{D}^2(\alpha_E)} \\
      {\mathcal{C}_*(E)} \ar[r]^-{\theta_F} &
      {\tildew{D}^2(\mathcal{C}_*(E))
        =\sheafHom(\sheafHom(\mathcal{C}_*(E), 
        \mathcal{C}_*(\mathcal{O})),\mathcal{C}_*(\mathcal{O}))}
    }
  \end{equation*}
  of closed degree zero morphisms.
  The morphism $\tildew{D}(\alpha_E)$ is a homotopy equivalence by 
  Lemma~\ref{l:sheafHomCC-vs-CsheafHom}, so a fortiori 
  the morphism $\tildew{D}^2(\alpha_E)$
  is a homotopy equivalence. Therefore it suffices to prove that
  $\theta_E$ is a quasi-isomorphism.

  Let $i \colon \mathcal{O} \ra
  \mathcal{C}_*(\mathcal{O})$ 
  be the canonical quasi-isomorphism.
  Note that the morphism 
  \begin{equation*}
    i_* \colon  E^\cek=\sheafHom(E, \mathcal{O}) \ra 
    \sheafHom(E,
    \mathcal{C}_*(\mathcal{O}))    
    =\tildew{D}(E)
  \end{equation*}
  coincides with the canonical
  quasi-isomorphism 
  $\alpha_{E^\cek} \colon  E^\cek \ra
  \mathcal{C}_*(E^\cek)$
  if we identify the targets according to 
  Remark~\ref{rem:Cech-of-dual}.
  Lemma~\ref{l:sheafHomCC-vs-CsheafHom} then implies that
  $\tildew{D}(i_*)$ is a homotopy equivalence.
  It is easy to check that the diagram
  \begin{equation*}
    \xymatrix{
      {E} \ar[r]^-{\theta_E} 
      \ar[d]^-{\sim} &
      {\tildew{D}^2(E)}
      \ar[d]^-{\tildew{D}(i_*)} \\
      {E^{\cek\cek}=\sheafHom(E^\cek, \mathcal{O})}
      \ar[r]^-{i_*} &
      {
        \sheafHom(E^\cek,
        \mathcal{C}_*(\mathcal{O}))
        =
        \tildew{D}(E^\cek)
      } 
    }
  \end{equation*}
  commutes, where the left vertical arrow is the canonical
  isomorphism. 
  Obviously, $i_*$ is a quasi-isomorphism. Hence the same is true
  for $\theta_E.$ 
\end{proof} 

\begin{corollary}
  \label{c:duality-lifted-to-enhancement}
  The dg functor 
  $\tildew{D}= \sheafHom(-, \mathcal{C}_*(\mathcal{O})) \colon 
  \mfPerf'_\Cechobjstar(X)^\opp \ra \mfPerf'_\Cechobjstar(X)$
  is a quasi-equivalence. The induced functor $[\tildew{D}]$ on
  homotopy categories is an equivalence and a duality in the
  sense that the natural morphism $\theta \colon  \id \ra
  [\tildew{D}]^2$ is an isomorphism.
\end{corollary}

\begin{proof}
  Lemma~\ref{l:map-to-double-dual-homotopy-equi}
  shows that $\theta \colon  \id \ra [\tildew{D}]^2$ is an isomorphism. 
  In particular, $[\tildew{D}]$ is an equivalence, and
  $\tildew{D}$ is a quasi-equivalence.
\end{proof}

\subsection{Products and enhancements}
\label{sec:prod-enhanc}

We work now over a field $k.$ 
Let $X$ and $Y$ be $k$-schemes such that $X,$ $Y,$ and $X \times
Y$ satisfy condition~\ref{enum:GSP+}.

\subsubsection{Enhancements for products}
\label{sec:products}


Fix ordered finite affine open coverings $\mathcal{U}=(U_s)_{s \in S}$ 
of $X$
and $\mathcal{V}=(V_t)_{t \in T}$ of $Y.$
On $X \times Y$ we consider the finite affine open covering
$\mathcal{W}:= \mathcal{U} \times \mathcal{V}=(W_{(s,t)})_{(s,t)
  \in S \times T}$ where $W_{(s,t)}=U_s \times V_t$ (and equip
$S \times T$ with some total ordering).  
Let $\mfPerf_\Cechobjstarbox(X \times Y)$ be the smallest full dg
subcategory 
of $C(\Qcoh(X \times Y))$ that contains all objects
$\mathcal{C}_*(P) \boxtimes \mathcal{C}_*(Q)$ for $P$ and
$Q$ vector bundles on $X$ and $Y,$
respectively,
all objects $\mathcal{C}_*(R)$ for $R$ a vector bundle on $X
\times Y,$ and is closed under shifts in both directions, cones
of  
closed degree zero morphisms and under taking homotopy equivalent
objects. It is strongly pretriangulated.
We will see in
Corollary~\ref{c:cech-*-object-enhancement-product}
that we have defined nothing new.

\begin{proposition}
  \label{p:cech-*-object-enhancement-product}
  The dg category $\mfPerf_\Cechobjstarbox(X \times Y)$
  is naturally an enhancement of $\mfPerf(X \times Y).$ 
\end{proposition}

\begin{proof}
  We essentially repeat the proof of
  Proposition~\ref{p:cech-*-object-enhancement}. 
  Note that $\mfPerf_\Cechobjstar(X \times Y) \subset
  \mfPerf_\Cechobjstarbox(X \times Y).$ 
  It remains to prove that
  \begin{equation*}
    \Hom_{[C(\Qcoh(X \times Y))]}(A, [m]B)
    \ra
    \Hom_{D(\Qcoh(X \times Y))}(A, [m]B)
  \end{equation*}
  is an isomorphism for any $m \in \DZ$ in the four cases where
  $A$ and $B$ are of the form $\mathcal{C}_*(P) \boxtimes
  \mathcal{C}_*(Q)$ or $\mathcal{C}_*(R).$ Note that
  $\mathcal{C}_*(P) \boxtimes \mathcal{C}_*(Q)$ is built up
  from objects $\leftidx{_{U_{S'}}}{P}{} \boxtimes
  \leftidx{_{V_{T'}}}{Q}{} \sira \leftidx{_{U_{S'} \times V_{T'}}}{(P
    \boxtimes Q)}{}$ for non-empty subsets $S' \subset S$ and $T'
  \subset T$ (see Lemma~\ref{l:boxtimes-restriction-extension}.\ref{enum:boxtimes-*-extension}), and that
  $\mathcal{C}_*(R)$ is built up from objects
  $\leftidx{_{U_{S'} \times V_{T'}}}{R}{}$ (where $S'$ (resp.\
  $T'$) is the image of some non-empty subset of $S \times T$
  under the first (resp.\ second) projection).
  Hence the two cases where $A$ is of the form $\mathcal{C}_*(R)$
  can be treated as in the proof of
  Proposition~\ref{p:cech-*-object-enhancement}. The remaining two
  cases follow from the following claim.

  Let $P,$ $Q,$ $R$ be vector bundles on $X, Y,$ and $X \times Y,$
  respectively, 
  and let $n \in \DZ.$
  Assume that $V'$ is an irreducible component of $U_I \times
  V_J$ where $I \subset S$ and $J \subset T$ are non-empty 
  subsets. Then 
  \begin{equation*}
    \Hom_{[C(\Qcoh(X \times Y))]}(\mathcal{C}_*(P) \boxtimes
    \mathcal{C}_*(Q), [n]\leftidx{_{V'}}R)
    \ra
    \Hom_{D(\Qcoh(X \times Y))}(\mathcal{C}_*(P) \boxtimes
    \mathcal{C}_*(Q), [n]\leftidx{_{V'}}R)
  \end{equation*}
  is an isomorphism.

  This claim is proved as Lemma~\ref{l:CD-isom-CP-VQ},
  using $P \boxtimes Q \ra \mathcal{C}_*(P) \boxtimes
  \mathcal{C}_*(Q).$
  The analog of the map $\epsilon'$ in
  \eqref{eq:diagram-for-CD-isom-CP-VQ} 
  is an isomorphism by the same argument.
  The analog of the left vertical map $p^*$ 
  comes from the
  augmentation map of the chain complex of the product of two 
  (non-empty) simplices and is therefore an isomorphism:
  consider $M(V'):=\{(s,t) \in S \times T \mid V' \subset
  U_s \times V_t\}$ and note that $M(V') = M(\pr_1(V')) \times
  M(\pr_2(V'))$ where $X \xla{\pr_1} X \times Y \xra{\pr_2} Y$
  are the projections and
  $M(\pr_1(V'))$ and $M(\pr_2(V'))$ are defined in the obvious
  way.  
\end{proof}

\begin{corollary}
  \label{c:cech-*-object-enhancement-product}
  We have 
  $\mfPerf_\Cechobjstar(X \times Y)=\mfPerf_\Cechobjstarbox(X
  \times Y).$ 
\end{corollary}

\begin{proof}
  Obviously 
  $\mfPerf_\Cechobjstar(X \times Y) \subset \mfPerf_\Cechobjstarbox(X
  \times Y).$ 
  Let $P$ and $Q$ be vector bundles on $X$ and $Y,$ respectively.
  Since $\mathcal{C}_*(P) \boxtimes \mathcal{C}_*(Q)$ and
  $\mathcal{C}_*(P \boxtimes Q)$ are isomorphic in $\mfPerf(X
  \times Y),$ they are isomorphic in 
  $\mfPerf_\Cechobjstarbox(X \times Y),$
  by Proposition~\ref{p:cech-*-object-enhancement-product}.
  This means that 
  $\mathcal{C}_*(P) \boxtimes \mathcal{C}_*(Q)$ is homotopy
  equivalent to 
  $\mathcal{C}_*(P \boxtimes Q)$ and hence already in
  $\mfPerf_\Cechobjstar(X \times Y).$
\end{proof}

Let $\Delta \colon X \ra X \times X$ be the diagonal (closed)
immersion. We consider the product covering $\mathcal{U} \times
\mathcal{U}$ of $X \times X.$

\begin{lemma}
  \label{l:boxtimes-cech-object-to-diagonal-cech-K-iso-D}
  Assume that $X \times X$ satisfies condition~\ref{enum:GSP+}.
  Let $P, Q, R$ be vector bundles on $X,$ and let $m \in \DZ.$
  Then the canonical map
  \begin{equation*}
    \leftidx{_{[C(\Qcoh(X \times X))]}^m}{(\mathcal{C}_*(P) \boxtimes
      \mathcal{C}_*(Q),
      \Delta_*(\mathcal{C}_*(R)))}{}
    \ra
    \leftidx{_{D(\Qcoh(X \times X))}^m}{(\mathcal{C}_*(P) \boxtimes
      \mathcal{C}_*(Q),
      \Delta_*(\mathcal{C}_*(R)))}{}
  \end{equation*}
  is an isomorphism.
  Here we write 
  $\leftidx{_?^m}{(-,-)}{}$ instead of $\Hom_{?}(-,[m]-).$
\end{lemma}

\begin{proof}
  We follow the idea of the proof of
  Proposition~\ref{p:cech-*-object-enhancement}.
  We need to show the
  following two claims where $V=U_I$ for some non-empty subset $I
  \subset S.$
  \begin{enumerate}
  \item
    \label{enum:boxtimes-PQR-KD}
    $\leftidx{_{[C(\Qcoh(X \times X))]}}{(P \boxtimes Q,
      [n]\Delta_*(\leftidx{_{V}}{R}{}))}{}
    \ra
    \leftidx{_{D(\Qcoh(X \times X))}}{(P \boxtimes Q,
      [n]\Delta_*(\leftidx{_{V}}{R}{}))}{}$
    is an isomorphism, for any $n \in \DZ.$
  \item
    \label{enum:boxtimes-PQR-KK}
    Let 
    $p \colon P \ra \mathcal{C}_*(P)$ and $q \colon Q \ra
    C_*(Q)$ denote the obvious maps, and let $V' \subset V$ be
    an irreducible component. 
    Then
    \begin{equation*}
      (p \boxtimes q)^* \colon  
      \leftidx{_{C(\Qcoh(X \times X))}}{(\mathcal{C}_*(P) \boxtimes
        \mathcal{C}_*(Q),
        \Delta_*(\leftidx{_{V'}}{R}{}))}{}
      \ra
      \leftidx{_{C(\Qcoh(X \times X))}}{(P \boxtimes Q,
        \Delta_*(\leftidx{_{V'}}{R}{}))}{}    
    \end{equation*}
    is a quasi-isomorphism.
  \end{enumerate}

  Proof of \ref{enum:boxtimes-PQR-KD}:
  Note that 
  $\Delta$ is affine as a closed immersion and hence
  $\bR \Delta_*=\Delta_*.$ Moreover we have 
  $\bL \Delta^*(P \boxtimes Q)=
  \Delta^*(P \boxtimes Q) =
  P \otimes Q.$
  Using the obvious adjunctions we can hence proceed as in the
  proof of Proposition~\ref{p:cech-*-object-enhancement}.

    

  Proof of \ref{enum:boxtimes-PQR-KK}:
  The right hand side is concentrated in degree zero and equal to
  \begin{equation*}
   H:=\Hom_{\mathcal{O}_X}((P \otimes Q)|_{V'}, R|_{V'})
  \end{equation*}
  there.
  The graded components of the left hand side are direct sums of
  objects 
  \begin{equation*}
    \Hom_{\mathcal{O}_{X \times X}}
    (\leftidx{_{U_K}}{P}{} \boxtimes
   \leftidx{_{U_L}}{Q}{},
    \Delta_*(\leftidx{_{V'}}{R}{}))
    \sila
    \Hom_{\mathcal{O}_X}
    (\leftidx{_{U_K\cap U_L}}{(P \otimes Q)}{},
    \leftidx{_{V'}}{R}{})
  \end{equation*}
  where $K, L \subset S$ are non-empty subsets and the
  isomorphism follows from Corollary~\ref{c:combined}.  Using
  Lemma~\ref{l:Hom-UPVQ-star} this implies that $(p \boxtimes
  q)^*$ is the augmentation map of the chain complex of the
  product of two (non-empty) simplices with coefficients in $H$
  and therefore a quasi-isomorphism.
  %
  %
\end{proof}

\subsubsection{Enhancements and external tensor product}
\label{sec:enhanc-extern-tens}

We come back to the product situation $X \times Y$ with covering
$\mathcal{U} \times \mathcal{V}.$

\begin{lemma}
  \label{l:dg-functor-boxtimes-full-and-faithful}
  The obvious dg functor
  \begin{equation}
    \label{eq:dg-functor-boxtimes}
    \boxtimes \colon  \mfPerf_\Cechobjstar(X) \otimes
    \mfPerf_\Cechobjstar(Y) \ra 
    \mfPerf_{\Cechobjstar}(X\times Y)
  \end{equation}
  is quasi-fully faithful, i.\,e.\ induces quasi-isomorphisms
  between morphisms spaces.
\end{lemma}

\begin{proof}
  Corollary~\ref{c:cech-*-object-enhancement-product}
  makes sure that this functor is well-defined.
  We can assume without loss of generality that $X$ and $Y$ are
  integral. 
  Let $T \subset X,$ $T' \subset X$, $U \subset Y$ and 
  $U' \subset Y$ be affine open subsets, and let
  $E,$ $E'$ be vector bundles on $X$ and $F,$ $F'$ be vector
  bundles 
  on $Y.$ We claim that the 
  morphism
  \begin{equation*}
    \boxtimes \colon 
    \Hom_{\mathcal{O}_X}(\leftidx{_T}{E}{}, \leftidx{_{T'}}{E'}{})
    \otimes
    \Hom_{\mathcal{O}_Y}(\leftidx{_U}{F}{}, \leftidx{_{U'}}{F'}{})
    \ra
    \Hom_{\mathcal{O}_{X \times Y}}(
    \leftidx{_{T \times U}}{(E \boxtimes F)}{}, 
    \leftidx{_{T' \times U'}}{(E' \boxtimes F')}{})
  \end{equation*}
  is an isomorphism
  (cf.\ Lemma~\ref{l:boxtimes-restriction-extension}.\ref{enum:boxtimes-*-extension}).
  This is trivial if one of the sets $T,$ $T',$ $U,$ $U'$ is
  empty, 
  and otherwise the condition
  ($T' \subset T$ and $U' \subset U$) is equivalent to 
  the condition $T' \times U' \subset T \times U.$
  Now we use
  Lemma~\ref{l:Hom-UPVQ-star}.
  If these conditions do not hold, both sides are zero;
  if they hold we need to show that
  \begin{multline*}
    \boxtimes \colon 
    \Hom_{\mathcal{O}_{T'}}(E|_{T'}, E'|_{T'})
    \otimes
    \Hom_{\mathcal{O}_{U'}}(F|_{U'}, F'|_{U'})\\
    \ra
    \Hom_{\mathcal{O}_{T' \times U'}}((E \boxtimes F)|_{T' \times U'},
    (E' \boxtimes F')|_{T' \times U'})
  \end{multline*}
  is an isomorphism. But since $T'$ and $U'$ are affine and $E$
  and $F$ are vector bundles this is obvious.
  This proves our claim, and it is easy to deduce the lemma.
  %
\end{proof}

\begin{remark}
  \label{rem:dg-lift-of-boxtimes}
  Consider the composition
  \begin{equation}
    \label{eq:dg-lift-of-boxtimes}
    \mfPerf_\Cechobjstar(X) \times \mfPerf_\Cechobjstar(Y)
    \xra{\otimes}
    \mfPerf_\Cechobjstar(X) \otimes \mfPerf_\Cechobjstar(Y)
    \xra{\boxtimes}
    \mfPerf_\Cechobjstar(X \times Y)
  \end{equation}
  where the first functor is the obvious dg bifunctor and the
  second functor is the dg functor
  \eqref{eq:dg-functor-boxtimes}.
  Passing to homotopy categories defines the upper row in the
  following diagram.
  \begin{equation*}
    \xymatrix{
      {[\mfPerf_\Cechobjstar(X)] \times [\mfPerf_\Cechobjstar(Y)]}
      \ar[r]^-{\otimes} \ar[d]^-{\epsilon \times \epsilon} &
      [\mfPerf_\Cechobjstar(X) \otimes \mfPerf_\Cechobjstar(Y)]
      \ar[r]^-{\boxtimes} &
      {[\mfPerf_\Cechobjstar(X \times Y)]} \ar[d]^-{\epsilon}\\
      {\mfPerf(X) \times \mfPerf(Y)}
      \ar[rr]^-{\boxtimes} &&
      {\mfPerf(X \times Y)}.
    }
  \end{equation*}
  This diagram is commutative. This means that
  the dg bifunctor \eqref{eq:dg-lift-of-boxtimes} is a lift of 
  the bifunctor $\boxtimes \colon  {\mfPerf(X) \times \mfPerf(Y)}
  \ra {\mfPerf(X \times Y)}$ of triangulated
  categories to the indicated enhancements.
\end{remark}





\subsubsection{Version for arbitrary sheaves on products}
\label{sec:vers-arbitr-sheav-products}

It is clear that the previous results of
section~\ref{sec:prod-enhanc}
remain true if we 
replace $\mfPerf_\Cechobjstar$ by $\mfPerf'_\Cechobjstar$ and
$\mfPerf$ by $\mfPerf'$ and
$C(\Qcoh(-))$ by
$C(\Sh(-))$ and
$D(\Qcoh(-))$ by
$D(\Sh(-)).$ 


 







\subsection{Homological smoothness and the structure sheaf of the
  diagonal -- the locally integral case}
\label{sec:homol-smoothn-struct}

\begin{theorem} 
  \label{t:mfPerf-Cechobj-smooth-vs-diagonal-sheaf-perfect}
  Let $X$ be a scheme over a field $k$ such that $X$ and $X
  \times X$
  satisfy 
  condition~\ref{enum:GSP+}.
  Let $\Delta \colon  X \ra X \times X$ be the diagonal (closed)
  immersion. 
  Then $\mfPerf(X)$ is 
  smooth over $k$ if and only if $\Delta_*(\mathcal{O}_X) \in
  \mfPerf(X \times X).$
\end{theorem}

This theorem is weaker than
Theorem~\ref{t:mfPerf-abstract-Cechobj-smooth-vs-diagonal-sheaf-perfect}
because a Nagata quasi-compact scheme is Noetherian. Anyway we
include a proof based on the results of this section.

\begin{proof}
  Fix an ordered finite affine open covering 
  $\mathcal{U}=(U_s)_{s \in S}$ of $X.$ 
  We can test $k$-smoothness of $\mfPerf(X) \sira \mfPerf'(X)$ on
  the 
  enhancement  
  $\mfPerf'_\Cechobjstar(X),$ cf.\
  Remark~\ref{rem:realization-vs-locally-integral}
  and section~\ref{sec:vers-arbitr-sheav}.
 
  As in the proof of
  Theorem~\ref{t:mfPerf-abstract-Cechobj-smooth-vs-diagonal-sheaf-perfect}
  we find a vector bundle
  $E$ on $X$ that is a classical generator of $\mfPerf'(X)$
  and see that 
  $k$-smoothness of $\mfPerf(X)$ is equivalent to
  $k$-smoothness of the dg algebra
  \begin{equation*}
    A:=\End_{C(\Sh(X)}(\mathcal{C}_*(E)),  
  \end{equation*}
  i.\,e.\ to the condition $A \in \per(A \otimes A^{\opp}).$
  We also know that
  $E^\cek$ is a classical generator of $\mfPerf'(X)$
  and that 
  $E \boxtimes E^\cek$ is a classical
  generator of $\mfPerf'(X \times X).$

  We use the $*$-\v{C}ech enhancement $\mfPerf'_\Cechobjstar(X
  \times X) 
  =\mfPerf'_\Cechobjstarbox(X \times X)$ of $\mfPerf'(X \times
  X)$ with 
  respect to $\mathcal{U} \times \mathcal{U}.$
  A lift of $E \boxtimes E^\cek$ to this enhancement is given by
  the object
  \begin{equation*}
    P:=\mathcal{C}_*(E) \boxtimes \tildew{D}(\mathcal{C}_*(E))
  \end{equation*}
  because Lemma~\ref{l:sheafHomCC-vs-CsheafHom}
  provides the homotopy equivalence
  \begin{equation}
    \label{eq:P-hequi-boxtimes}
    \id \boxtimes \tildew{D}(\alpha) \colon 
    P \ra \mathcal{C}_*(E)
    \boxtimes \mathcal{C}_*(E^\cek). 
  \end{equation}
  Note that
  \begin{equation*}
    \tildew{D} \colon  A^\opp =
    \End_{C(\Sh(X))}(\mathcal{C}_*(E))^\opp
    \ra
    \End_{C(\Sh(X))}(\tildew{D}(\mathcal{C}_*(E)))
  \end{equation*}
  is a quasi-isomorphism of dg algebras
  by Corollary~\ref{c:duality-lifted-to-enhancement}. 
  Since we work over a field,
  this and
  Lemma~\ref{l:dg-functor-boxtimes-full-and-faithful}
  imply that both arrows in 
  \begin{multline}
    \label{eq:AoAopp-EndP}
    A \otimes A^\opp
    \xra{\id \otimes \tildew{D}}
    \End_{C(\Sh(X))}(\mathcal{C}_*(E))
    \otimes
    \End_{C(\Sh(X))}(\tildew{D}(\mathcal{C}_*(E)))\\
    \xra{\boxtimes}
    \End_{C(\Sh(X \times X))}(P)
  \end{multline}
  are quasi-isomorphisms of dg algebras.  

  As in the proof of
  Theorem~\ref{t:mfPerf-abstract-Cechobj-smooth-vs-diagonal-sheaf-perfect}
  we use the 
  enhancement $C^\hinj_\Qcoh(\Sh(X \times X))$ of  
  $D(\Qcoh(X \times X)) \sira D_\Qcoh(\Sh(X \times X)).$
  Let $z \colon  P \ra I$ be a quasi-isomorphism with  
  $I \in C^\hinj(\Sh(X \times X)).$
  Note that $I$ is isomorphic to the vector bundle $E \boxtimes
  E^\cek$ in $D(\Sh(X \times X))$ and hence a compact generator of 
  $[C^\hinj_\Qcoh(\Sh(X \times X))] \sira
  D_\Qcoh(\Sh(X \times X)).$
  We now can apply
  Proposition~\ref{p:homotopy-categories-triang-via-dg-algebras}.\ref{hoI-via-B}
  to the dg subcategory $C^\hinj_\Qcoh(\Sh(X \times X))$ of
  $C(\Sh(X \times X))$ 
  and $\beta$ there the composition in
  \eqref{eq:AoAopp-EndP}.
  Note for this that
  \begin{equation*}
    z_* \colon \Hom_{C(\Sh(X \times X))}(P,P) \ra \Hom_{C(\Sh(X \times X))}(P,I)
  \end{equation*}
  is a quasi-isomorphism: the induced map on the $m$-th
  cohomology 
  identifies with the isomorphism
  \begin{equation*}
    z_* \colon \Hom_{D(\Sh(X \times X))}(P,[m]P) \ra
    \Hom_{D(\Sh(X \times X))}(P,[m]I) 
  \end{equation*}
  since $P$ is in the enhancement $\mfPerf'_\Cechobjstar(X \times
  X)$ of $\mfPerf'(X \times X) \subset D(\Sh(X \times X))$ and
  $I$ is h-injective. We deduce that
  \begin{equation*}
    F:=\Hom_{C(\Sh(X \times X))}(P,-) \colon 
    [C^\hinj_\Qcoh(\Sh(X \times X))] \sira D(A \otimes A^\opp)
  \end{equation*}
  is an equivalence of triangulated categories.
  The category on the left identifies with
  $D_\Qcoh(\Sh(X \times X)),$
  and $F$ induces an equivalence 
  $\mfPerf'(X \times X) \sira \per(A \otimes A^\opp)$ on the
  subcategories of compact objects.
  Hence it is enough to
  show that $F$ maps (an h-injective lift of)
  $\Delta_*(\mathcal{O}_X)$ to (an object isomorphic to) $A.$

  The \v{C}ech resolution $\mathcal{O}_X \ra
  \mathcal{C}_*(\mathcal{O}_X)$ yields a resolution
  $\Delta_*(\mathcal{O}_X) \ra
  \Delta_*(\mathcal{C}_*(\mathcal{O}_X).$
  Let $\Delta_*(\mathcal{C}_*(\mathcal{O}_X)) \ra T$ be a
  quasi-isomorphism with $T \in C^\hinj_\Qcoh(\Sh(X \times X)).$
  So $T$ is an h-injective lift of $\Delta_*(\mathcal{O}_X).$
  Note that the induced map
  \begin{equation*}
    \Hom_{C(\Sh(X \times X))}(P,
    \Delta_*(\mathcal{C}_*(\mathcal{O}_X))) \ra
    F(T)=\Hom_{C(\Sh(X \times X))}(P,T)
  \end{equation*}
  is a quasi-isomorphism (of dg $A \otimes A^\opp$-modules):
  use the homotopy equivalence 
  \eqref{eq:P-hequi-boxtimes},
  Lemma~\ref{l:boxtimes-cech-object-to-diagonal-cech-K-iso-D},
  $D(\Qcoh(X \times X)) \sira D_\Qcoh(\Sh(X \times X)),$
  and the h-injectivity of $T.$
  
  The obvious adjunctions of dg functors provide isomorphisms
  of dg $A \otimes 
  A^\opp$-modules 
  \begin{align*}
    \Hom_{C(\Sh(X \times X))}(P,
    \Delta_*(\mathcal{C}_*(\mathcal{O}_X)))
    & \sira
    \Hom_{C(\Sh(X))}(\Delta^*(\mathcal{C}_*(E) \boxtimes
    \tildew{D}(\mathcal{C}_*(E))),
    \mathcal{C}_*(\mathcal{O}_X))\\ 
    & =
    \Hom_{C(\Sh(X))}(\mathcal{C}_*(E) \otimes
    \tildew{D}(\mathcal{C}_*(E)),
    \mathcal{C}_*(\mathcal{O}_X))\\
    & \sira
    \Hom_{C(\Sh(X))}(\mathcal{C}_*(E), \sheafHom(
    \tildew{D}(\mathcal{C}_*(E)),
    \mathcal{C}_*(\mathcal{O}_X)))\\
    & =
    \Hom_{C(\Sh(X))}(\mathcal{C}_*(E),
    \tildew{D}^2(\mathcal{C}_*(E))).
  \end{align*}
  Now use
  Lemma~\ref{l:map-to-double-dual-homotopy-equi}.
  The canonical morphism
  $\theta_{\mathcal{C}_*(E)} \colon  \mathcal{C}_*(E) 
  \ra
  \tildew{D}^2(\mathcal{C}_*(E))$ is a homotopy equivalence,
  so 
  \begin{equation*}
    (\theta_{\mathcal{C}_*(E)})_* \colon 
    \Hom_{C(\Sh(X))}(\mathcal{C}_*(E),
    \mathcal{C}_*(E))
    \ra
    \Hom_{C(\Sh(X))}(\mathcal{C}_*(E),
    \tildew{D}^2(\mathcal{C}_*(E)))
  \end{equation*}
  is a homotopy equivalence; moreover, it is a morphism of 
  dg $A \otimes A^\opp$-modules. The object on the left is the
  diagonal dg $A \otimes A^\opp$-module $A.$ 
  
  The above (quasi-)isomorphisms provide an isomorphism $F(T)
  \cong 
  A$ in $D(A \otimes A^\opp).$ This shows that $F$ maps
  $\Delta_*(\mathcal{O}_X)$ to $A.$
\end{proof}


\begin{thebibliography}{BvdB03}

\bibitem[AM69]{atiyah-macdonald}
M.~F. Atiyah and I.~G. Macdonald, \emph{Introduction to commutative algebra},
  Addison-Wesley Publishing Co., Reading, Mass.-London-Don Mills, Ont., 1969.

\bibitem[AVP92]{avramov-vigue-poirrier-hochschild-homology-criteria-for-smoothness}
Luchezar~L. Avramov and Micheline Vigu{\'e}-Poirrier, \emph{Hochschild homology
  criteria for smoothness}, Internat. Math. Res. Notices (1992), no.~1, 17--25.

\bibitem[BGI71]{berthelot-grothendieck-illusie-SGA-6}
P.~Berthelot, A.~Grothendieck, and L.~Illusie, \emph{Th\'eorie des
  intersections et th\'eor\`eme de {R}iemann-{R}och}, Lecture Notes in
  Mathematics, Vol. 225, Springer-Verlag, Berlin, 1971, S{\'e}minaire de
  G{\'e}om{\'e}trie Alg{\'e}brique du Bois-Marie 1966--1967 (SGA 6), Dirig{\'e}
  par P. {B}erthelot, A. {G}rothendieck et L. {I}llusie. Avec la collaboration
  de D. Ferrand, J. P. Jouanolou, O. Jussila, S. Kleiman, M. Raynaud et J. P.
  Serre.

\bibitem[BH93]{bruns-herzog-cm}
Winfried Bruns and J{\"u}rgen Herzog, \emph{Cohen-{M}acaulay rings}, Cambridge
  Studies in Advanced Mathematics, vol.~39, Cambridge University Press,
  Cambridge, 1993.

\bibitem[BLL04]{bondal-larsen-lunts-grothendieck-ring}
Alexey~I. Bondal, Michael Larsen, and Valery~A. Lunts, \emph{Grothendieck ring
  of pretriangulated categories}, Int. Math. Res. Not. (2004), no.~29,
  1461--1495.

\bibitem[BN93]{neeman-homotopy-limits}
Marcel B{\"o}kstedt and Amnon Neeman, \emph{Homotopy limits in triangulated
  categories}, Compositio Math. \textbf{86} (1993), no.~2, 209--234.

\bibitem[BvdB03]{bondal-vdbergh-generators}
A.~Bondal and M.~van~den Bergh, \emph{Generators and representability of
  functors in commutative and noncommutative geometry}, Mosc. Math. J.
  \textbf{3} (2003), no.~1, 1--36, 258.

\bibitem[Con00]{conrad-grothendieck-duality-bc}
Brian Conrad, \emph{Grothendieck duality and base change}, Lecture Notes in
  Mathematics, vol. 1750, Springer-Verlag, Berlin, 2000.

\bibitem[Con07]{conrad-compactification-nagata}
\bysame, \emph{Deligne's notes on {N}agata compactifications}, J. Ramanujan
  Math. Soc. \textbf{22} (2007), no.~3, 205--257.

\bibitem[CS15]{canonaco-stellari-uniqueness-of-dg-enhancements}
Alberto Canonaco and Paolo Stellari, \emph{Uniqueness of dg enhancements for
  the derived category of a {G}rothendieck category}, 2015,
  \href{http://arxiv.org/abs/1507.05509}{arxiv:1507.05509}.

\bibitem[Gro61]{EGAIII-i}
A.~Grothendieck, \emph{\'{E}l\'ements de g\'eom\'etrie alg\'ebrique. {III}.
  \'{E}tude cohomologique des faisceaux coh\'erents (premi\`{e}re partie)},
  Inst. Hautes \'Etudes Sci. Publ. Math. (1961), no.~11, 167.

\bibitem[Gro05]{SGA-2-new}
Alexander Grothendieck, \emph{Cohomologie locale des faisceaux coh\'erents et
  th\'eor\`emes de {L}efschetz locaux et globaux ({SGA} 2)}, Documents
  Math\'ematiques (Paris), 4, Soci\'et\'e Math\'ematique de France, Paris,
  2005, S{\'e}minaire de G{\'e}om{\'e}trie Alg{\'e}brique du Bois Marie, 1962,
  Revised reprint of the 1968 French original.

\bibitem[GW10]{goertz-wedhorn-AGI}
Ulrich G{\"o}rtz and Torsten Wedhorn, \emph{Algebraic geometry {I}}, Advanced
  Lectures in Mathematics, Vieweg + Teubner, Wiesbaden, 2010.

\bibitem[Har66]{hartshorne-residues-duality}
Robin Hartshorne, \emph{Residues and duality}, Lecture notes of a seminar on
  the work of A. Grothendieck, given at Harvard 1963/64. With an appendix by P.
  Deligne. Lecture Notes in Mathematics, No. 20, Springer-Verlag, Berlin, 1966.

\bibitem[Har77]{Hart}
\bysame, \emph{Algebraic geometry}, Springer-Verlag, New York, 1977, Graduate
  Texts in Mathematics, No. 52.

\bibitem[Kel94]{Keller-deriving-dg-cat}
Bernhard Keller, \emph{Deriving {DG} categories}, Ann. Sci. \'Ecole Norm. Sup.
  (4) \textbf{27} (1994), no.~1, 63--102.

\bibitem[Kem80]{kempf-elementary-proofs}
George~R. Kempf, \emph{Some elementary proofs of basic theorems in the
  cohomology of quasicoherent sheaves}, Rocky Mountain J. Math. \textbf{10}
  (1980), no.~3, 637--645.

\bibitem[Kra05]{krause-stable-derived}
Henning Krause, \emph{The stable derived category of a {N}oetherian scheme},
  Compos. Math. \textbf{141} (2005), no.~5, 1128--1162.

\bibitem[KS94]{KS}
Masaki Kashiwara and Pierre Schapira, \emph{Sheaves on manifolds}, Grundlehren
  der Mathematischen Wissenschaften, vol. 292, Springer-Verlag, Berlin, 1994.

\bibitem[KS06]{KS-cat-sh}
\bysame, \emph{Categories and sheaves}, Grundlehren der Mathematischen
  Wissenschaften, vol. 332, Springer-Verlag, Berlin, 2006.

\bibitem[LC07]{le-chen-karoubi-trcat-bdd-t-str}
Jue Le and Xiao-Wu Chen, \emph{Karoubianness of a triangulated category}, J.
  Algebra \textbf{310} (2007), no.~1, 452--457.

\bibitem[Liu02]{Liu}
Qing Liu, \emph{Algebraic geometry and arithmetic curves}, Oxford Graduate
  Texts in Mathematics, vol.~6, Oxford University Press, Oxford, 2002.

\bibitem[LO10]{lunts-orlov-enhancement}
Valery~A. Lunts and Dmitri~O. Orlov, \emph{Uniqueness of enhancement for
  triangulated categories}, J. Amer. Math. Soc. \textbf{23} (2010), no.~3,
  853--908.

\bibitem[LSa]{valery-olaf-matrix-factorizations-and-motivic-measures}
Valery~A. Lunts and Olaf~M. Schn\"urer, \emph{Matrix factorizations and motivic
  measures}, accepted by Journal of Noncommutative Geometry,
  \href{http://arxiv.org/abs/1310.7640}{arXiv:1310.7640}.

\bibitem[LSb]{valery-olaf-matfak-semi-orth-decomp}
\bysame, \emph{Matrix-factorizations and semi-orthogonal decompositions for
  blowing-ups}, accepted by Journal of Noncommutative Geometry,
  \href{http://arxiv.org/abs/1212.2670}{arXiv:1212.2670}.

\bibitem[LS14]{valery-olaf-smoothness-equivariant}
\bysame, \emph{Smoothness of equivariant derived categories}, Proc. Lond. Math.
  Soc. (3) \textbf{108} (2014), no.~5, 1226--1276.

\bibitem[Lun10]{lunts-categorical-resolution}
Valery~A. Lunts, \emph{Categorical resolution of singularities}, J. Algebra
  \textbf{323} (2010), no.~10, 2977--3003.

\bibitem[L{\"u}t93]{luetkeboehmert-compactification-nagata}
W.~L{\"u}tkebohmert, \emph{On compactification of schemes}, Manuscripta Math.
  \textbf{80} (1993), no.~1, 95--111.

\bibitem[Mur06]{murfet-der-cat-qcoh-sheaves}
Daniel Murfet, \emph{{D}erived categories of quasi-coherent sheaves}, Note
  (2006), 1--62,
  \href{http://therisingsea.org/notes/DerivedCategoriesOfQuasicoherentSheaves.pdf}{therisingsea.org/notes/DerivedCategoriesOfQuasicoherentSheaves.pdf}.

\bibitem[Orl14]{orlov-smooth-proper-glueing-arxiv}
Dmitri Orlov, \emph{Smooth and proper noncommutative schemes and gluing of dg
  categories}, Preprint (2014),
  \href{http://arxiv.org/abs/1402.7364}{arXiv:1402.7364 [math.AG]}.

\bibitem[Rei95]{reid-undergrad-comm-alg}
Miles Reid, \emph{Undergraduate commutative algebra}, London Mathematical
  Society Student Texts, vol.~29, Cambridge University Press, Cambridge, 1995.

\bibitem[Sch]{olaf-OWR-enhancements}
Olaf~M. Schn\"urer, \emph{Some enhancements of derived categories of coherent
  sheaves and applications}, to appear in Oberwolfach Reports, abstract from
  the talk given in May 2014 at the workshop on Interactions between Algebraic
  Geometry and Noncommutative Algebra.

\bibitem[Ser00]{serre-local-algebra}
Jean-Pierre Serre, \emph{Local algebra}, Springer Monographs in Mathematics,
  Springer-Verlag, Berlin, 2000, Translated from the French by CheeWhye Chin
  and revised by the author.

\bibitem[Shk07]{shklyarov-serre-duality-cpt-smooth-arXiv}
Dmytro Shklyarov, \emph{On {S}erre duality for compact homologically smooth
  {D}{G} algebras}, Preprint (2007),
  \href{http://arxiv.org/abs/math/0702590v1}{arXiv:math/0702590v1 [math.RA]}.

\bibitem[Spa88]{spaltenstein}
N.~Spaltenstein, \emph{Resolutions of unbounded complexes}, Compositio Math.
  \textbf{65} (1988), no.~2, 121--154.

\bibitem[SS]{wolfgang-olaf-locallyproper}
Olaf~M. Schn\"urer and Wolfgang Soergel, \emph{Proper base change for separated
  locally proper maps}, to appear in Rendiconti del Seminario Matematico della
  Università di Padova,
  \href{http://arxiv.org/abs/1404.7630}{arXiv:1404.7630}.

\bibitem[{Sta}14]{stacks-project}
The {Stacks Project Authors}, \emph{Stacks project},
  \url{http://stacks.math.columbia.edu}, 2014.

\bibitem[To{\"e}07]{toen-homotopy-of-dg-cats-morita}
Bertrand To{\"e}n, \emph{The homotopy theory of {$dg$}-categories and derived
  {M}orita theory}, Invent. Math. \textbf{167} (2007), no.~3, 615--667.

\bibitem[Tot04]{totaro-resolution}
Burt Totaro, \emph{The resolution property for schemes and stacks}, J. Reine
  Angew. Math. \textbf{577} (2004), 1--22.

\bibitem[TT90]{thomason-trobaugh-higher-K-theory}
R.~W. Thomason and Thomas Trobaugh, \emph{Higher algebraic {$K$}-theory of
  schemes and of derived categories}, The {G}rothendieck {F}estschrift, {V}ol.\
  {III}, Progr. Math., vol.~88, Birkh\"auser Boston, Boston, MA, 1990,
  pp.~247--435.

\end{thebibliography}

\def\cprime{$'$} \def\cprime{$'$} \def\cprime{$'$} \def\cprime{$'$}
  \def\Dbar{\leavevmode\lower.6ex\hbox to 0pt{\hskip-.23ex \accent"16\hss}D}
  \def\cprime{$'$} \def\cprime{$'$}
\providecommand{\bysame}{\leavevmode\hbox to3em{\hrulefill}\thinspace}
\providecommand{\MR}{\relax\ifhmode\unskip\space\fi MR }
\providecommand{\MRhref}[2]{%
  \href{http://www.ams.org/mathscinet-getitem?mr=#1}{#2}
}
\providecommand{\href}[2]{#2}

\end{document}